\documentclass[a4paper,oneside,11pt]{article}

\usepackage{longtable,geometry}
\usepackage{amssymb,amsmath,amsthm,amsfonts,mathrsfs}

\usepackage{indentfirst}
\usepackage{enumitem}
\usepackage[utf8]{inputenc}

\numberwithin{equation}{section}

\usepackage{verbatim}
\usepackage{hyperref}
\usepackage{graphicx}
\usepackage{bbm}
\DeclareUnicodeCharacter{2265}{$\ge$}
\DeclareUnicodeCharacter{2264}{$\le$}

\theoremstyle{definition}
\newtheorem{theorem}{Theorem}[section]
\usepackage[dvipsnames]{xcolor}
\newtheorem{proposition}[theorem]{Proposition}

\newtheorem{lemma}[theorem]{Lemma}

\theoremstyle{definition}
\newtheorem{definition}[theorem]{Definition}
\newtheorem{remarks}[theorem]{Remark}

\usepackage{pifont}
\usepackage{mathtools} 
\usepackage{extarrows} 
\mathtoolsset{showonlyrefs}

\newcommand{\Cov}{\mathrm{Cov}}
\newcommand{\dd}{\mathrm d}

\newcommand{\scrC}{\mathscr{C}}

\newcommand{\calA}{\mathcal{A}}

\newcommand{\calC}{\mathcal{C}}
\newcommand{\calD}{\mathcal{D}}

\newcommand{\calN}{\mathcal{N}}

\newcommand{\calR}{\mathcal{R}}
\newcommand{\calS}{\mathcal{S}}
\newcommand{\calT}{\mathcal{T}}

\newcommand{\bbC}{\mathbb{C}}

\newcommand{\bbE}{\mathbb{E}}

\newcommand{\bbP}{\mathbb{P}}

\newcommand{\bbR}{\mathbb{R}}

\newcommand{\bbT}{\mathbb{T}}

\newcommand{\bbZ}{\mathbb{Z}}

\newcommand{\eps}{\varepsilon}

\newcommand{\ind}{\mathbbm 1}

\setlist[itemize]{itemsep=0pt, topsep=2pt}
\setlist[enumerate]{itemsep=1pt, topsep=4pt}

\usepackage[backend=biber,minbibnames=5,maxbibnames=5,maxcitenames=5]{biblatex}
\begin{document}

\title{The near-critical random bond FK-percolation model}

\author{Emile Avérous\thanks{\textsc{Université de Genève. \emph{E-mail:}  \texttt{emile.averous@unige.ch, remy.mahfouf@unige.ch}}}\ , Rémy Mahfouf\footnotemark[1]}

\maketitle

\abstract{We study FK-percolation where the edge parameters are chosen as independent random variables in the near-critical regime. We show that if these parameters satisfy a natural centering condition around the critical point, then the quenched model typically exhibits critical behaviour at scales much larger than the deterministic characteristic length. More precisely, in a box of size $N$, if the homogeneous model with deterministic edge parameter $p$ looks critical in the regime $|p-p_c|\le \textrm W$, then the quenched model with random edge parameters $\mathbf p$ that typically satisfy $|\mathbf p-p_c|\le \textrm W^{1/3}$ looks critical, assuming some conjectured inequality on critical exponents, and up to logarithmic corrections. We also treat the special case of Bernoulli percolation, where we show that if one first samples non-degenerate independent random edge parameters centered around $\frac12$, and then a percolation configuration on these edges, the quenched model almost surely looks critical at large scales.}

\section{Introduction}

\subsection{General context}

The two-dimensional Fortuin--Kasteleyn (FK) random cluster model, introduced in \cite{fortuin1972random}, provides a unifying framework for studying several central models in statistical mechanics, including Bernoulli percolation ($q = 1$), the Ising model ($q = 2$), and the Potts model for general $q \geq 1$. It was developed as a representation of the Potts model through random subgraphs of a fixed domain, weighted by a bond parameter $p \in (0;1)$ and the number of connected components raised to a power governed by another parameter $q > 0$. In the planar case with nearest-neighbour interactions, the FK model is conjectured to exhibit \emph{conformal invariance} in the scaling limit at its critical point when $q \leq 4$, in agreement with predictions from conformal field theory (CFT) \cite{BPZ84a,BPZ84b} and the Coulomb gas formalism. In the infinite-volume limit, the critical point on the square lattice, given by $p_c(q) = \sqrt{q}\cdot(1+\sqrt{q})^{-1}$, has been rigorously established for over a decade \cite{beffara2012self}, confirming the longstanding conjecture that it coincides exactly with the self-dual point. On the square lattice, the only value of $q$ for which precise and complete statements of conformal invariance are known is the so-called FK-Ising model (corresponding to $q = 2$). In this case, Smirnov’s breakthrough result \cite{Smi07} proved the conformal invariance of FK interfaces by showing their convergence to $\mathrm{SLE}(16/3)$, using the so-called Ising fermionic observables. This approach shares many features with another seminal result by the same author, namely the proof of the Cardy formula for critical percolation on the triangular lattice \cite{Smi01}. Beyond crossing probabilities, it was later shown in this case that interfaces converge to $\mathrm{SLE}(6)$ \cite{WerSmi}, and that the entire loop ensemble converges to $\mathrm{CLE}(6)$ \cite{camia2006sle}. To date, it appears that generalisations of fermionic observables techniques are insufficient to directly address conformal invariance for all parameters $1 \leq q \leq 4$, where the scaling limits of FK interfaces are conjectured to be described by the nested $\mathrm{CLE}(\kappa)$ loop ensemble, where $\kappa \in [4,6]$ is liked to $q$ via the relation
\[
\sqrt{q} = -2 \cos\left( \frac{4\pi}{\kappa} \right).
\]

An extensive body of research on the critical FK model over the past decade \cite{duminil2021planar,duminil2016new, duminil2021macroscopic,glazman2025delocalisation,duminil2019sharp, duminil2017continuity} has yielded numerous important results, significantly advancing our understanding of the model. In this context, two landmark developments are the recent proof of rotational invariance in the full-plane setting \cite{duminil2020rotational}, and the convergence of the associated six-vertex model to the Gaussian Free Field (GFF) \cite{GFFConv}, both of which represent major steps toward establishing conformal invariance.

It is widely believed that all reasonable statistical mechanics models exhibit a rather specific type of behaviour near their critical point, notably through the existence of the so-called \emph{critical exponents}, which link the large-scale properties of the critical model to the behaviour of near-critical models. Among other things, the connectivity properties, cluster density, susceptibility, and the regularity of the free energy are expected to follow algebraic laws near the critical point. Moreover, these critical exponents are predicted to be related to one another through elegant relations originating from physics \cite{essam1963pade,fisher1964correlation}, as recalled in \cite[Introduction]{FK_scaling_relations}. In \cite{FK_scaling_relations}, Duminil-Copin and Manolescu provided a rigorous proof that, assuming the existence of these exponents, they indeed satisfy the relations conjectured by physicists. Their results extend Kesten’s scaling relations \cite{kesten1987scaling} (see also \cite{nolin2008near}), originally proved for Bernoulli percolation (where independence of edges plays a crucial role), to the full range of parameters $1 \leq q \leq 4$. Understanding the precise behaviour of near-critical model exactly at its scaling window remains an active area of research for general FK models. Let us still mention that significantly more is known in the case of the near-critical FK-Ising model \cite{beffara2012smirnov,duminil2014near,park2018massive,park2019ising,park2022convergence,chelkak2023universality,wan2022statistical}, where the extension of Ising fermionic observable techniques to the massive regime has led to substantial progress, in line with analogous results in the dimer context \cite{chhita2012height,rey2024doob,berestycki2022near}.

Once a refined understanding of critical models has been achieved, a natural question arises -that also connects to physical experiments- concerning the behaviour of critical models in \emph{random environments}, particularly those that might average out to the deterministic critical setting. One may ask what happens if the bond parameters of the random cluster model are first randomly sampled, and one then performs statistical mechanics in this random environment. This is in the spirit of the famous (and still largely open) question concerning the critical behaviour of the random bond Ising model \cite{cho1997criticality,shalaev1994critical,merz2002two}. It is believed that for general FK models, such questions fall within the broader meta-principle of \emph{universality} in statistical mechanics, which roughly states that the macroscopic behaviour at criticality should not be influenced by microscopic details, provided the parameters are tuned appropriately. In this setting, the critical parameter for the model in a random environment should almost surely (with respect to the randomness of the environment) exist and be well-defined, although it is generally different from the deterministic critical one. The study of statistical mechanics models in random environments remains largely unexplored, with the notable exception of random walks, for which homogenisation techniques yield rather precise results (see, e.g.\ \cite{egloffe2015random,gloria2015quantification, armstrong2019quantitative}). For the random cluster model in a random environment, one notable result is the continuity of the quenched phase transition \cite{aizenman1990rounding,aizenman1989rounding}.

 In most of the present paper, we focus on the regime $1 < q \leq 4$ and study \emph{weakly random environments}, where the deviation from the homogeneous critical parameter in a box of size $N$ is modelled by independent random variables on each edge, centred at the critical point, with variances tending to zero as $N \to \infty$. Our main result is that, under mild assumptions, the weak randomness of the environment around the critical point averages out effectively, keeping the system in its critical phase. In particular, we show that, with high probability over the environment, one can allow random variables whose standard deviation from the critical point is \emph{cubicly larger than the intrinsic deterministic critical window of the model}, while still remaining within the critical phase \emph{at all scales}. In all of these cases, had one replaced the deviation from criticality by its absolute value (i.e., enforcing a fixed deterministic offset), the system would have been significantly off-critical. We apply an approach similar to the one in \cite{Mah25}, moving continuously the random environment from the deterministic critical point to the desired random one.

We also give special attention to the case of Bernoulli percolation, where we prove that it is possible to take bond parameters whose standard deviation from the self-dual parameter is, in the exact Gaussian case, $O(\log(N)^{-1/2})$ in a box of size $N$, still ensuring criticality \emph{at every scale}. Moreover, by leveraging noise sensitivity results \cite{benjamini1999noise,garban2013pivotal}, we observe that it is even possible to allow \emph{macroscopic deviations} (i.e., deviations \emph{don't go to } $0$ as $N$ goes to $\infty $), centred at the critical point, while retaining critical behaviour \emph{at large scales}. For instance, we establish an asymptotic version of Cardy’s formula that holds with high probability over the environment. As an concrete example, one can sort for each vertex $v$ of the triangular lattice $\mathcal{T}$, a random bond parameter $\mathbf{p}_{v}\in \{\frac{1}{4},\frac{3}{4}\}$, whose average is $\mathbf{E}[p_v]=\frac{1}{2}$, creating a random environment $\frak{p}=(\mathbf{p}_v)_{v\in \mathcal{T}}$ for some probability measure $\mathbf{P}$. Consider the percolation $\phi_{\frak{p}}$ whose weights are given by $\frak{p}$. Then, with high $\mathbf{P}$-probability, the percolation measure  $\phi_{\frak{p}}$ satisfies the Cardy formula at large scale. This means that, the \emph{quenched} percolation satisfies (with high $\mathbf{P}$-probability), the Cardy formula at large scale, exactly as its annealed version $\phi_{\mathbf{E}[\frak{p}]}$, which corresponds to the critical deterministic Bernoulli percolation on the triangular lattice. In particular, it appears that the for Bernoulli percolation, the quenched disorder is irrelevant.

\subsection{Setting and graph notations}

Let $G=(V,E)$ be a finite subgraph of $\mathbb{Z}^2$ and define its boundary as $\partial G:= \{v\in V(G): v \text{ is incident to at least one edge outside of } G \} $. We call a boundary condition $\xi$ a partition of the set $\partial G$. Among all possible boundary conditions, we highlight two special ones, the so called \emph{wired} boundary conditions $\xi =1 $ corresponding to the partition $\{\partial G\}$ where all the vertices of $\partial G$ are wired together as a single vertex, and the \emph{free} boundary conditions $\xi =0$ coresponding to the partition made of singletons, where no vertices of $\partial G$ are wired together. We are now going to define a probability measure on random subgraphs of $G$. 
\begin{definition}[Planar FK-percolation]
	Given $q\geq 1$ and $p\in (0;1)$, one can define the FK-percolation probability measure $\phi_{G,p,q}^{\xi} $ on configurations $\omega \in \{0;1\}^{E}$ by setting
	\begin{equation}
		\phi_{G,p,q}^{\xi}[\omega]:=\frac{1}{Z^{\xi}_{G,p,q}}p^{|\omega|}(1-p)^{|E\backslash \omega|}q^{k(\omega^\xi)}, 
	\end{equation}
where $|\omega|$ denotes the number of vertices in the configuration $\omega $, $k(\omega^\xi)$ the number of connected components of $\omega$ when identifying boundary vertices in the same part of $\xi$ and $Z^{\xi}_{G,p,q}$ is the partition function of the model, which is chosen so that $\phi_{G,p,q}^{\xi}[\cdot]$ is a probability measure.

 In the case where the bonds weights are not uniform and given by $\underline p=(p_{e})_{e\in E} $, one can generalise the previous definition setting 
 \begin{equation}
		\phi_{G,\underline{p},q}^{\xi}[\omega]:=\frac{1}{Z_{G,\underline{p},q}}  q^{k(\omega^{\xi})} \prod\limits_{e\in \omega} p_{e} \prod\limits_{e\not \in \omega}(1-p_{e}). 
	\end{equation}
\end{definition}
An edge $e$ is called open in $\omega$ if $\omega(e)=1$, and closed otherwise. For two vertices $v,v'\in G\cup \partial G$, we define the event $v \overset{\omega}{\longleftrightarrow} v'$, or simply $v \longleftrightarrow v' $ to be the event that the two vertices $v,v'$ belong to the same connected component in $\omega$. In the same way, for two (disjoint) subsets $A,B$ of $G$, define $A\overset{\omega}{\longleftrightarrow}B$ or simply $A\longleftrightarrow B$ to be the event that there exist $v\in A,v'\in B$ such that $v\longleftrightarrow v'$. This model of dependent percolation is a generalisation of the usual independent bond percolation percolation on $\mathbb{Z}^2$ (corresponding to $q=1$), interpolating also the so-called FK-Ising model (corresponding to $q=2$). 
The dual of the lattice $\mathbb{Z}^2$ is $(\mathbb{Z}^2)^\star:=\mathbb{Z}^2+(\frac{1}{2};\frac{1}{2})$, corresponding to faces of $\mathbb{Z}^2$. Each edge $e\in E$ can be mapped to the associated dual edge $e^\star $ that links the two faces separated by $e$. One can now define the dual configuration $\omega^\star$ as a graph on $G^\star$ by setting
\begin{equation}
	\omega^{\star}(e^\star):= 1 -\omega(e).
\end{equation}
In the case where the random cluster on $G$ is chosen with boundary conditions $\xi $ equal to $0$ or $1$, the duality can even be extended as a probability measures on clusters $\omega^\star\in G^\star\cup \partial G^{\star}$, with boundary conditions $\xi^\star =1-\xi $.

To circumvent some additional technicalities appearing when dealing with the uniqueness of the infinite volume measure, a fair share of our work is done on the torus, which requires changing slightly the above definition to take into account the non-planar nature of the torus. Let us state once and for all that all the results presented in the present paper can be easily transferred to the planar case. 

For $N\geq 1$, set $\mathbb{T}_N:=(\mathbb{Z}/N\mathbb{Z})^2$ the standard torus of width $N$, whose edges are denoted by $E_N$. In that context, one can make a slightly different definition of the random cluster model. Note that the torus does not have a boundary, so there is no need to specify boundary conditions in the definition. In the following, a cluster is called \emph{non-contractible} if it contains a path with non-trivial winding around the torus, and \emph{contractible} otherwise.
\begin{definition}[Torus FK-percolation]
	Given $q\geq 1$ and $\underline p=(p_e)_{e\in \mathbb{T}_N} $ one can define the FK-percolation probability measure $\phi_{\mathbb{T}_N,p,q}$ on $\{0;1\}^{E_N}$ by setting
 \begin{equation}
		\phi_{\bbT_N,\underline{p},q}[\omega]:=\frac{\sqrt{q}^{s(\omega)}}{Z_{\bbT_N,\underline{p},q}}  q^{k(\omega)} \prod\limits_{e\in \omega} p_e \prod\limits_{e\not \in \omega}(1-p_e),
	\end{equation}
where $|\omega|$, $k(\omega)$ are defined as above, $Z_{\mathbb{T}_N,p,q}$ is the partition function of the model and 
\begin{itemize}
	\item $s(\omega)=0 $ is all clusters of $\omega $ are contractible
	\item $s(\omega)=1 $ if there exists at least one non-contractible cluster in $\omega $ and one non-contractible cluster in $\omega^{\star} $.
	\item $s(\omega)=2 $ if all clusters of $\omega^{\star} $ are contractible.
\end{itemize}
\end{definition}
The modification in the definition compared to the planar case is encoded in the term $\sqrt{q}^{s(\omega)}$ which is added to make the model self-dual in the homogeneous bond case, using the fact that for any configuration $\omega$, one has $s(\omega)=2-s(\omega^*)$ (see \cite{beffara2012self} or \cite[Exercise 28]{pims_lecture_duminil}). One could have in principle used a definition without the $\sqrt{q}^{s(\omega)}$ which would have led to the same large scale properties of the critical model, since this factor is always between $1$ and $q$.

In the last two decades, several groundbreaking results provided a very refined understanding of the sharp phase transition of the FK-random cluster models on $\mathbb{Z}^2$, with a specific focus on the large scale behaviour at critical model. The identification of the critical point was made by Beffara and Duminil-Copin in \cite{beffara2012self} using the self-duality of the model on $\mathbb{Z}^2$ and differential equations on the probability of crossing large topological rectangles. More precisely, the critical point $p_c(q)$ is given by the self-dual parameter $p_{sd}(q)$ that makes the model self-dual (i.e.\ $\omega \overset{(d)}{=} \omega^\star$), which is given explicitly by
\begin{equation}
	p_{sd}(q)=p_c(q):=\frac{\sqrt{q}}{1+\sqrt{q}}.
\end{equation}
In the range of parameters $1\leq q \leq 4$, some sharp dichotomy in the connectivity properties of the full-plane model appears. The restriction of this is two-fold. In a nutshell, if $p<p_c(q)$ the model is called sub-critical and the probability that two vertices are connected decays to $0$ exponentially fast in the distance separating them. Conversely if $p>p_c(q)$, the model is called super-critical and the connection probability between vertices converges, exponentially in the distance separating the vertices, to some positive constant. The most interesting and delicate question to treat is that of the behaviour of the model at its critical point $p_c(q)$. In that case, the probability that two points are connected decays polynomially fast in their distance and the model displays some intermediate behaviour. Let us also mention (see \cite{duminil2021discontinuity}) that when $q>4$, the phase transition of the model is discontinuous and correlations at criticality for the full plane free measure decay exponentially fast. On the other hand, when $q<1$, the models is negatively associated, meaning that increasing events disfavour each other.

The scaling limit of the model is conjectured to be conformally invariant, which was only proven by Smirnov in the special case of site percolation on the triangular lattice \cite{Smi01} and the Ising model \cite{Smi07}. Let us highlight the recent result that all FK-percolation models with $1\leq q\leq 4$ are rotationally invariant at large scales \cite{duminil2020rotational}, providing some landmark progress towards the rigorous identification of the full-plane scaling limit, which is strengthened by the recent identification of the scaling limit of the associated $6$ vertex model \cite{GFFConv} in the full-plane. 

To date, one of the best understood features of the critical self-dual models is some coexistence phenomenon of large macroscopic primal and dual clusters that surround each other. One way to quantify this phenomenology is the so-called Russo-Seymour-Welsh (RSW) \cite{kesten1982russo} strong box crossing property which reads as follows. Let $\Lambda_N:=[-N;N]^2$ the box centred at the origin and define $\scrC(\Lambda_N)$ to be the event that there exists an open circuit linking the left to the right vertical boundaries of $ \Lambda_N$. Then there exists an absolute constant $c(q)>0 $ (which, crucially, does not depend on $N$) \cite{beffara2012self,duminil2017continuity} such that
\begin{equation}\label{eq:crossing-annulus}
	c(q)\leq \phi_{\Lambda_{2N},p_c(q),q}^{0}\Big[ \textrm { there exist a vertical crossing in }  \Lambda_N\Big]\leq 1-c(q),
\end{equation}
where FK-measure is chosen with the free boundary conditions. One can fix some universal $\delta >0$ small enough (depending on $q$), whose exact value is not relevant for the theory and can be explicitly recovered by \cite[Theorem 2.1]{FK_scaling_relations}. Defining the event $\scrC(\Lambda_N)$ to be the existence of a vertical crossing between the boundaries of $\Lambda_N$, then for the unique full-plane Gibbs measure $\phi_p $ on $\mathbb{Z}^2$, one can define the \emph{characteristic length} by setting 
\begin{equation}\label{eq:characteristic-length}
	\textrm L(p)=\textrm L(p,q):=\inf \Big\{ R\geq 1, \phi_p [\scrC(\Lambda_N)] \not\in [\delta;1-\delta] \Big\} \in [1;\infty].
\end{equation}
The duality and the sharpness of the phase transition in \cite{beffara2012self} ensure that $\textrm L(p)<\infty $ when $p\ne p_c(q)$ while $\textrm L(p_c(q))=\infty$ by \eqref{eq:crossing-annulus}. Informally speaking, the characteristic length is the maximal size of a box for which the FK-model of parameters $(p,q)$ still behaves like the critical one $(p_c(q),q)$, at least in terms of crossing events. One can define similarly the \emph{critical window} on the box $\Lambda_N $ by setting
\begin{equation}
	\textrm{W}(N)=\textrm{W}(N,q):= \inf\limits_{p\in [0,1]} \Big\{ |p-p_c(q)|, \, \exists \textrm{ a square } S \subset \Lambda_N, \,  \phi_p [\scrC(S)] \not\in [\delta;1-\delta] \Big\}. 
\end{equation}
The critical window and characteristic length can be interpreted as being two parametrisations of the set of pairs $(R,p)$ such that, under $\phi_p$, one has some RSW property at scale $R$. From this, one gets that these functions are essentially inverses of each other, and it can be shown that they satisfy the identities $\textrm L(p_c(q)\pm \textrm W(N))\asymp N$ and $\textrm W(\textrm L(p))\asymp |p-p_c(q)|$. In the case of Bernoulli bond percolation, one of the scaling relations proved by Kesten \cite{kesten1987scaling} relates the critical window $\textrm W(N)$ to some geometric quantity called the $4$-\emph{arm event probability} at criticality $\pi_{4,p_c}(N)$, through the identity $\textrm W(N)\asymp N^{-2}\pi_{4,p_c}(N)^{-1}$. For general FK models (with $1<q\leq 4 $), one can define the so-called \emph{mixing rates} $\Delta_{p_c}(N)$
\begin{equation}\label{eq:def-mixing-rate}
	\Delta_{p_c}(N):=\phi^{1}_{p_c(q),q}[0\longleftrightarrow \partial \Lambda_N]-\phi^{0}_{p_c(q),q}[0\longleftrightarrow \partial \Lambda_N].
\end{equation}
Informally speaking, the mixing rate measures precisely the influence of boundary conditions for connection probabilities for the model at criticality. One of the core innovations of \cite{FK_scaling_relations} is that for $q\in (1,4]$, one can replace the stability of the $4$-arm event probabilities below the characteristic length by the stability of the mixing rates to derive the scaling relations.  Note that when $q=1$, the bond independence of the edges enforces that $\Delta_{p_c}(N)=0$ for any $N\geq 1$, making this quantity useless. One of the main results of \cite{FK_scaling_relations} is the identity 
\begin{equation}\label{eq:scaling-Delta-W}
	\textrm W(N)\asymp N^{-2}\Delta_{p_c}(N)^{-1},
\end{equation}
which allows to rewrite the scaling window (inherently tied to the off-critical behaviour of the model) to the stochastic geometry at criticality.
\bigskip

\textit{Notational convention} In the entire paper, for $f,g$ two functions that may depend on the bond parameter $p$, we denote respectively $f \asymp g $, $f\lesssim g $ and $f\gtrsim g $ the fact that there exist some universal contants $C$ only depending on $1\leq q \leq 4$, and the constants appearing in Definition \ref{def:proprietes_pe}, such that one has $C^{-1}f \leq g  \leq Cf $ (respectively $C^{-1}f\leq g $, $g\leq Cf$). When $f\lesssim g$, we also use sometimes the standard notation $f=O(g)$. Moreover, all the constants will implicitly be allowed to depend on the $\delta$ parameter of the $\textrm{RSW}(\delta,N)$ strong box crossing property of Definition \ref{def:RSW-Torus}.
\subsection{Main results}

The main goal of the present paper is to define a similar notion of critical window in the case of a random environment, meaning here that
one first sorts the edge weights of the random cluster model under some probability measure $\mathbf{P}_N$ and then studies the random cluster model attached to these weights. We will stick here to the case where the random weights are centred around the self-dual value in the deterministic case, with an addition natural condition on its variance detailed below. In a nutshell, the main output of the present paper is that randomness when sorting the environment under $\mathbf{P}_N$ averages very well and preserves the critical phase of the model way beyond the deterministic parameters, at least with a very high $\mathbf{P}_N$-probability. Let us fix some parameter $ 1 < q \leq 4$, excluding the case of percolation which will be treated separately, and add a few notations. First define, for each $q$, the value
\begin{equation}
\label{eq:mu_q_def}
    \mu(q):= \frac{q-1}{2\sqrt q}.
\end{equation}
\begin{definition}[Properties $(\star)_N^{\text A}$ and $(\star)_N^{\text B}$]
\label{def:proprietes_pe}
Consider a collection of random variables $\frak{p}=(\mathbf{p}_e)_{e\in E_N}$ indexed by the edges of $\mathbb{T}_N$, which will be called a \emph{random environment}. For each $e\in E_N$, denote by ${\sigma_e:=\bbE[(\mathbf{p}_e-p_c(q))^2]^{\frac12}}$ the standard deviation of $\mathbf p_e$. Consider the following properties for the random environment $\frak p$: 
\begin{enumerate}
\item (Independence) The variables $(\mathbf{p}_e)_{\in E_N}$ are mutually independent.
\item \begin{enumerate}[label=\Alph*)]
\item (Naive centering) Their first moment is centred at $p_c(q)$ i.e.\
 \begin{equation*}\mathbb{E}_{\mathbf{P}_N}[\mathbf{p}_e] =p_c(q).
 \end{equation*}

\item (Natural centering) Their first and second moments are related by 
 \begin{equation*}\mathbb{E}_{\mathbf{P}_N}[\mathbf{p}_e]=p_c(q)-\mu(q)\sigma_e^2+O(\sigma_e^3).
    \end{equation*}
    where $\mu(q)$ is defined above and the constant in $O$ is uniform in $N$.
\end{enumerate}
\item (Exponential tails) For each $e\in E_N$, the tails of $\mathbf p_e$ satisfy
    \begin{equation*}
        \forall s>0,\quad \bbP_N\Big[|\mathbf p_e-p_c(q)|>s\Big]\lesssim \exp\left(-c\cdot \left(\frac s{\sigma_e}\right)\right)
    \end{equation*}
    for some fixed $c>0$ independent of $N$. 
    \end{enumerate}
We say that the variables $\frak p=(\mathbf p_e)_e$ satisfy $(\star)_N^{\text A}$ (respectively $(\star)_N^{\text B}$) if they satisfy properties $1,2\text A$ and $3$ (respectively $1,2\text B$ and $3$). Note that these definitions implicitly depend on the constants $O,\lesssim$ and the constant $c$, which we do not mention for the sake of readability. In particular, all the constants appearing in the following theorems might depend on $O,\lesssim,c$.
\end{definition}

The first condition is set to run some central limit theorem regarding the random environment $\mathbf{P}_N$, while the second condition on moments comes from some natural approximate self-duality for the first two moments of the random variables $\mathbf p_e$, as we want \emph{both} crossing probabilities of the primal and dual models not to degenerate. The third condition is only technical and can be relaxed. In principle, all the results in the present paper also hold if the $\mathbf p_e$ have stretched exponential tails, with modified estimates. We also do not assume that the variables $\mathbf p_e$ are $[0,1]$-valued, but their tails ensure that, with high probability, all of the $\mathbf p_e$ belong to $[0,1]$ as long as the $\sigma_e$ go to $0$ fast enough in $N$ (see e.g.\ the discussion at the beginning of Section \ref{sub:centred-gaussian}). Except for the special case of Bernoulli percolation, the spirit of the present paper is to send $\Sigma_{\frak{p}}:= \max_{e\in E_N} \sigma_e \to 0$ as the size of $\bbT_N $ goes to $\infty$, but at a speed which is much slower than the deterministic critical window $\textrm W(N)$. 
In particular, randomness on the environment produces with a high $\mathbf{P}_N$ probability some FK model which is a mixture of locally subcritical and supercritical models which average to a critical model on a much larger scale than the deterministic one. To make the setup more concrete, let us present two families of i.i.d.\ random variables that satisfy respectively $(\star)_N^{\text A}$ and $(\star)_N^{\text B}$ :
\begin{enumerate}
	\item Given a family $(X^{(e)})_{e\in E_N} $ of i.i.d.\ centred $\pm 1$ random variables and some non-negative $\sigma_N$, set 
\begin{flalign*}
	(\textrm{A}) \quad & \mathbf p_e=p_c(q)+ \sigma_N \cdot X^{(e)} = p_c(q) \pm \sigma_N \\
	(\textrm{B}) \quad & \mathbf p_e=p_c(q)+ \sigma_N \cdot X^{(e)} - \mu(q)\cdot \sigma_N^2 = p_c(q) \pm \sigma_N - \mu(q)\cdot \sigma_N^2.
	\end{flalign*}
	\item  Given a family $(\mathcal{N}^{(e)})_{e\in E_N} $ of i.i.d.\ standard centred gaussian variables and some non-negative $\sigma_N$, set 
\begin{flalign*}
	(\textrm{A}) \quad & \mathbf p_e=p_c(q)+ \sigma_N \cdot \mathcal{N}^{(e)},  \\
	(\textrm{B}) \quad & \mathbf p_e=p_c(q)+ \sigma_N \cdot \mathcal{N}^{(e)} - \mu(q)\cdot \sigma_N^2.
	\end{flalign*}
\end{enumerate}
When $\sigma_N\to 0 $ as $N\to \infty$, both families of variables satisfy respectively $(\star)_N^{\text A}$ and $(\star)_N^{\text B}$, while the $\sigma_N$ scaling ensures that $\Sigma_{\frak{p}}\to 0$. 
The Gaussian setup of the second item is of crucial interest and will be treated separately, as it allows to deduce the general case via some form of Skorokhod embedding theorem. 

In the random non-homogeneous setup, the FK percolation measure with bonds inherited from $\mathbf{P}_N$ lacks translation invariance and self-duality. Therefore, to state some criticality conditions, one needs to check that crossing events (both primal and dual) are preserved in all of a given torus. This is the classical Russo-Seymour-Welsh (RSW) theory, which leads us to the following definition.

\begin{definition}\label{def:RSW-Torus}
    Let $\delta>0$ be some small enough positive constant. A \emph{$2$ by $1$ rectangle} in $\bbT_N$ is a rectangle $\calR$ of dimensions $R \times R/2$ or $R/2 \times R$  in $\bbT_N$, with $1\le R\le N/4$. For such an $\calR$, let $\calR'$ be the twice bigger rectangle with the same center as $\calR$. Denote $\scrC_h(\calR)$ (respectively $\scrC^{\star}_h(\calR)$) the event that $\calR$ is crossed in the horizontal direction by a primal (respectively dual) path, and $\scrC_v(\calR),\scrC^{\star}_v(\calR)$ the same events but for a vertical crossing. For an environment $\underline p=(p_e)_{e\in E_N}$, we say that the measure $\phi_{\bbT_N,\underline p,q}$ satisfies the strong RSW property on the torus $\bbT_N $ for the parameter $\delta$ if for \emph{any} $2$ by $1$ rectangle $\calR\subset \bbT_N$ one has
    \begin{equation*}
        \begin{cases}
            \phi^0_{\calR',\underline p,q}[\scrC_{\bullet}(\calR)]\ge\delta, \\
            \phi^1_{\calR',\underline p,q}[\scrC^{\star}_{\bullet}(\calR)]\ge\delta,
        \end{cases}
    \end{equation*}
    for $\bullet=h$ and $\bullet =v$. We denote $\textrm{RSW}(\delta,N)$ the set of measures $\phi_{\bbT_N,\underline p,q}$ for which the strong RSW property with parameter $\delta$ is satisfied on $\bbT_N$.
\end{definition}
In our setup, individual edge weights could in principle get close to $0$ or $1$. Therefore, we take the convention that a rectangle of dimensions $2\times 1$ or $1\times 2$ is a single edge, so that the $\textrm{RSW}(\delta,N)$ property implies that each edge is open with some probability between $\delta$ and $1-\delta$, and this implies that for each edge $e \in \bbT_N$, one has ${p_e \in \big[\delta,q(1-\delta)\big(\delta+q(1-\delta)\big)^{-1}\big]}$, remaining away from $0$ and $1$. It is standard to see (e.g.\ \cite[Theorem 2.1]{FK_scaling_relations} that the rectangle crossing introduced in Definition \ref{def:RSW-Torus} implies a stronger version of the strong box crossing property, for \emph{both} the primal and dual models. For example, assuming that $\phi_{\bbT_N,\underline{p},q}$ satisfies $\textrm{RSW}(\delta,N)$, if we consider the FK measure on an annulus ${A(R,2R)=\Lambda_{2R}\backslash\Lambda_R\subset \bbT_N}$ with \emph{free} boundary conditions, then with probability bounded from below by a positive constant depending only on $\delta$, there is a primal cluster in $A(R,2R)$ surrounding the inner boundary of the annulus. A similar result holds for the dual model.
We are now in position to state the first result of the paper, which proves that, when naively centerings the random bond parameters at the critical value, randomness with respect to the environment allows to deviate all coupling constants from the critical one inside $\bbT_N $ by some random variables whose standard deviation is of order $\textrm W(N)^{\frac{1}{2}}$.

\begin{theorem}
\label{thm:extension_scaling_window}
Fix $1<q\leq 4$. Then there exists $\delta=\delta(q)>0$ and $c=c(q)>0$ such that the following holds. Let $N\geq 1$ be an integer and consider a random environment  $\frak p=(\mathbf{p}_e)_{e\in \bbT_N}$ satisfying $(\star)_N^{\text A}$. Define $\Sigma_{\frak p}:=\max_{e\in E_N}\sigma_e$. Then, assuming that ${\Sigma_{\frak p}\le c\cdot \textrm W(N)^{1/2}}$, one has
\begin{equation*}
        \mathbf P_N\Big[\phi_{\bbT_N,\frak p,q}\in \textrm{RSW}(\delta,N)\Big]>1-O\left( \exp\Bigg[-c N^c\cdot \Bigg(\frac{\textrm W(N)^{1/2}}{\Sigma_{\frak p}}\Bigg)^{1/2}\Bigg] \right).
\end{equation*}
\end{theorem}
The previous statement shows that randomness in the environment sampled according to $\mathbf{P}_N$, enhances the mixing between slightly subcritical and slightly supercritical bond parameters (chosen at random), whose standard deviation matches the square root of the original critical window, while preserving the model within its critical phase. This result can, in fact, be improved. Under the assumption $(\star)_N^{\text A}$, the environmental randomness is, to first order, adapted to a linear deviation around the critical point. However, there is no \emph{a priori} reason for the environment to linearise perfectly around criticality (except in the case of percolation). In particular, one can extend the expansion of a self-dual setup beyond the first moment by incorporating second-order terms. This correction leads to the refined condition $(\star)_N^{\text B}$. Under this improved condition, one can in principle replace the critical window $\textrm W(N)$ by its cube root $\textrm W(N)^{1/3}$, up to some sub-polynomial correction. More precisely, define
\begin{equation}
\label{def:W_tilde}
    \widetilde{\textrm W}(N) := \textrm W(N) \Big( \sum_{r \leq N} r \Delta_{p_c}(r)^4 \Big)^{-1/2}.
\end{equation}
Predictions from conformal field theory and the convergence to CLE \cite{liu2025mixing,FK_scaling_relations} suggest that $\Delta_{p_c}(r) \asymp r^{-\iota(q)}$, where $\iota(q) \geq \frac{1}{2}$, with equality if and only if $q = 4$. In particular, one should have $(\sum_{r \leq N} r \Delta_{p_c}(r)^4)^{-1/2} \asymp 1$ when $1 < q < 4$, and ${(\sum_{r \leq N} r \Delta_{p_c}(r)^4)^{-1/2} = N^{-o(1)}}$ in the case $q = 4$. Let us mention an ongoing work \cite{FK_two_arm_exponent}, where notably the bound $\iota(q) > \frac{1}{4}$ can be derived for all $1<q\leq 4$, while no other results are known except for the FK-Ising model. Therefore, assuming conformal invariance of the critical limiting model, one gets $\widetilde{\textrm W}(N) \asymp \textrm W(N)$, with a potential sub-polynomial correction when $q=4$. One can now extend the critical window in a random environment to the cubic root of the original one, when working with a random environment naturally centred at $p_c(q)$. Let us point out that in the following theorem, the estimates and the exponents in the logarithmic factors depend on the tail bounds for the $\mathbf p_e$, and better tail bounds would allow us to get better estimates. Recall that $\Sigma_{\frak p}$ is defined as $\Sigma_{\frak p}:=\max_{e\in E_N}\sigma_e$.
\begin{theorem}
\label{thm:extension_scaling_window_drift}
Fix $1<q\leq 4$. Then there exists $\delta=\delta(q)>0$ and $c=c(q)>0$ such that the following holds. For any $N\geq 1 $, any random environment  $\frak p=(\mathbf{p}_e)_{e\in E_N}$ satisfying $(\star)_N^{\text B}$ and such that $\Sigma_{\frak{p}} \leq c\cdot \widetilde{\textrm{W}}(N)^{1/3} \log(N)^{-2}$, one has
\begin{equation*}
        \mathbf P_N\Big[\phi_{\bbT_N,\frak p,q}\in \textrm{RSW}(\delta,N)\Big]>1-O\left(\exp\left[-c\cdot \left(\frac{\widetilde{\textrm{W}}(N)^{1/3}}{\Sigma_{\frak p}}\right)^{1/2}\right]\right).
 \end{equation*}
\end{theorem}
One should read the previous theorem keeping in mind the conjectured convergence of FK-percolation models towards CLE loop ensembles. In that setup, (see e.g.\ \cite[Table Introduction]{FK_scaling_relations}) there should exist $\nu = \nu(q)>0$  such that $\textrm W(N)=N^{-\nu + o(1)}$. In the random environment, it is in principle possible to work with bond parameters whose standard deviation from the critical point satisfies $\max_{e\in \bbT_N}\sigma_{e} \asymp N^{-\frac{\nu}{3}+o(1)} \gg N^{-\nu}$ while remaining in the critical phase of the model, at least with high $\mathbf{P}_N$-probability. If one replaces $\frak{p}=(\mathbf{p}_e)_{e\in E_N} $ by $(p_c(q)+|\mathbf{p}_e-p_c(q)|)_{e\in E_N}$, the model would have exhibited a fairly off-critical behaviour. One could wonder about the existence of some $(\star)^{\text C}_N$ set of conditions on the random bonds $\frak{p}$ that would have kept $\Sigma_{\frak p}$ allowing some even (polynomially) larger window than $\textrm W(N)^{1/3}$. While we do not claim that finding such an environment $\frak p$ is impossible, we argue in Section \ref{sub:optimality-random-window} that for an environment with independent bonds $\mathbf p_e$, the window $\textrm W(N)^{1/3}$ is a natural barrier when only requiring approximate self-duality for the bond parameters. Still, one can explore some slightly less restrictive setup, allowing the bonds to be \emph{dependent}. For instance, one can construct a model where the bond parameters are the sum of an i.i.d\ process at each edge, together with a correction (of much smaller order of magnitude) which takes into account \emph{all} the edges in $\bbT_N$. Furthermore, the bond parameters in this interacting model should become essentially independent for edges far away from each other in $\bbT_N$. In this case, the (random) correction term is chosen in such a way that \emph{all} the terms of the form $\phi_{\bbT_N,\frak{p},q}[\scrC(\calR)]$ have expectation their value at criticality $\phi_{\bbT_N,p_c,q}[\scrC(\calR)]$. We explain how to achieve this in the proof of the following result. To lighten notations, set $\Xi(N):=\left(\sum_{r\le N}r\Delta_{p_c}(r)^2\right)^{-1}$.
\begin{theorem}\label{thm:non-independent-bonds}
Fix $1<q\le 4$ and let $\frak p=(\mathbf p_e)_{e\in E}$ a random environment satisfying $(\star)^{\text B}_N$. There exist some positive constants $\delta,O,c_{1,2}$ such that if $\Sigma_{\frak p}\le c_1\cdot \Xi(N)^{1/2}\log(N)^{-2}$, one can construct some modified environment $\widetilde{\frak p}=(\widetilde{\mathbf p}_e)_{e\in E_N}$ such that
\begin{itemize}
	\item Under $\mathbf P_N$, for each $e\in E_N$, the difference $|\widetilde{\mathbf p}_e-\mathbf p_e|$ is of order $\Xi(N)^{-1/2}\sigma_e^3$, in the sense that
	\begin{equation*}
        \forall \alpha>0,\quad \mathbf P_N\Big[|\widetilde{\mathbf p}_e-\mathbf p_e|>\alpha\cdot\Xi(N)^{-1/2}\sigma_e^3\Big]\lesssim \exp(-c\alpha^{1/3}).
    \end{equation*}
	\item One has \begin{equation*}
        \mathbf P_N\Big[\phi_{\bbT_N,\widetilde{\frak p},q}\in \textrm{RSW}(\delta,N)\Big]>1- O\Bigg(   \exp \Bigg[ -c_2\cdot \left(\frac{\Xi(N)^{1/2}}{\Sigma_{\frak p}}\right)^{1/2} \Bigg] \Bigg).
    \end{equation*}
\end{itemize}
\end{theorem}
In words, the \emph{random} correction to $\frak{p}$ is typically of order $\Sigma_\frak{p}^3 \ll \Sigma_\frak{p} $. This slightly interacting model becomes of particular interest when $1<q\leq 2$ where $\Xi(N)$ is expected to be constant (or logarithmic when $q=2$), providing an already logarithmic near-critical window in the random environment, similarly to the result we derive below for percolation.

\bigskip

Let us now turn to the specific case of Bernoulli bond percolation, corresponding to the random cluster model with parameter $q=1$. In this setting, we prove that the critical phase of the model in the random environment is preserved \emph{at every scale} for random bond parameters centred at the critical point, even when allowing variable whose standard deviation is \emph{logarithmic} in the size of the box. Moreover, if one only focuses on \emph{on the large scale crossing probabilities}, we show, using a short noise sensitivity argument, that the critical phase of the model is preserved even with \emph{macroscopic deformation} from the critical point, \emph{not} requiring the variances to decay to $0$ as $N\to \infty $. Before presenting precise results, let us define the natural moment condition on the random variables $\mathfrak{p}$ in the case $q=1$, where the dependencies are linear with respect to each edge bond. This leads to the definition of the condition $(\star)_{N}^{q=1}$, given by
\begin{enumerate}
\item (Independence) The variables $(\mathbf{p}_e)_{\in E_N}$ are mutually independent.
\item (Exact self-dual mean) They all have mean $\mathbb{E}_{\mathbf{P}_N}[\mathbf p_e]=p_c(1)=\frac{1}{2}$.
\item (Exponential tails) For each $e\in E_N$, the tails of $\mathbf p_e$ satisfy
    \begin{equation*}
        \forall s>0,\quad \bbP_N\Big[\big|\mathbf p_e-\frac{1}{2}\big|>s\Big]\lesssim \exp\left(-c\cdot\left(\frac s{\sigma_e}\right)\right)
    \end{equation*}
    for some fixed $c>0$ independent from $N$. 
    \end{enumerate}
We still denote $\sigma_e := \mathbb{E}[(p_e - \tfrac{1}{2})^2]^{1/2}$ and $\Sigma_{\mathfrak{p}} := \max_{e \in \mathbb{T}_N} \sigma_e$. In contrast to the general FK-percolation case with $1<q\leq 4$, one requires that the expectation of each $p_e$ be \emph{exactly} $\tfrac{1}{2}$, and not merely close to it. Indeed, the self-duality condition for the random environment can be expressed via its first moment only. Also note that when defining $\mathrm{RSW}(\delta, N)$ the boundary conditions ($0$ or $1$) on the rectangle $\calR'$ used in Definition~\ref{def:RSW-Torus} are irrelevant. This allows us to state the following theorem, which establishes that the permissible deviations from criticality can be of order $O(\log(N)^{-O(1)})$ in the torus $\bbT_N$ while remaining in the critical phase at every scale.
\begin{theorem}\label{thm:percolation-square}
There exists $\delta=\delta(q)>0$ and $c=c(q)$, such that for any $N\geq 1 $ and any random environment $\frak p=(\mathbf{p}_e)_{e\in E_N}$ satisfying $(\star)_N^{q=1}$ with $\Sigma_{\frak p}\le c\cdot \log(N)^{-2}$, one has
\begin{equation}
	\mathbf{P}_N\Bigg[ \phi_{\bbT_N,\frak p,1}\in \textrm{RSW}(\delta,N)\Bigg] >1- O\Bigg( \exp\Big[ - c\cdot (\Sigma_{\frak p})^{-1/2}  \Big] \Bigg).
\end{equation}
\end{theorem}
As one can see within the proof of this theorem, there are a priori \emph{no} conceptual reason to only use our technique to near-critical random environments. In the simplified Gaussian case, one can continuously move the random environments from the critical one to the Gaussian one using i.i.d.\ Brownian motions. Along the deformation, the crossing probability in a rectangle, whose initial value is bounded away from $0$ and $1$, is a local martingale with some uniformly bounded (in the size of the rectangle and also in $N$) quadratic variation, leaving some hope that one can control it up to some small but fixed positive time. Still, to carry out the main steps of our reasoning, it is essential to preserve criticality \emph{at every scale}, meaning that we have to control crossing probabilities for \emph{all} the $O(N^3)$ $2$ by $1$ rectangles of $\bbT_N$. Restricting to random variables with standard deviation  
 $O(\log(N)^{-O(1)})$ is sufficient to control crossing probabilities in all the rectangles of $\bbT_N$, via some crude large deviation principle. To go beyond this regime, one would like to develop some renormalisation argument, informally stating that if only a small fraction of the rectangles inside $\mathbb{T}_N$ become off-critical, then large-scale critical behaviour might persist. A more careful analysis shows that, using the same Brownian deformation process, the quadratic variation of the local martingales encoding crossing probabilities decays to $0$ with the size of the box (improving the uniform positive bound mentioned few lines ago). This hints that one can try to shortcut such renormalisation argument using \emph{noise sensitivity} of Bernoulli percolation \cite{benjamini1999noise}, which turns out to be sufficient.

Let us precise the previous discussion for site percolation on the triangular lattice, whose conformal invariance has been proved in \cite{Smi01} (see also \cite{WerSmi} for the rigorous  derivation of the critical exponents). Consider the infinite lattice $\mathcal{T}$ made of equilateral triangles of edge length $1$, where the segment $[0,1]\subset \bbC $ is the side of one the triangles. Define the triangle $\mathcal{T}_N$ whose three extremal vertices are located at $A_N=(-\frac{N}{2};0)$, $B_N=(\frac{N}{2};0)$ and $C_N=(0;\frac{\sqrt{3}}{2}N)$. The (homogeneous) critical site percolation on the triangular lattice is a random colouring of the vertices of $ \mathcal{T}_N$, where each site is coloured in blue or yellow with probability $\frac{1}{2}$ and independently from the others. Fix some positive number $x\in [0;1]$ and consider the point $x_N=(\lfloor(\frac{1}{2}-x)N\rfloor,0)$. In this context, denote $\phi_{\calT,\frac12,1}$ the fair site percolation measure on the infinite lattice $\calT$. One can define the crossing event
\begin{equation}
	\scrC([C_N;A_N],[x_N;B_N]):= \Big\{ [C_N;A_N] \textrm{ is connected to }[x_N;B_N] \textrm{ by a blue path in } \calT_N\Big\}.
\end{equation}
\begin{figure}
    \centering
    \includegraphics[width=0.5\textwidth]{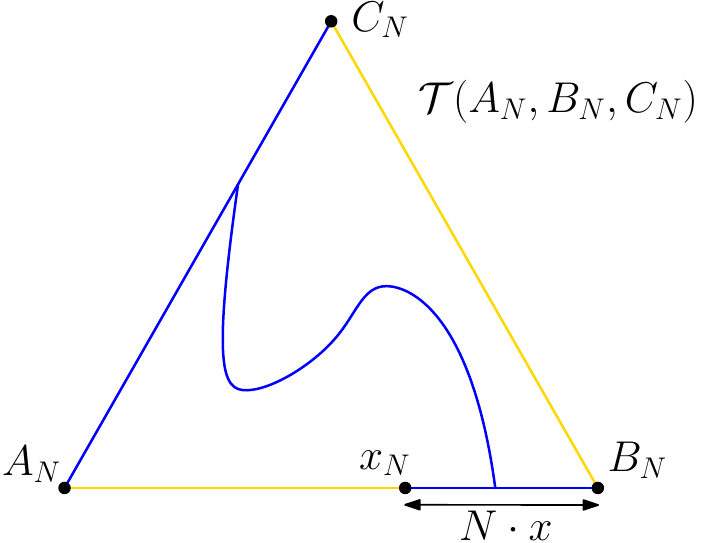}
    \caption{An illustration of the event $\scrC([C_N;A_N];[x_N;B_N])$, the blue curve depicts a path of blue hexagons.}
    \label{fig:cardy}
\end{figure}
The following beautiful theorem is due to Smirnov \cite{Smi01} settling Cardy's conjecture for crossing probabilities of percolation. One simplified version reads as 
\begin{theorem}[Cardy-Smirnov formula \cite{Smi01}]\label{thm:cardy_formula}
    In the previous context
    \begin{equation}
    	\lim\limits_{N\to \infty} \phi_{\calT,\frac12,1} \Bigg[ \scrC([C_N;A_N],[x_N;B_N]) \Bigg] = x.
    \end{equation}
\end{theorem}
In our framework, we take advantage of the stability of crossing events to see that the Cardy formula is preserved when working with in a random environments, even with macroscopic deviation. Fix $0< \varepsilon \leq \frac{1}{2}$ small enough. Similarly to the notations used for the square lattice, set $\frak{p}=(\mathbf{p}_{v})_{v\in \mathcal{T}}$ a random environment on the sites of $\calT$ under some probability $\mathbf{P}$. We say that the random environment $\frak p$ satisfies the condition $(\star)_{\mathcal{T}_N}^{q=1,\varepsilon}$ if the $\mathbf p_v,v\in \calT$ are independent and if they all almost surely belong to $[\eps,1-\eps]$. It turns out that these rather weak requirements are sufficient to preserve the Cardy formula in a random environment. 
\begin{theorem}[Cardy's formula in random environment]\label{thm:Cardy-random-environment}For any $\eps>0$ and any random environment $\frak p$ satisfying $(\star)_{\mathcal{T}_N}^{q=1,\varepsilon}$ one has the following convergence in probability
    \begin{equation}
    	\sup\limits_{x\in [0,1]}\Big| \phi_{\calT,\frak{p},1}\big[ \scrC([C_N;A_N];[x_N;B_N]) \big] - x \Big|\xlongrightarrow[N\to\infty]{(\mathbb{P})}0.
    \end{equation}
\end{theorem}
In principle, one can try to use the quantitative version of noise sensitivity \cite{garban2012noise} to obtain explicit speeds of convergence in the previous result. Adapting almost verbatim the proofs of Theorem \ref{thm:percolation-square} ensures that Theorem \ref{thm:Cardy-random-environment} already holds in the weakly random environment when chosing $\varepsilon=\varepsilon(N)=\frac{1}{2}-O(\log(N)^{-2})$ with polynomial rates of convergence. This Cardy formula in random environment is in sharp contrast with the work of Nolin and Werner \cite{nolin2009asymmetry}, which focuses on the deterministic near-critical percolation regime. In that setting, the critical window is given by $\textrm{W}(N) = N^{-3/4 + o(1)}$, and there exists a non-trivial scaling limit where the interface is neither degenerate nor trivial, exhibiting the same Hausdorff dimension $7/4$ as $\mathrm{SLE}(6)$, yet is \emph{not} absolutely continuous with respect to $\mathrm{SLE}(6)$. In particular, in the deterministic near-critical setup, the Cardy formula is not true anymore, while the additional layer of randomness in sorting the environment allows to preserve it.

\textit{Disclaimer: To the best of our knowledge, we did not find any prior reference to an argument relating the behaviour of the critical random bond percolation and noise sensitivity. In particular, we discovered this argument independently from any external inputs.}

\textit{Acknowledgemnts} First of all, we would like to warmly thank Hugo Duminil-Copin for encouraging us to work together and for sharing with us many insights. We would also like to thank Tiancheng He, Ioan Manolescu, and Stanislav Smirnov for inspiring discussions. Rémy Mahfouf would like to thank Dmitry Chelkak for discussions while working on \cite{Mah25}. This project has received funding from the Swiss National Science
Foundation and the NCCR SwissMAP. 
\section{A crash introduction into scaling relations}

\subsection{Recalling scaling relations in the homogenous setup}\label{sub:recalling_scaling_relations}
In this section, we recall in a concise fashion the so-called scaling relations that link the behaviour at large scale of the critical model to properties of the near-critical models. We present here some of the main results of \cite{FK_scaling_relations} and their counterparts when the bonds of the FK-percolation models are not uniform and translation anymore. We keep the notations of \cite{FK_scaling_relations}. Before going into precise statements, we recall standard properties of the FK-random cluster models when $1\leq q \leq 4 $. Let $\leq$ be the partial ordering on edges of $\{0,1 \}^{E}$ given by $\omega\le \omega'$ if and only if for each $e\in E$, $\omega(e)\le \omega'(e)$. An event $A$ is called \emph{increasing} if for any $\omega\leq \omega' $, the event $\omega \in A $ implies that $\omega'\in A$. Increasing events are those which are preserved when adding open edges to the clusters. Similarly, we say that the boundary condition $\xi'$ \emph{dominates} the boundary condition $\xi$ if the partition $\xi$ is finer than the partition $\xi'$ (that is to say any wired vertices in $\xi$ are automatically wired in $\xi'$). This is denoted $\xi'\geq \xi $. The FK-percolation models display some fairly simple properties (that do not require any kind of translation invariance) that we summarise here\begin{itemize}
	\item \textbf{Positive association} For $A,B$ increasing events and $\xi $ a boundary condition,
\begin{equation}
	\phi_{G,\underline p,q}^{\xi}[A\cap B]\geq \phi_{G,\underline p,q}^{\xi}[A] \cdot \phi_{G,\underline p,q}^{\xi}[B].
\end{equation}
	\item \textbf{Boundary monotonicity} For $A$ an increasing event and  $\xi'\geq \xi $,
	\begin{equation}
	\phi_{G,\underline p,q}^{\xi'}[A]\geq \phi_{G,\underline p,q}^{\xi}[A].
	\end{equation}
	\item \textbf{Domain Markov property} For $H$ a subgraph of $G$ and $\xi$ a boundary condition on $G$,
	\begin{equation}
		\phi_{G,\underline p,q}^{\xi}[\omega \textrm{ on } H| \omega \textrm{ on } G\backslash H]=\phi_{H,\underline p,q}^{\zeta}[\omega \textrm{ on } H],
	\end{equation}
	where $\zeta$ is the boundary condition on $\partial H$ induced by wiring vertices together if they are connected in $G\backslash H$ through $\omega$, potentially using connections through $\xi$. 
\end{itemize}

One on the key tools to understand the role of boundary conditions on events far from the boundary is the so-called \emph{mixing rate}, which replaces the $4$-arm probabilities in Kesten's original proof of scaling relation for bond percolation \cite{kesten1987scaling}. In most notations used until the end of the paper, the parameter $1\leq q \leq 4$ is not written in the probabilistic statements to lighten the reading.
\begin{definition}[\textbf{Mixing rate}- Definition 1.5 in \cite{FK_scaling_relations}]
	For $1< q \leq 4$, $1\leq r\leq R $, $p\in (0,1)$ and an edge $e_0$ incident to the origin, write
	\begin{equation}
		\Delta_{p}(R):=\phi_{\Lambda_R,p}^{1}[\omega_{e_0}]-\phi_{\Lambda_R,p}^{0}[\omega_{e_0}]
	\end{equation}
		\begin{equation}
		\Delta_{p}(r,R):=\phi_{\Lambda_R,p}^{1}[\scrC(\Lambda_r)]-\phi_{\Lambda_R,p}^{0}[\scrC(\Lambda_r)]
	\end{equation}
\end{definition}
We use the convention that $\Lambda_1=\{e_0\}$ so that for all $R$, one has $\Delta_{p}(R)=\Delta_p(1,R)$. In the uniform deterministic setup, the above quantities are invariant by translation thus defining them for some specific edge around the origin is costless. The following properties of the mixing rate are at the heart of \cite{FK_scaling_relations}. Recall that for $p\in (0,1)$, $L(p)$ is the scale under which the model with parameter $p$ satisfies RSW. In what follows, $d(e,f)$ denotes the distance between two edges $e$ and $f$.

\begin{theorem}[Theorem 1.4 and 1.6 in \cite{FK_scaling_relations}]\label{thm:scaling-relations}
Fix $1<q\leq4$.
\begin{enumerate}
	\item For every $\eta <1$,  $p\in (0,1)$, $R\leq L(p)$ and every $2$ by $1$ rectangle $\calR$ containing $\Lambda_{\eta R}(e) $, there exist constants $\asymp $ depending on $\eta$ such that
	\begin{equation}
		\textrm{Cov}_{p}[\omega_e,\scrC(\calR)] \asymp \Delta_p(R).
	\end{equation}
	\item (Quasi-multiplicativity) For every $p\in (0,1)$ and $1\leq r \leq R \leq L(p)$, 
\begin{equation}
	\Delta_p(r)\Delta_{p}(r,R) \asymp \Delta_p(R)
\end{equation}
	\item For any two edges $e,f$ such that $d(e,f)= R\le L(p) $, 
\begin{equation}
	\textrm{Cov}_{p}(\omega_e,\omega_f) \asymp \Delta_p(R)^2.
\end{equation}
\end{enumerate}
\end{theorem}

We note that all of the results above rely purely on RSW-type reasoning and estimates, and notably make very limited use of the model’s translation invariance, which we will discuss in more detail in the following section. We now provide a brief explanation of the scaling relation stating that the scaling window for FK-percolation satisfies $ \textrm W(N) \asymp \left(N^2 \Delta_{p_c}(N)\right)^{-1}$.
To formulate this in a way more suited to the random bond setting, we work on the torus $\mathbb{T}_N$. In this context, one may still define the mixing rates $\Delta_p(R)$ for $R \leq N$, which remain of the same order as in the full-plane case. Indeed, the mixing rates are defined based on measures with free boundary conditions on boxes, which does not depend on whether we are working in the plane or in the torus.
In the next theorem and its proof, one sets 
\begin{equation*}
    \widehat{\textrm{W}}(N):= (N^2\Delta_{p_c}(N))^{-1}
\end{equation*}
to be the ``expected value'' of the critical window.
The following theorem reads as the inequality $\textrm{W}(N)\gtrsim \widehat{\textrm{W}}(N)$.
\begin{theorem}[Theorem 1.4 and 1.6 in \cite{FK_scaling_relations}]\label{thm:scaling-relations_second}
Fix $1<q\leq4$, for any $N\geq 1 $,
\begin{enumerate}
	\item For every $2$ by $1$ rectangle $\calR$ in $\bbT_N$ of width $R\le N/2$ and for every $p\in (0,1)$ one has
\begin{equation}
	\Big| \phi_{\bbT_N,p}[\scrC(\calR)]-\phi_{\bbT_N,p_c}[\scrC(\calR)] \Big| \lesssim  \frac{|p-p_c|}{\widehat{\textrm W}(N)}.
\end{equation}
\item (Mixing rates are stable below the characteristic length) There is a small constant $c>0$ such that $|p-p_c |\leq c \cdot \widehat{\textrm{W}}(N) $ and any $R\le N$,
	\begin{equation}
		\Delta_p(R)\asymp \Delta_{p_c}(R),
	\end{equation}
where the mixing rates are defined for the homogeneous full-plane measure \eqref{eq:def-mixing-rate}.
\end{enumerate}
\end{theorem}
It turns out that this observation is the starting point of many reasonings appearing in the proofs in random environments, as our analysis often boils down to controlling mixing rates there. In this sketch of proof we only focus on the planar supercritical phase $p>p_c$, the subcritical phase can be treated similarly. Let us explicitly mention that all the ingredients in missing to this sketch are carefully performed in the proofs of Lemmas \ref{lem:stability-RSW} and \ref{lem:stability-StabDelta}.

\begin{proof}[Sketch of proof of Theorem \ref{thm:scaling-relations_second} when $p>p_c$] Let us start by proving the first item, by controlling the rate of growth of its derivative in $p$, which reads as (see e.g. \cite[(1.11)]{FK_scaling_relations}
\begin{equation}
\label{eq:crossing_derivative}
    \frac{\dd}{\dd p}\phi_{\bbT_N,p}[\scrC(\calR)]=\frac 1{p(1-p)}\sum_{e\in E_N}\Cov_p(\scrC(\calR),\omega_e)
\end{equation}
The analysis in \cite[(i) inTheorem 1.6]{FK_scaling_relations} estimates $\Cov_{p}(\scrC(\calR),\omega_e)$ using the mixing rates $\Delta_{p}$, via the measure with parameter $p$. Here, one wants to make a similar control, but \emph{only involving critical mixing rates} $\Delta_{p_c}$. This will involve a circular argument where mixing rates control the rate of growth (with $p$) of covariances between edges which control in return the rate of growth (still with $p$) of mixing rates. In order to derive some control via $\Delta_{p_c}$ quantities, one possible route is to show that $\Delta_p(R)$ remains comparable (up to constant) to $\Delta_{p_c}(R)$ when $R\le N$ and $p$ is not too far from $p_c$. To carry this analysis, one can first apply the third item of Theorem \ref{thm:scaling-relations} which reads as $\Delta_p(R)\asymp \sqrt{\Cov_p(\omega_e,\omega_f)}$, where $d(e,f)=R$. Differentiating this time the edge-covariances one gets
\begin{equation}
\label{eq:covariance_derivative}
    \frac{\dd}{\dd p}\Cov_p(\omega_e,\omega_f)=\frac 1{p(1-p)}\sum_{g\in E_N}\kappa_3^p(e,f,g),
\end{equation}
where the definition of the cumulant $\kappa_3^p(e,f,g)$ is recalled at the end of Proposition \ref{prop:differentiate-bonds}. Provided the FK measure $\phi_{\bbT_N,p}$ satisfies some $\textrm{RSW}(\delta,N)$ property, one can control each cumulant $\kappa_3^p(e,f,g)$ using associated mixing rates for the \emph{same} measure $\Delta_p(\cdot)$ as in \cite[Proposition 5.6]{FK_scaling_relations}. Let us now be more concrete about the aforementioned circular reasoning on controlling the whole FK measure when moving it within the critical window. Define the sets of measures on $\bbT_N$ whose mixing rates are comparable to the critical ones i.e.\ 
\begin{equation*}
    \textrm{Stab}_{\Delta}(\delta,N):=\left\{\phi_{\bbT_N,p}\Big| \forall R\le N,\quad \frac{\Delta_p(R)}{\Delta_{p_c}(R)}\in[\delta,1/\delta]\right\}.
\end{equation*}
This allows us to define the \emph{breaking probability} 
\begin{equation*}
    p_b(\delta,N):=\inf\left\{p>p_c\Big|\phi_{\bbT_N,p}\not\in \textrm{RSW}(\delta,N)\cap\textrm{Stab}_{\Delta}(\delta,N)\right\}.
\end{equation*}
In principle, one should be careful with the choice of a joint $\delta$ parameter in $\textrm{RSW}(\delta,N)$ and  $\textrm{Stab}_{\Delta}(\delta,N)$, as asking for a smaller $\delta$ in $\textrm{RSW}(\delta,N)$ means that one worsens constants in the estimate $\Delta_p(R)\asymp \sqrt{\Cov_p(\omega_e,\omega_f)}$ and hence one may not be able to show that the stability event $\textrm{Stab}_{\Delta}(\delta,N)$ holds with the same $\delta$. Assume now that one can bypass this difficulty (we refer to the proof of Lemma \ref{lem:stability-StabDelta} where this is addressed).  For any $p_c\leq p \leq p_{b}(\delta,N)$, one can replace, in all the above  formulae, the mixing rates for the measure $\phi_{\bbT_N,p}$ by the critical ones. Using some geometric estimates linking covariances and mixing rates (which are recalled in Section \ref{sub:geometric-estimates}), one gets, for a $2$ by $1$ rectangle $\calR$ of width $R$ and any two edges $e,f$ such that $d(e,f)=R$,
\begin{align*}
	\frac{\dd}{\dd p}\phi_{\bbT_N,p}[\scrC(\calR)]&\lesssim R^2\Delta_{p_c}(R)+\sum_{R<\ell\le N}\ell \Delta_{p_c}(\ell)\Delta_{p_c}(\ell,R)\\
	 \frac{\dd}{\dd p}\Cov_p(\omega_e,\omega_f)&\lesssim \Delta_{p_c}(R)^2\Bigg( R^2\Delta_{p_c}(R)+\sum_{R<\ell\le N}\ell \Delta_{p_c}(\ell)\Delta_{p_c}(\ell,R)\Bigg).
\end{align*}
We will see in Remark \ref{rem:arm-estimates-transation-invariant-environment} that there exists $c=c(q,\delta)>0$ such that $\Delta_{p_c}(R)\gtrsim R^{2-c}$. Therefore one gets for $p_c\leq p \leq p_{b}(\delta,N)$
\begin{align}
	\frac{\dd}{\dd p}\phi_{p}[\scrC(\calR)]&\lesssim N^2\Delta_{p_c}(N)=\widehat{\textrm{W}}(N)^{-1}\label{eq:crossing_derivative_estimate},\\
	\frac{\dd}{\dd p}\Cov_p(\omega_e,\omega_f)&\lesssim \Delta_{p_c}(R)^2\cdot N^2\Delta_{p_c}(N)= \Delta_{p_c}(R)^2\widehat{\textrm{W}}(N)^{-1}\label{eq:covariance_derivative_estimate}.
\end{align}
To conclude, one should show that $p_b(\delta,N)-p_c \gtrsim \widehat{\textrm{W}}(N)$. Assuming this fails, there exist a subsequence of $N$ such that $|p_b(\delta,N)-p_c|=o(\widehat{\textrm{W}}(N))$. Integrating \eqref{eq:crossing_derivative_estimate} and \eqref{eq:covariance_derivative_estimate} between $p_c$ and $p_b(\delta,N)$, one would have $|\phi_{p_b(\delta,N)}[\scrC(\calR)]-\phi_{p_c}[\scrC(\calR)]|=o(1)$ while $\frac{\dd}{\dd p}\Cov_p(\omega_e,\omega_f)=o(\Delta_{p_c}(R)^2)=o(\Cov_{p_c}(\omega_e,\omega_f))$. Using the fact that crossing probabilities of $2$ by $1$ rectangles in $\bbT_N$ are bounded away from $0$ and $1$ at $p_c$, we get that both crossing probabilities of $2$ by $1$ rectangles and covariances between edges change by some multiplicative $1+o(1)$ factor when going from $p_c$ to $p_b(\delta,N)$. For some small enough $\delta$, this will imply that the the measure $\phi_{\bbT_N,p_b(\delta,N),q}$ satisfies $\textrm{RSW}(2\delta,N)$ and $\textrm{Stab}_{\Delta}(2\delta,N)$ (note that for RSW, one actually requires estimates on crossing probabilities \emph{with boundary conditions}, but we ignore this difficulty here). By continuity in $p$ all the $\phi_{p}[\scrC(\calR)]$ and $\Delta_p(R)$ for all the $2$ by $1$ rectangles $\calR\subset \bbT_N$, this contradicts the definition of the breaking bond probability $p_b(\delta,N)$. 

Once we have proven the inequality $p_b(\delta,N)-p_c \gtrsim \widehat{\textrm{W}}(N)$, it is straightforward to conclude the proof of Theorem \ref{thm:scaling-relations_second}. For the first point, if $p>p_b(\delta,N)$, one crudely bounds $\Big| \phi_{\bbT_N,p}[\scrC(\calR)]-\phi_{\bbT_N,p_c}[\scrC(\calR)] \Big| $ by $1$ and uses the fact that $p-p_c\gtrsim\widehat{\textrm W}(N)$. If ${p\in [p_c,p_b(\delta,N)]}$, one simply integrates (\ref{eq:covariance_derivative_estimate}) between $p_c$ and $p$. For the second point, one may choose $c$ so that $c\widehat{\textrm W}(N)\le p_b(\delta,N)-p_c$, and then the result comes from integrating equation (\ref{eq:covariance_derivative_estimate}) between $p_c$ and $p_b(\delta,N)$, and using the relation between mixing rates and covariances $\Cov_p(\omega_e,\omega_f)\asymp \Delta_p(d(e,f))^2$.\end{proof}

\subsection{Scaling relations in non-translation invariant environments}\label{sub:scaling_relations_random_environment}

When working with a generic environment $\underline p = (p_e)_{e\in E_N}$ on $\mathbb{T}_N$, the obtained probability measures are no longer translation invariant, which creates additional (but in fact limited) complications within the proofs. In our approach focused in keeping some control of the stochastic derivatives of crossing probabilities, mixing rates and arm events, it is necessary to generalise the properties of Theorem \ref{thm:scaling-relations} around each edge $e$ provided the environment $\underline p$ still satisfies $\textrm{RSW}(\delta,N)$. More precisely, for some bond environement $\underline p=(p_e)_{e\in E_N}$ on $\bbT_N$, one can define a local notion of mixing rates around each edge. In what follows, set $\Lambda_R(e):=e+\Lambda_R$ the translated box of center $e$. For $e\in \mathbb{T}_N$ and $1\le r\le R\le N/4$, define
\begin{equation}
    \Delta^{(e)}_{\underline p}(R):=\phi_{\Lambda_R(e),\underline p,q}^1[\omega_e]-\phi_{\Lambda_R(e),\underline p,q}^0[\omega_e]
\end{equation}
\begin{equation}
    \Delta^{(e)}_{\underline p}(r,R):=\phi_{\Lambda_R(e),\underline p,q}^1\big[\scrC[\Lambda_r(e)]\big]-\phi_{\Lambda_R(e),\underline p,q}^0\big[\scrC[\Lambda_r(e)]\big]
\end{equation}
The following proposition generalises Theorem \ref{thm:scaling-relations} to non-translation invariant measures that still satisfy some strong box crossing property. 
\begin{proposition}\label{prop:generalisation-thm-scaling-relations}
Assume that the edge weights $\underline p=(p_e)_{e\in \mathbb{T}_N}$ satisfy $\textrm{RSW}(\delta,N)$. In what follows all the constants $\asymp $ depend a priori on the choice of $\delta$. 
\begin{enumerate}
   \item (Mixing-rate/covariance) For any two edges $e,f\in \mathbb{T}_N$ such that $d(e,f)=R < N/4$,    
  \begin{equation}
        \Cov_{\underline p}(\omega_e,\omega_f)\asymp \Delta^{(e)}_{\underline p}(R)\Delta^{(f)}_{\underline p}(R).
    \end{equation}
    \item (Quasimultiplicativity) For any edge $e\in \mathbb{T}_N$ any $1\leq r<R\le N/4$  one has
    \begin{equation}
        \Delta^{(e)}_{\underline p}(r)\Delta^{(e)}_{\underline p}(r,R)\asymp \Delta^{(e)}_{\underline p}(R).
    \end{equation}
    \item (Comparability at small scales) For any two edges $e,f\in \mathbb{T}_N$ such that $d(e,f)\leq r$ and any $1\leq r<R\le N/4$, one has
    \begin{equation}
        \Delta^{(e)}_{\underline p}(r,R)\asymp \Delta^{(f)}_{\underline p}(r,R).
    \end{equation}
\end{enumerate}
\end{proposition}

Providing a complete proof of this proposition would require introducing a lot of additional side tools involved in \cite{FK_scaling_relations}, but ending up doing the same proof, almost verbatim. As those statements go beyond the general scope and philosophy of the present paper, we rather provide some intuition on why those results still hold, and then sketch the proof of the second and third item. Fix $e\in E_N$, center of the boxes $\Lambda_r(e)\subset\Lambda_R(e)$ for $r\le R\le N/4$. Assume that on $\partial \Lambda_R(e)$ one starts with a pair of boundary conditions $\xi\le \xi'$, such that $\xi'$ is ``much more wired" than $\xi$ (this can be formalised rigorously with the notion of \emph{boosting pair of boundary conditions} in \cite[Section 3]{FK_scaling_relations}). The monotonicity properties of FK-percolation models and the coupling via decision trees allow to couple two configurations $\omega,\omega'$ distributed according to $\phi_{\Lambda_R,\underline p,q}^{\xi}$ and $\phi_{\Lambda_R,\underline p,q}^{\xi'}$ in such a way that $\omega\le \omega'$. Assume we are given such a coupling, one can start by revealing all of the edges in the annulus $\Lambda_R(e)\backslash \Lambda_r(e)$ in both $\omega$ and $\omega'$, and read the boundary conditions $\zeta,\zeta'$ imposed respectively by $\omega_{|\Lambda_R(e)\backslash \Lambda_r(e)},\omega'_{|\Lambda_R(e)\backslash \Lambda_r(e)}$ to $\partial \Lambda_r(e)$. The key feature of the coupling via decision trees studied in \cite{FK_scaling_relations} reads in the non-translation invariant setup as
\begin{equation}\label{eq:much-more-wired}
    \Delta^{(e)}_{\underline p}(r,R)\asymp \bbP[\zeta\ne\zeta']\asymp \bbP[\zeta'\text{ is ``much more wired than" } \zeta].
\end{equation}
The proof of such a statement only requires $\textrm{RSW}(\delta,N)$ to transfer information through different scales. Notably, no translation invariance of the model appears within the proofs.
\begin{proof}[Sketch of the proof of Proposition \ref{prop:generalisation-thm-scaling-relations} ]
Let us start by explaining how the second item follows from \eqref{eq:much-more-wired}. We do not claim originality here, repeating in a concise way the arguments of \cite{FK_scaling_relations} providing some intuition to the reader less familiar with the use of mixing rates. The proof of the first item is skipped as it follows the same spirit as the one for the second item, but exploring instead the annuli from inside to outside.

\paragraph{Proof of the second item.} 
We keep the above setup with $\xi=0$ and $\xi'=1$. In that case, $\Delta^{(e)}_{\underline p}(R)$ is \emph{exactly} the probability, when coupled in a monotone way (i.e. so that $\omega'\ge \omega$), that the configurations $\omega$ and $\omega'$ \emph{do not agree at the edge} $e$. On the other hand, one can start by revealing all the edges in the annulus $\Lambda_R(e)\backslash\Lambda_r(e)$ and obtain boundary conditions $\zeta,\zeta'$ on $\partial\Lambda_r(e)$. The Domain Markov Property implies that (using here that $\omega'\ge \omega$)
\begin{equation*}
    \bbP[\omega_e\ne\omega'_e|\zeta,\zeta'] = \phi_{\Lambda_r(e),\underline p,q}^{\zeta'}[\omega_e]-\phi_{\Lambda_r(e),\underline p,q}^{\zeta}[\omega_e].
\end{equation*}
The RHS of the above equation vanishes when $\zeta=\zeta'$. Assuming that $\zeta\ne \zeta'$, one may use the crude stochastic dominations for $0\le \zeta$ and $\zeta'\le 1$ to see that for any $\zeta\le\zeta'$,
\begin{equation*}
    \bbP[\omega_e\ne\omega'_e|\zeta,\zeta'] \le \phi_{\Lambda_r(e),\underline p,q}^{1}[\omega_e]-\phi_{\Lambda_r(e),\underline p,q}^{0}[\omega_e]=\Lambda^{(e)}_{\underline p}(r).
\end{equation*}
Using that $\bbP[\zeta\ne \zeta']\asymp \Delta^{(e)}_{\underline p}(r,R)$, one gets a first inequality
\begin{equation*}
    \Delta^{(e)}_{\underline p}(R)=\bbP[\omega_e\ne\omega'_e] \le \bbP[\zeta\ne\zeta']\Delta^{(e)}_{\underline p}(r)\asymp \Delta^{(e)}_{\underline p}(r,R)\Delta^{(e)}_{\underline p}(r).
\end{equation*}
Let us prove the converse (up to constant) inequality. Assuming that $\zeta'$ is ``much more wired" than $\zeta$, one works for the simplest case where $\Lambda_1(e):=\{e\}$ and reveal the configurations $\omega,\omega'$ in the annulus $\Lambda_R(e)\backslash\Lambda_r(e)$. The boundary conditions obtained for $\omega$ and $\omega'$ on $\{e\}$ are denoted $\widehat\zeta,\widehat\zeta'$. For the single edge $\{e\}$, only the wired and the free boundary conditions are possible. This enforces on the event  $\widehat\zeta\ne\widehat\zeta'$ that $\widehat\zeta=0$ and $\widehat\zeta'=1$. As a consequence 
\begin{align*}
    \bbP\big[\omega_e\ne\omega'_e\big|\widehat\zeta,\widehat\zeta'\big]&=\phi_{\{e\},\underline p,q}^1[\omega_e]-\phi_{\{e\},\underline p,q}^0[\omega_e]\\
    &=p_e-\frac{p_e}{p_e+q(1-p_e)}\gtrsim 1,
\end{align*}
where in the inequality, we use $q>1$ and the fact that the bond parameter $p_e$ is bounded away from $0$ and $1$ due to the $\textrm{RSW}(\delta,N)$ property. Furthermore, on the event $E$ corresponding to $\zeta'$ being ``much more wired" that $\zeta$, one can apply the reasoning above this proof to the annulus $\Lambda_r(e)\backslash\Lambda_1(e)$ which implies that $\bbP\big[\widehat\zeta\ne\widehat\zeta'\big|\zeta,\zeta'\big]\asymp \Delta_{\underline p}^{(e)}(r)$. All together, this gives 
\begin{align*}
    \Delta_{\underline p}^{(e)}(R)=\bbP[\omega_e\ne \omega_e']&\ge\bbP\big[\omega_e\ne\omega_e',\widehat\zeta\ne\widehat\zeta',E\big]\\
    &=\bbP\big[\omega_e\ne\omega_e'\big|\widehat\zeta\ne\widehat\zeta',E\big]\cdot\bbP\big[\widehat\zeta\ne\widehat\zeta'\big|E\big]\cdot\bbP[E]\\
    &\gtrsim 1\cdot\Delta_{\underline p}^{(e)}(r)\cdot\Delta_{\underline p}^{(e)}(r,R).
\end{align*}
\paragraph{Proof of the third item.}
The goal here is to prove $\Delta^{(f)}_{\underline p}(r,R)\gtrsim \Delta^{(e)}_{\underline p}(r,R)$, the other inequality follows by symmetry. Notice that one can always assume that $R>16r$, otherwise both terms are uniformly bounded away from $0$ and $1$ by the $\textrm{RSW}(\delta,N)$ property. Consider the annulus $A=\Lambda_{R}(e)\backslash\Lambda_{r}(e)$, which contains the annulus $A'=\Lambda_{R/2}(f)\backslash\Lambda_{2r}(f)$ (as $d(e,f)\le r\le R/2$). Consider $\omega\le\omega'$ obtained monotone coupling of the measures $\phi_{\Lambda_R(e),\underline p,q}^1$ and $\phi_{\Lambda_R(e),\underline p,q}^0$. We reveal successively the edges of $A$ in both $\omega$ and $\omega'$ as follows:
\begin{enumerate}
    \item First reveal all the edges of $\Lambda_R(e)\backslash \Lambda_{R/2}(f)$, and let $\zeta_1,\zeta_1'$ be the boundary conditions respectively enforced by $\omega$ and $\omega'$ on $\partial\Lambda_{R/2}(f)$.
    \item Second, reveal all the edges in $A'$, and let $\zeta_2,\zeta_2'$ be the boundary conditions respectively enforced by $\omega$ and $\omega'$ on $\partial\Lambda_{2r}(f)$.
    \item Finally, reveal all edges in $\Lambda_{2r}(f)\backslash\Lambda_r(e)$, and let $\zeta_3,\zeta_3'$ be the boundary conditions respectively enforced by $\omega$ and $\omega'$ on $\partial\Lambda_r(e)$.
\end{enumerate}
The connection between mixing rates and exploration by decision trees \eqref{eq:much-more-wired} ensures that $\Lambda^{(e)}_{\underline p}(r,R)\asymp \bbP[\zeta_3\ne\zeta_3']$. On the other hand, for any $\zeta_1,\zeta_1'$, monotonicity with respect to boundary conditions gives that
\begin{equation*}
    \bbP[\zeta_2\ne \zeta_2'|\zeta_1,\zeta_1']\le \bbP[\zeta_2\ne \zeta_2'|\zeta_1=0,\zeta_1'=1]\asymp \Delta_{\underline p}^{(f)}(2r,R/2).
\end{equation*}
Notice that if $\zeta_2=\zeta_2'$, the configurations $\omega$ and $\omega'$ must agree inside $\Lambda_{2r}(f)$, imposing that $\zeta_3=\zeta_3'$. Therefore, un-conditioning the previous equation on $\zeta_1,\zeta_1'$ gives that
\begin{align*}
    \Delta_{\underline p}^{(e)}(r,R)\asymp \bbP[\zeta_3\ne \zeta_3']\le \bbP[\zeta_2\ne\zeta_2']\lesssim \Delta_{\underline p}^{(f)}(2r,R/2).
\end{align*}
Now by quasimultiplicativity which we have just proved above, one may write
\begin{equation*}
    \Delta_{\underline p}^{(f)}(2r,R/2)\asymp\frac{\Delta_{\underline p}^{(f)}(r,R)}{\Delta_{\underline p}^{(f)}(r,2r)\Delta_{\underline p}^{(f)}(R/2,R)}
\end{equation*}
But now some classical RSW estimates allow us to show that for any edge $g\in E_N$ and $\ell\le N/4$, one has $\Delta^{(g)}_{\underline p}(\ell,2\ell)\asymp 1$, so that putting everything together gives $\Delta^{(e)}(r,R)\lesssim \Delta^{(f)}(r,R)$ as desired.
\end{proof}

\subsection{Geometric estimates on covariances with crossings and arm events}\label{sub:geometric-estimates}
In this section we introduce the notion of arm events, which encode the probability of alternating primal and dual clusters meeting around a point. This notion, which is one of the central objects in the analysis of Kesten's original proof of scaling relations for Bernoulli percolation, provides some meaningful control on the derivatives of crossing events. Proofs of all of the results in this section can be found in Appendix \ref{app:geometric_estimates}.
\begin{definition}
Let $e\in E_N$ be an edge and fix $1\leq r\le R\le N/2$. One can identify $\bbT_N $ with $(\bbZ/N\bbZ)^2$, assuming that the image of $e$ is $(0,0)\in (\bbZ/N\bbZ)^2$. We define the \emph{upper half-plane} in $\bbT_N$ centred at $e$ as $\mathbb{H}^{(e)}_{N}:=(\mathbb{Z}/N\mathbb{Z})\times [0,(N-1)/2]$. One can define similarly the \emph{lower, left and right half-planes}. This allows us to define (see also Figure \ref{fig:arm-events}):
\begin{enumerate}
    \item The $4$-arm event $\calA_4^{(e)}(r,R)$ around $e$, which consists of the existence of $4$ alternating parity (primal and dual) arms linking $\partial\Lambda_r(e)$ to $\partial\Lambda_R(e)$ in the annulus $\Lambda_R(e)\backslash\Lambda_r(e)$. We denote by $\pi_{4,\underline p}^{(e)}(r,R)$ the probability of this event under $\phi_{\bbT_N,\underline p}$.
    \item The upper half-plane three-arm event $\calA_{3^+}^{(e)}(r,R)$ around $e$, which consists of to the existence of $3$ alternating parity ($2$ primal and $1$ dual) arms linking $\partial\Lambda_r(e)$ to $\partial\Lambda_R(e)$ inside the half- annulus $\Lambda_R(e)\backslash\Lambda_r(e) \cap \mathbb{H}^{(e)}_{N}$. We denote by $\pi_{3^+,\underline p}^{(e)}(r,R)$ the probability of this event under $\phi_{\bbT_N,\underline p}$. One can define similarly the lower, right and left half-plane three arm events denoted respectively $\calA_{3^-}^{(e)}(r,R),\calA_{3^{\textrm{rg}}}^{(e)}(r,R)$ and $\calA_{3^{\textrm{lf}}}^{(e)}(r,R)$.
\end{enumerate}
\end{definition}

\begin{figure}
    \centering
    \includegraphics[width=0.40\textwidth]{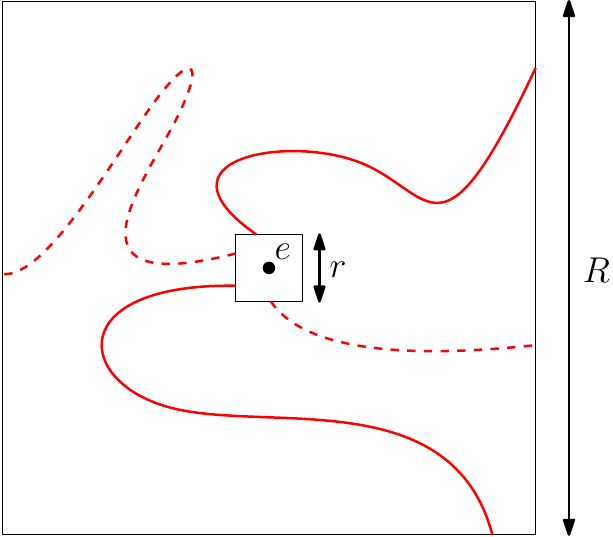}
    \hspace{0.08\textwidth}
    \includegraphics[width=0.40\textwidth]{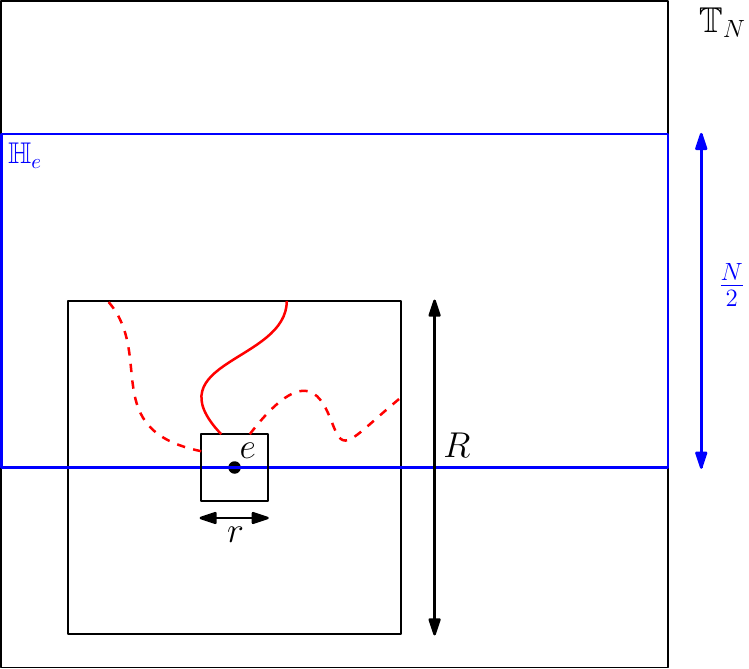}
    \caption{An illustration of the events $\calA^{(e)}_4(r,R)$ (left) and $\calA^{(e)}_{3^+}(r,R)$ (right). Primal paths are represented by solid  red curves and dual paths are represented by dashed red curves.}
    \label{fig:arm-events}
\end{figure}

\begin{remarks}\label{rem:arm-estimates-transation-invariant-environment}
At criticality, the probabilities of the arm events no longer depend on the edge $e$ so we remove it from the notation, and one can estimate the aforementioned arm events uniformly in $r\le R\le N/2$. There exist constants $\lesssim,\asymp$ and $c>0$, only depending on $\delta$ such that
\begin{enumerate}
    \item[(i)] $\pi_{4,p_c}(r,R)\lesssim (r/R)^c\Delta_{p_c}(r,R)$,
    \item[(ii)] $\pi_{3^+,p_c}(r,R)\asymp (r/R)^2$,
    \item[(iii)] $(r/R)^c\gtrsim \Delta_{p_c}(r,R) \gtrsim (r/R)^{2-c} $,
    \item[(iv)] For $q=1$, $(r/R)^{1+c}\gtrsim\pi_{4,p_c}(r,R) \gtrsim (r/R)^{2-c}$.
\end{enumerate}
The first and third items are proved in \cite[Theorem 1.6 (iv), Proposition 1.8]{FK_scaling_relations} and the second item is proved in \cite[Proposition 6.6 (6.7)]{duminil2021planar} for $q<4$, but the proof adapts verbatim for $q=4$. The fourth item reads as the fact that the four-arm exponent $\alpha_4$ is strictly bigger than $1$ and strictly smaller than $2$ for percolation.
The upper bound has been proven by different means (see e.g.\ \cite{van2020four} for an extensive exposition) while the upper bound is proven in \cite[Proposition 6.8]{duminil2021planar} for example.
\end{remarks}
In this section we present, in the simplified critical setup at $p_c$, one additional technicality that appears in the computations in non-translation invariant random environments. Namely, it will be crucial to  estimate the sum of \emph{squares} of covariances between an edge and a crossing (this is in contrast with \cite{FK_scaling_relations} where it is enough to control the sum of covariances to run the proofs). Let us start by presenting more precise estimates regarding the covariances between an edge some crossing events.
\begin{lemma}
\label{lem:Cov-crossing-edge}
Let $R\le N/2$, $\calR$ be a $2$ by $1$ rectangle of width $R$ and $e \in E_N$. Then one can estimate $\Cov_{p_c}(\scrC_{\bullet}(\calR),\omega_e)$ by:
\begin{itemize}
    \item If $ \textrm{dist}(e,\calR)=\ell \ge R$ one has
    \begin{equation*}
        \Cov_{p_c}(\scrC_{\bullet}(\calR),\omega_e)\asymp \Delta_{p_c}(R,\ell)\Delta_{p_c}(\ell)
    \end{equation*}
    \item If the distance from $e$ to of the four corners of $\calR$ is $n\lesssim R$, one has
    \begin{equation*}
        \Cov_{p_c}(\scrC_{\bullet}(\calR),\omega_e)\lesssim \Delta_{p_c}(n).
    \end{equation*}
\end{itemize}
A similar estimate holds for arbitrary boundary conditions as well as dual crossings.
\end{lemma}
The proof of this lemma, which is key to deduce stability of arm events in non-translation invariant environments, is delayed to Appendix \ref{app:geometric_estimates}. Let us still mention that in the second case, the correct upper bound should be of order $\Delta_{p_c}(R)$, but proving this requires better estimates on certain arm events than one could get from RSW alone.
\begin{remarks}
\label{rem:sum-Cov-crossing-edge}
Under the previous assumptions, one has
\begin{align}
    \sum_{e\in E_N}\Cov_{p_c}(\scrC_{\bullet}(\calR),\omega_e)&\lesssim \sum_{R< r\le N}r \Delta_{p_c}(R,r)\Delta_{p_c}(r)+\sum_{r\le R}r\Delta_{p_c}(r)\\
    &\lesssim \sum_{r \le N}r\Delta_{p_c}(r)\asymp \textrm W(N)^{-1}\label{eq:sum_cov_crossing_edge}
\end{align}
and
\begin{align}
    \sum_{e\in E_N}\Cov_{p_c}(\scrC_{\bullet}(\calR),\omega_e)^2&\lesssim \sum_{R<r\le N}r\Delta_{p_c}(R,r)^2\Delta_{p_c}(r)^2 + \sum_{r\le R} r\Delta_{p_c}(r)^2 \\
    &\lesssim \sum_{r\le N}r\Delta_{p_c}(r)^2=\Xi(N)^{-1}, \label{eq:sum_cov_crossing_edge_squared}
\end{align}
where we recall the definition $\Xi(N)=\big(\sum_{r\le N}r\Delta_{p_c}(r)^2\big)^{-1}$.
\end{remarks}

We now some analogous statement for covariances between edges and arm events.
\begin{lemma}
\label{lem:Cov-arms-edge}
    Let $e,f\in \bbT_N$ and $R\le N/2$. Then for $\#\in\{4,3^+,3^{-},3^{\textrm{rg}},3^{\textrm{lf}}\}$, one has
 \begin{equation*}
        \Cov_{p_c}(\calA_{\#}^{(e)}(R),\omega_f)\lesssim \pi_{\#}(R)\Delta_{p_c}(n),
    \end{equation*}
where $n$ is the minimal distance from $f$ to either $e$ or one of the bottom corners of $\Lambda_R(e)\cap \mathbb{H}^{(e)}_N$.
\end{lemma}

\begin{remarks}
\label{rem:sum-Cov-arms-edge}
Summing Lemma \ref{lem:Cov-arms-edge} over edges in $\bbT_N$ one gets for any $e\in E_N$ and $1\leq R \leq N/2$,
\begin{align}
	    \sum_{f\in E}\Cov_{p_c}(\calA_{\#}^{(e)}(R),\omega_f)&\lesssim \pi_{\#,p_c}(R)\sum_{r\le N}r\Delta_{p_c}(r)\asymp \pi_{\#,p_c}(R)\textrm W(N)^{-1} \label{eq:sum_cov_arm}\\
	   \sum_{f\in E}\Cov_{p_c}(\calA_{\#}^{(e)}(R),\omega_f)^2 &\lesssim \pi_{\#,p_c}(R)^2\sum_{r\le N}r\Delta_{p_c}(r)^2=\pi_{\#,p_c}(R)^2\Xi(N)^{-1}. \label{eq:sum_cov_arm_squared}
\end{align}
\end{remarks}
All of the previous results are valid for $1<q\le 4$. For $q=1$, we may also get similar estimates by replacing mixing rates by four-arm events. We note that in this case, the fourth item of Remark \ref{rem:arm-estimates-transation-invariant-environment} allows us to show that $\Xi(N)$ is of constant order.
\begin{lemma}
\label{lem:geometric-estimates_perco}
    Assume that $q=1$, and that we are at $p_c(1)=\frac12$. Then all of the results of Lemmas \ref{lem:Cov-crossing-edge} and \ref{lem:Cov-arms-edge}, as well as the computations that follow in Remarks \ref{rem:sum-Cov-crossing-edge} and \ref{rem:sum-Cov-arms-edge} still hold, when replacing all mixing rates $\Delta_{p_c}(r,R)$ by four-arm events $\pi_{4,p_c}(r,R)$. Furthermore, in the computations in the remarks, one may replace $\Xi(N)$ by $1$.
\end{lemma}

The quantity $\Xi(N)$ will appear numerous times throughout the paper, so we would like to relate it to the critical window $\textrm W(N)$. We show a slightly stronger result which will be useful to us later. Recall that $\widetilde{\textrm W}(N)$ is the modified critical window, given by
\begin{equation*}
    \widetilde{\textrm W}(N) = \textrm W(N)\left(\sum_{r\le N}r\Delta_{p_c}(r)^4\right)^{-1/2}.
\end{equation*}
The following lemma is a simple computation, whose proof is postponed to the Appendix.
\begin{lemma}
\label{lem:window-inequalities}
    There is some constant $c>0$ such that one has the following inequalities

    \begin{equation*}
        \textrm W(N)N^c\lesssim\widetilde{\textrm W}(N)^{2/3}\lesssim \Xi(N).
    \end{equation*}
\end{lemma}

\subsection{Derivatives of events for non-translation invariant environments}
\label{sub:derivative-events}
The overall approach that we take in the present paper is to deform continuously, via a continuous family of bonds $\frak{p}(s)$, all the edge weights from the critical homogeneous setup at $p_c$ towards the set of weights $\frak{p}$. In the deterministic setup, at criticality and below the characteristic length, the strong box-crossing property allows us to deduce all relations recalled inside Theorems \ref{thm:scaling-relations} and \ref{thm:scaling-relations_second}. As time evolves along the deformation process, the fact that some $\textrm{RSW}(\delta,N)$ property remains true (with high probability) for the set of weights $\frak{p}(s)$ ensures that 
Proposition \ref{prop:generalisation-thm-scaling-relations} remains true for the weights $\frak{p}(s)$ and with all involved quantities remaining the same as the homogeneous ones up to multiplicative constant. Therefore, one can run some infinitesimal deformation from the weights $\frak{p}(s)$ to  $\frak{p}(s+ds)$, which implies in return that $\textrm{RSW}(\delta,N)$ holds at time $s+ds$, implying in return Proposition \ref{prop:generalisation-thm-scaling-relations} for the weights $\frak{p}(s+ds)$. 

One can now compute explicitly the derivatives of different quantities with respect to each edge bond $p_e$. The proof of the following proposition is given in Appendix \ref{app:derivative_formulas}.
\begin{proposition}\label{prop:differentiate-bonds}
Given an edge $e$ in $E_N$, under the measure $\phi_{\underline{p}}=\phi_{\bbT_N,\underline{p},q}$, the derivatives of the probability of an event $A$ occurring in $\bbT_N$ are given by 
\begin{align*}
	\frac{\partial}{\partial p_e}\phi_{\underline{p}}[A]&= \frac1{p_e(1-p_e)}\textrm{Cov}_{\underline{p}}(A,\omega_e), \\
	\frac {\partial^2}{\partial p_e^2 }\phi_{\underline{p}}[A]& = \frac 2{p_e(1-p_e)^2} \left(1-\frac1{p_e}\phi_{\underline p}[\omega_e]\right) \textrm{Cov}_{\underline{p}}(A,\omega_e).
\end{align*}
Furthermore, if $e,f,g$ are three edges in $E_N$, the derivatives of $\Cov_{\underline p}(\omega_e,\omega_f)$ are given by
\begin{align*}
\frac{\partial }{\partial p_g}\Cov_{\underline{p}}(\omega_e,\omega_f)	  &=\frac{1}{p_g(1-p_g)}\kappa_{3}^{\underline{p}}(e,f,g),\\
	\frac{\partial ^2}{\partial p_g^2}\Cov_{\underline{p}}(\omega_e,\omega_f)	  &=\frac 2{p_g(1-p_g)^2}\left(1-\frac1{p_g}\phi_{\underline p}[\omega_g]\right)\kappa_{3}^{\underline{p}}(e,f,g) \\
	 &\quad -\frac 2{p_g^2(1-p_g)^2} \textrm{Cov}_{\underline{p}}(\omega_e,\omega_g)\textrm{Cov}_{\underline{p}}(\omega_f,\omega_g),
\end{align*}
where $\kappa_{3}^{\underline{p}}(e,f,g)$ is the third cumulant of the variables $\omega_e,\omega_f,\omega_g$ and is given by
\begin{multline}
	\kappa_{3}^{\underline{p}}(e,f,g):=\phi_{\underline p}[\omega_e\omega_f\omega_g]-\phi_{\underline p}[\omega_e\omega_f]\phi_{\underline p}[\omega_g]-\phi_{\underline p}[\omega_e\omega_g]\phi_{\underline p}[\omega_f]-\phi_{\underline p}[\omega_f\omega_g]\phi_{\underline p}[\omega_e]\\+2\phi_{\underline p}[\omega_e]\phi_{\underline p}[\omega_f]\phi_{\underline p}[\omega_g]
\end{multline}
\end{proposition}

\subsection{Geometric estimates in near-critical environments}
\label{sub:near-critical_geometric_estimates}
In this section, we gather estimates on the averaged sums over $\bbT_N$ of $\Cov_{\underline p}(\scrC(\calR),\omega_e)$, $\kappa_3^{\underline p}(e,f,g)$ and $\Cov_{\underline p}(\omega_e,\omega_g)\Cov_{\underline p}(\omega_f,\omega_g)$ for some environment $\underline p$ which is \emph{near-critical} (in a sense made precise below). The first natural near-criticality condition on the environment is that $\underline p$ is that the measure $\phi_{\bbT_N,\underline p}$ satisfies $\textrm{RSW}(\delta,N)$ for some $\delta>0$. Still, we require slightly more here, namely that arm events and mixing rates do not degenerate too much from the critical measure. Hence, we define the following for $\delta>0$ and $\#\in\{4,3^+, 3^{-}, 3^{\textrm{rg}}, 3^{\textrm{lf}} \}$,
\begin{align*}
    \textrm{Stab}_{\#}(\delta,N)&:=\left\{\phi_{\bbT_N,\underline p,q}\ \Bigg|\ \forall e\in E,\ \forall ~1\le r\le R\le N/4,\quad \frac{\pi^{(e)}_{\#,\underline p}(r,R)}{\pi_{\#,p_c}(r,R)}\in[\delta,1/\delta]\right\},\\
    \textrm{Stab}_{\Delta}(\delta,N) &:= \left\{\phi_{\bbT_N,\underline p,q}\ \Bigg| \ \forall e\in E,\  \forall ~ 1\le r\le R\le N/4,\quad \frac{\Delta_{\underline p}^{(e)}(r,R)}{\Delta_{p_c}(r,R)}\in[\delta,1/\delta]\right\}.
\end{align*}
In the rest of this section, one assumes that
\begin{equation}
\label{eq:all_stabilities}
    \phi_{\bbT_N,\underline p,q}\in \textrm{RSW}(\delta,N)\cap\left(\bigcap_{\#\in\{4,3^+,3^{-},3^{\textrm{rg}},3^{\textrm{lf}}\}}\textrm{Stab}_{\#}(\delta,N)\right)\cap\textrm{Stab}_{\Delta}(\delta,N).
\end{equation}
If the measure $\phi_{\bbT_N,\underline p,q}$ indeed satisfies \eqref{eq:all_stabilities}, then it is straightforward to see that all the bounds in Remark \ref{rem:arm-estimates-transation-invariant-environment} remain valid, with constants only depending on $\delta$, as well as the proofs resulting from them. This reads in the following lemma, which states that all estimates of Section \ref{sub:geometric-estimates} remain valid at $\underline p$.
\begin{lemma}
\label{lem:near-critical_estimates}
Assume $1<q\le 4$, and $\underline p$ satisfies \eqref{eq:all_stabilities}. Then the results of Lemmas \ref{lem:Cov-crossing-edge} and \ref{lem:Cov-arms-edge} still hold when replacing the left hand sides with the respective quantities at $\underline p$, with constants additionally depending on $\delta$. As a corollary, the results of Remarks \ref{rem:sum-Cov-crossing-edge} and \ref{rem:sum-Cov-arms-edge} also hold with the same modifications. Furthermore, if $q=1$ and $\underline p$ satisfies \eqref{eq:all_stabilities}, then the results of Lemma \ref{lem:geometric-estimates_perco} also hold at $\underline p$.
\end{lemma}

Now we move to estimating, when fixing two edges $e,f$,  the sum of cumulants $\kappa_3^{\underline p}(e,f,g)$ and covariances of the form $\Cov_{\underline p}(\omega_e,\omega_g)\Cov_{\underline p}(\omega_f,\omega_g)$ over edges $g\in E_N$ . At criticality, this is done in the proof of \cite[Proposition 5.6]{FK_scaling_relations}. The proof can directly be adapted to the case of a general environment $\underline p$ satisfying $\textrm{RSW}(\delta,N)$ using Proposition \ref{prop:generalisation-thm-scaling-relations} to get that \emph{if $e,f,g$ are ordered such that $d(e,f)$ is the smallest distance between any two of the edges}, then
\begin{align*}
	   |\kappa_3^{\underline p}(e,f,g)|&\lesssim \Delta_{\underline p}^{(e)}(d(e,f))\Delta_{\underline p}^{(e)}(d(e,g))^2\\
	   &\lesssim \Delta_{p_c}(d(e,f))\Delta_{p_c}(d(e,g))^2,
\end{align*}
where passing to the second line we used that $\underline p$ that satisfies \eqref{eq:all_stabilities}. One can sum this bound over edges $g\in E_N$. If $d(e,f)=R$ one has
\begin{align*}
    \sum_{g\in E_N}|\kappa_3^{\underline p}(e,f,g)|&\lesssim\sum_{r\le R}r\Delta_{p_c}(r)\Delta_{p_c}(R)^2+\sum_{R<r\le N}r\Delta_{p_c}(R)\Delta_{p_c}(r)^2\\
    &\lesssim \Delta_{p_c}(R)^2\sum_{r\le N}r\Delta_{p_c}(r)\\
    &\asymp \Delta_{p_c}(R)^2\cdot N^2\Delta_{p_c}(N).
\end{align*}
In the above, we parametrised the position of $g\in E_N$ by its distance $r$ to $\{e,f\}$, which can be rewritten as
\begin{equation}\label{eq:sum_kappa3}
    \sum_{g\in E_N}|\kappa_3^{\underline p}(e,f,g)|\lesssim \Delta_{p_c}(R)^2\textrm W(N)^{-1}.
\end{equation}
An analogous computation gives
\begin{align*}
    \sum_{g\in E_N}|\kappa_3^{\underline p}(e,f,g)|^2&\lesssim\sum_{r\le R}r\Delta_{p_c}(r)^2\Delta_{p_c}(R)^4+\sum_{R<r\le N}r\Delta_{p_c}(R)^2\Delta_{p_c}(r)^4\\
    &\lesssim \Delta_{p_c}(R)^4\sum_{r\le N}r\Delta_{p_c}(r)^2,
\end{align*}
meaning that
\begin{equation}
    \sum_{g\in E_N}|\kappa_3^{\underline p}(e,f,g)|^2\lesssim \Delta_{p_c}(R)^4\cdot\Xi(N)^{-1}.\label{eq:sum_kappa3_squared}
\end{equation}
We now evaluate another sum which is useful in our proofs. Recall that for any edges $e,f$, one has, under the assumption \eqref{eq:all_stabilities}, that $\Cov_{\underline p}(\omega_e,\omega_f)\asymp \Delta_{p_c}(d(e,f))^2$. Then one has
\begin{align*}
	\sum_{g\in E_N}\big|\Cov_{\underline p}(\omega_e,\omega_g)\Cov_{\underline p}(\omega_f,\omega_g)\big|&=\sum_{ d(g,e) \textrm{ or } d(g,f) \leq R }\big|\Cov_{\underline p}(\omega_e,\omega_g)\Cov_{\underline p}(\omega_f,\omega_g)\big| \\
&\quad + \sum_{ d(g,e) \textrm{ and } d(g,f) > R }\big|\Cov_{\underline p}(\omega_e,\omega_g)\Cov_{\underline p}(\omega_f,\omega_g)\big| \\
	&  \lesssim   \sum_{\ell\le R}\ell\Delta_{p_c}(\ell)^2\Delta_{p_c}(R)^2  +  \sum_{R<\ell\le N}\ell \Delta_{p_c}(\ell)^4 \\
    &\lesssim \Delta_{p_c}(R)^2\sum_{r\le N}r\Delta_{p_c}(r)^2.
\end{align*}
This can be summarized as
\begin{equation}\label{eq:sum_cov_g}
    \sum_{g\in E_N}\Cov_{\underline p}(\omega_e,\omega_g)\Cov_{\underline p}(\omega_f,\omega_g)|\lesssim  \Delta_{p_c}(R)^2\cdot \Xi(N)^{-1}.
\end{equation}

\section{Near-critical random bond FK-percolation for naively centered random variables}\label{sec:naive-random-bonds}

\subsection{Gaussian variables centred at $p_c(q)$}\label{sub:centred-gaussian}

\subsubsection{Statement of the result and setup of the proof}
Throughout this section, we fix some $1<q\le 4$, and drop the dependencies in $q$ from our notations. We start by proving Theorem \ref{thm:extension_scaling_window} in the special case of a sequence of i.i.d.\ Gaussian variables whose law is given by 
\begin{equation}\label{eq:Gaussian-case}
    \mathbf{p}_e \overset{(d)}{=} \calN(p_c(q),\sigma_N^2),
\end{equation}
where the parameter $\sigma_N$ goes to $0$ as the size of $\bbT_N$ goes to infinity. Note that the probability that a given $\mathbf p_e$ stays in the interval $[0,1]$ is at most $\lesssim \exp(-c \cdot \sigma_N^{-2} )$ for some universal constant $c>0$. In practice, $\sigma_N$ will be chosen to decay polynomially fast in $N$, implying, by a straightforward union bound, that all the bond parameters $\frak{p}=(\mathbf{p}_e)_{e\in E_N}$ belong to $[0,1]$ with high probability, making the FK measures $\phi_{\bbT_N,\frak{p}}$ well defined. The special role given in our analysis of the case of i.i.d.\ Gaussian variables is related to the so-called Skorokhod embedding theorem, which informally states that any centred random variable with regular enough integrability properties and tails can indeed be realised as some Brownian motion stopped at some well chosen stopping time. As a toy example in our analysis, we present a simplified proof of Theorem \ref{thm:extension_scaling_window} when the random environment $\mathbf{P}_N$ is given by \eqref{eq:Gaussian-case}, as it will be a first incremental step to introduce the core arguments necessary to work both for random variables satisfying either $(\star)^{\text A}_N$ or $(\star)^{\text B}_N$. This reads in the following proposition.
\begin{proposition}\label{prop:L2-case}
    Assume that the random environment $\frak p=(\mathbf{p}_e)_{e\in E_N}$ under $\mathbf{P}_N$ is given by i.i.d\ Gaussians satisfying \eqref{eq:Gaussian-case} for some uniform parameter $\sigma_N$. Then there exists $\delta=\delta(q)>0$ and some positive constants $c=c(q,\delta)$, $O=O(q,\delta)$ such that for any $N\ge 1$ and any $\sigma_N\le c\cdot \textrm W(N)^{1/2}$, one has
    \begin{equation*}
        \mathbf{P}_N\Bigg[ \phi_{\bbT_N,\frak p}\in \textrm{RSW}(\delta,N)\Bigg] \geq 1-O\Bigg(\exp\Bigg[-cN^c\Bigg(\frac{\textrm W(N)^{1/2}}{\sigma_N}\Bigg)^2\Bigg]\Bigg).
    \end{equation*}
\end{proposition}

Before diving into the proof, let us introduce additional notations, and highlight pieces of the strategy we use. As aforementioned, rather than directly looking at the FK measure with bonds parameters given by $(\mathbf{p}_e)_{e\in E_N}$, we introduce a time indexed continuous family of FK measures, starting at the homogeneous critical setup and ending at $(\mathbf{p}_e)_{e\in E_N}$. More precisely, for any $t\geq 0$, define the process $\frak p(t):=(\mathbf{p}_e(t))_{e\in E_N}$, whose components are mutually independent and given by
\begin{equation*}
    \mathbf{p}_e(t) := p_c(q)+B^{(e)}_t,
\end{equation*}
where for each $e\in E_N$, $t\mapsto B^{(e)}_t$ is a standard Brownian motion attached to the edge $e$. One defines the natural (completed) filtration $\mathcal{F}^{N}_t:=\sigma((B^{(e)}_{s})_{0\leq s\leq t}, e \in E_N) $, and by abuse of notation, also write $\mathbf P_N$ for the probability measure associated to the Brownian motions. It is straightforward to see that $\frak p(0)=(p_c)_{e\in E_N}$ and that $t\mapsto \frak p(t)$ is $\mathbf{P}_N$-almost surely continuous with $\frak p(\sigma_N^2)\stackrel{(d)}=\frak p$. To prove Proposition \ref{prop:L2-case}, we will show that for $t\lesssim \textrm W(N)$, all the measures $s\mapsto (\phi_{\bbT_N,\frak p(s),q})_{s\leq t}$ remain within the $\textrm{RSW}(\delta,N)$ class, at least  with high $\mathbf{P}_N$-probability. For the sake of notations, we will couple the evolving environment $\frak p(t)$ with $\frak p$ so that $\frak p(\sigma_N^2)=\frak p$, that is to say for each $e$, one has $\mathbf p_e(\sigma_N^2)=\mathbf p_e$. Our strategy enhances the ideas present in the original proof of Theorem \ref{thm:scaling-relations_second} to our random context. More practically, this means keeping track along the time deformation of the stability of the RSW box crossing property, the mixing rates and the arm exponents \emph{everywhere in} $E_N$. Some circular argument appears, as the crossing probabilities in rectangles and the arm exponents control the time dependent stochastic derivatives of the mixing rates, which control in return the time derivatives of the crossing probabilities and arm exponents. The overall interplay between those quantities will allow the FK measures to stay resembling the critical one.

Set two positive parameters $\delta,\rho$ whose value will be fixed through the proof, and will in fact only depend on $1<q\leq 4$. In particular, all the constants $\lesssim,\asymp$ will depend depend on $\delta$ and $\rho$ (thus on $q$), except within the proof of Lemma \ref{lem:stability-StabDelta} where the dependency will be explicited. In order to quantify the fact that the FK model with parameters $\frak{p}(t) $ deviates too much from the critical one, we introduce some \emph{breaking time} $T_b(\delta,\rho)$, where the model no longer exhibits critical behavior, in the spirit of the deterministic breaking point $p_b$ introduced in the proof of Theorem \ref{thm:scaling-relations_second}. More precisely, the random variable $T_b(\delta,\rho)$ is the stopping time for the filtration $\mathcal{F}^{N}_t $ corresponding to the first instant that $\frak p(t)$ no longer satisfies \emph{at least one} of the stability conditions discussed in Section \ref{sub:near-critical_geometric_estimates}, i.e.\
\begin{equation*}
    T_b=T_b(\delta,\rho):=\inf\left\{t\ge 0 \Bigg| \phi_{\bbT_N,\frak p(t)}\not\in \textrm{RSW}(\delta,N)\cap\left(\bigcap_{\#}\textrm{Stab}_{\#}(\delta,N)\right)\cap\textrm{Stab}_{\Delta}(\rho,N)\right\}
\end{equation*}
where the indices  $\#$ run through $\{4,3^+,3^-,3^{\textrm{rg}},3^{\textrm{lf}}\}$. Note that the $\Delta$ (mixing rates) stability involves some constant $\rho$ instead of $\delta$, which is a priori different, and will be fixed during the proofs. Requiring the strong box-crossing property at each scale $\textrm{RSW}(\delta,N)$ enforces (for a box of size $1$) that all the edges in $\bbT_N$ are open with probability between $\delta$ and $1-\delta$. As a consequence, for any $0\leq t\leq T_b(\delta,\rho)$, all of the bond parameters  $\mathbf{p}_e(t) $ remain bounded away from $0$ and $1$, making sense of the measure $\phi_{\bbT_N,\frak p(t)}$ up to time $T_b(\delta,\rho)$. Using the very definition of $T_b(\delta,\rho)$, one can see Proposition \ref{prop:L2-case} as a consequence of the fact that, with high $\mathbf{P}_N$ probability, the breaking time $T_b(\delta,\rho)$ is not too small. In fact, it will suffice to show that there exist constants $\delta,\rho,c,\lesssim$ such that for all $\sigma_N\le c\cdot \textrm W(N)^{1/2}$,
\begin{equation}
\label{eq:rewrite_L2-case}
    \mathbf P\Big[T_b(\delta,\rho)<\sigma_N^2\Big]\lesssim \exp\Bigg[-cN^c\left(\frac{\textrm W(N)^{1/2}}{\sigma_N}\right)^2\Bigg].
\end{equation}
Once Proposition \ref{prop:L2-case} is related to the probability that $T_b$ is not too small, the proof goes via proving that probabilities of crossing events of the form $\phi^0_{\calR',\frak p(t),q}[\scrC_{\bullet}(\calR)]$ (similar to those introduced in Definition \ref{def:RSW-Torus}), mixing rates, and probabilities of arm events, do not deviate too much from their original value critical value $\frak{p}(0)$. One (small) subtlety when applying standard FK percolation arguments relating different scales is the influence of boundary conditions on mixing (notably those on $\calR'$ when considering the crossings of $\calR$). In order to sidestep this problem, one can consider some slightly more evolved crossing events, in the spirit of the properties involved in Definition $\textrm{RSW}(\delta,N)$
\begin{definition}
\label{def:modified_RSW}
Let $\calR$ be a $2$ by $1$ rectangle of width $R\le N/8$ in $\bbT_N$, and let $\calR'$ be the rectangle with the same center as $\calR$ but twice larger side lengths. Define $\scrC(\calR'\backslash \calR)$ (respectively $\scrC^{\star}(\calR'\backslash \calR)$) to be the event that there exists a primal (respectively a dual) circuit inside the annulus $\calR'\backslash \calR$ and separating the outer boundaries of $\calR$ and $\calR'$. We can then define the modified crossing events by setting, for $\bullet\in\{h,v\}$ (see Figure \ref{fig:modified_crossing} for an illustration)
\begin{equation*}
\begin{cases}
    \scrC_{\bullet}^0(\calR)&:=\scrC_{\bullet}(\calR)\cap \scrC^{\star}(\calR'\backslash \calR), \\
    \scrC_{\bullet}^{\star1}(\calR)&:=\scrC_{\bullet}^{\star}(\calR)\cap \scrC(\calR'\backslash \calR). \\
\end{cases}
\end{equation*}
\end{definition}

\begin{figure}
    \centering
    \includegraphics{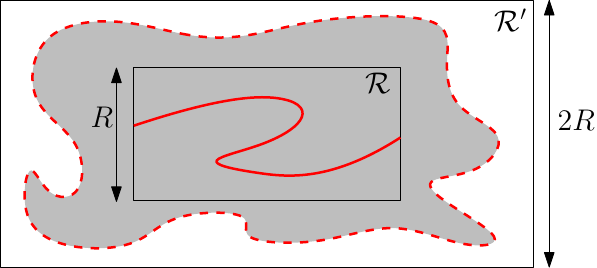}
    \caption{Illustration of the event $\scrC^0_{h}(\calR)$. Primal paths are depicted in full red lines and dual paths are in dashed red lines. The greyed out region is the domain $\calD$ introduced after Definition \ref{def:modified_RSW}.}
    \label{fig:modified_crossing}
\end{figure}
Standard conditioning and monotonicity arguments ensure that, for any $\frak p$ and any $\bullet \in \{ h,v \}$ one has
\begin{equation*}
    \phi_{\bbT_N,\frak p}[\scrC^0_{\bullet}(\calR)]\le\phi_{\bbT_N,\frak p}[\scrC_{\bullet}(\calR)|\scrC^{\star}(\calR'\backslash \calR)].
\end{equation*}
Conditionally on the event $\scrC^{\star}(\calR'\backslash \calR)$, one can explore the FK clusters from $\partial \calR'$ inwards, until finding a dual path circulating the annulus $\calR'\backslash \calR$. A piece of this path can be seen as the boundary of some domain simply connected domain $\calD$, which hasn't been explored and contains $\calR$ as a sub-domain (see Figure \ref{fig:modified_crossing}). The dual-connected boundary corresponds to \emph{free} boundary conditions on $\calD$, which implies via the domain Markov property that
\begin{equation*}
    \phi_{\bbT_N,\frak p}[\scrC_{\bullet}(\calR)|\scrC^{\star}(\calR'\backslash \calR)]=\phi_{\bbT_N,\frak p}\Big[\phi_{\calD,\frak p}^0[\scrC_{\bullet}(\calR)]\Big|\scrC^{\star}(\calR'\backslash \calR)\Big].
\end{equation*}
As the event $\scrC_{\bullet}(\calR)$ is increasing, domain monotonicity ensures that for all domains $\calD$ contained in $\calR'$, one has $\phi_{\calD,\frak p}^0[\scrC_{\bullet}(\calR)]\le \phi^0_{\calR',\frak p}[\scrC_{\bullet}(\calR)]$. Putting everything together gives
\begin{equation*}
    \phi_{\bbT_N,\frak p}[\scrC_{\bullet}^0(\calR)]\le \phi^0_{\calR',\frak p}[\scrC_{\bullet}(\calR)].
\end{equation*}
A similar analysis for the dual model ensures that for any $\frak p$ and $\bullet \in \{ h,v\}$ one has
\begin{equation*}
    \phi_{\bbT_N,\frak p}[\scrC_{\bullet}^{\star 1}(\calR)]\le \phi^1_{\calR',\frak p}[\scrC^{\star}_{\bullet}(\calR)].
\end{equation*}
Therefore, if one wants to check the $\textrm{RSW}(\delta,N)$ property on $\bbT_N$, it is enough to check that the both  probabilities of the events $\scrC_{\bullet}^0(\calR)$ and $\scrC_{\bullet}^{\star 1}(\calR)$ remain are larger (or equal to) than $\delta$. In particular, one can restrict the analysis to events for the full torus measure $\phi_{\bbT_N,\frak p}$, not involving boundary conditions in sub-domains. Let us add that, conversely, standard RSW theory ensures that provided the measure $\phi_{\bbT_N,\frak p}$ satisfies $\textrm{RSW}(\delta,N)$, then both lower bounds $\phi_{\bbT_N,\frak p}[\scrC_{\bullet}^0(\calR)] \geq \delta'$ and $\phi_{\bbT_N,\frak p}[\scrC_{\bullet}^{\star 1}(\calR)] \geq \delta'$ hold for some $\delta'>0$ only depending on $\delta$.

\subsubsection{Proving stability of the FK measures below $T_b $}

One can now pass to the core of the proof, based on the following two lemmas, that evaluate how crossing probabilities, arm events and mixing rates evolve when $\frak{p}(t)$ moves from the uniform measure $\frak{p}(0)$. All the crossing event derivatives with respect to some specific edge-weight $p_e$ are recalled in Section \ref{sub:derivative-events}. Hence, it remains to apply the Itô's formula and estimate the order of magnitude of each term involved. In order to obtain $\delta$-dependent absolute constants, one should enforce in the reasonings presented below the possibility to chose constants only depending on $\delta$ (and in particular $\rho=\rho(\delta)$) provided $\delta$ is small enough. In the proof of lemma \ref{lem:stability-StabDelta}, we explicitly adopt the $\asymp_{\rho,\delta},\lesssim_{\rho,\delta}$ convention (and similarly when constants only depend on $\delta$ or $\rho$) to make the logic more transparent.
\begin{lemma}
\label{lem:stability-RSW}
    Assume that $\delta$ is chosen small enough, only depending on $1< q\leq 4$. Then for all $\rho>0$, there are some positive constants $c=c(\delta,\rho,q)>0$ and $O_{\delta,\rho}(\cdot)$ such that if $\sigma_N^2\le c\cdot \textrm W(N)$, one has
\begin{equation*}
    \mathbf P_N \Bigg[ \forall \textrm{ } 0\leq  s\leq \Big( \sigma_N^2\wedge T_b(\delta,\rho)\Big), \textrm{ } \phi_{\bbT_N,\frak p(s)}\in \textrm{RSW}(2\delta,N) \Bigg]\geq 1-O_{\delta,\rho}\Bigg( \exp\Bigg[-c N^{c}\frac{\textrm W(N)}{\sigma_N^2} \Bigg] \Bigg).
\end{equation*}
Moreover, the same statement holds when replacing $\textrm{RSW}(2\delta,N)$ by $\textrm{Stab}_{\#}(2\delta,N)$ for any $\#\in\{4,3^+,3^-,3^{\textrm{rg}},3^{\textrm{lf}}\}$.
\end{lemma}

\begin{lemma}
\label{lem:stability-StabDelta}
Assume that $\delta>0$ is arbitrary. Then there exists a constant $\rho=\rho(\delta)>0$ small enough and constants $c=c(\delta,q)>0$ and $O_{\delta}(\cdot)$ such that if $\sigma_N^2\leq c\cdot \textrm W(N)$, one has
 \begin{equation*}
     \mathbf P_N \Bigg[ \forall ~~ 0\leq  s\leq \Big( \sigma_N^2\wedge T_b(\delta,\rho)\Big), \phi_{\bbT_N,\frak p(s)}\in \textrm{Stab}_{\Delta}(2\rho,N) \Bigg] \geq 1-O_{\delta}\Bigg( \exp\Bigg[-c N^{c}\frac{\textrm W(N)}{\sigma_N^2} \Bigg] \Bigg).
 \end{equation*}
\end{lemma}
Let us comment on the asymmetry of the statements in these two lemmas. The stability event $\textrm{Stab}_{\Delta}(\rho,N)$ estimates the mixing rates $\Delta_{\frak p(t)}^{(e)}(R)$, while those mixing rates are only used in computations up to multiplicative constants coming from the $\textrm{RSW}(\delta,N)$ bounds. Therefore, picking a smaller $\delta$ in the $\textrm{RSW}(\delta,N)$ bounds worsens the bounds the mixing rates, which implies that one should in return change $\rho$ in order show that the FK measures still belong to the $\textrm{Stab}_{\Delta}(\rho,N)$ class. To conclude the proof of Proposition \ref{prop:L2-case}, one needs to make a joint choice of $\delta$ and $\rho$. One first picks $\delta$ according to Lemma \ref{lem:stability-RSW}, and then selects $\rho$ according to Lemma \ref{lem:stability-StabDelta} according to $\delta$. We now pass to the proof of the  aforementioned lemmas, starting with Lemma \ref{lem:stability-RSW}, using the Itô formula to bound stochastic derivatives in the random environment.

\begin{proof}[Proof of Lemma \ref{lem:stability-RSW}]
Let $\mathcal{S}$ be any event defined on $\bbT_N$. Since the standard Brownian motions $\big( (B^{(e)}_t)_{t\geq 0}\big)_{e\in E_N}$ are i.i.d\,, one can apply the Itô lemma which reads as 
\begin{equation*}
    \dd\phi_{\bbT_N,\frak p(t)}[\mathcal{S}]= \sum_{e\in E_N}\left(\frac{\partial}{\partial p_e}\bigg\rvert_{\underline p = \frak p(t)}\phi_{\bbT_N,\underline p}[\mathcal{S}]\right)\dd B^{(e)}_t + \frac12\Bigg(\sum_{e\in E_N}\left(\frac{\partial^2}{\partial p_e^2}\bigg\rvert_{\underline p = \frak p(t)}\phi_{\bbT_N,\underline p}[\mathcal{S}]\right)\Bigg)\dd t,
\end{equation*}
where the $\dd$ operator corresponds to the standard stochastic differentiation. Plugging the explicit formulae obtained in Proposition \ref{prop:differentiate-bonds} one gets
\begin{align*}
	 \dd\phi_{\bbT_N,\frak p(t)}[\mathcal{S}]&=\sum_{e\in E_N}\frac 1{\mathbf{p}_e(t)(1-\mathbf{p}_e(t))}\Cov_{\frak p(t)}(\mathcal{S},\omega_e)\dd B^{(e)}_t \\
	 &\quad + \sum_{e\in E_N}\frac1{\mathbf{p}_e(t)(1-\mathbf{p}_e(t))^2}\left(1-\frac 1{\mathbf{p}_e(t)}\phi_{\bbT_N,\frak p(t)}[\omega_e]\right)\Cov_{\frak p(t)}(\mathcal{S},\omega_e)\dd t.
\end{align*}
We are going to evaluate the order of magnitude of the two previous lines for times $0\leq t \leq T_b(\delta,\rho)$ for the special case of the events 
$\mathcal{S}=\scrC^0_{\bullet}(\calR)$ or $\mathcal{S}=\scrC^{\star 1}_{\bullet}(\calR)$. One of the takeaway points of \cite{FK_scaling_relations} is the possibility, at and near criticality, to express event covariances and edge-influences only using mixing rates $\Delta$. Below the breaking time $T_b(\delta,\rho)$, all the mixing rates remain comparable up to constant to the critical ones $\Delta_{p_c}$, whose analysis in \cite{FK_scaling_relations} was recalled in Section \ref{sub:near-critical_geometric_estimates}. We only work out the details when $\mathcal{S}=\scrC^0_{\bullet}(\calR)$, the case $\mathcal{S}=\scrC^{\star 1}_{\bullet}(\calR)$ can be treated similarly. The cases $\bullet=h$ and $\bullet=v$ are exactly identical so we just write $\bullet$ for the proof.
\paragraph{Step 0: Canonical decomposition of the semimartingale.} Fix a $2$ by $1$ rectangle $\calR\in \bbT_N$ of width $R\leq 8$. Using the classical semimartingale decomposition, one can write $\phi_{\bbT_N,\frak p(t)}[\scrC^0_{\bullet}(\calR)]=M^\calR_t+A^\calR_t$ where $M^\calR$ is a local martingale started at $0$ and $A^\calR$ is an adapted finite variation process, whose exact expression are given by
\begin{align*}
	\dd M^{\calR}_t&=\sum_{e\in E_N}\frac 1{\mathbf{p}_e(t)(1-\mathbf{p}_e(t))}\Cov_{\frak p(t)}(\scrC^0_{\bullet}(\calR),\omega_e)\dd B^{(e)}_t, \\
	\dd A^\calR_t &= \sum_{e\in E_N}\frac1{\mathbf{p}_e(t)(1-\mathbf{p}_e(t))^2}\left(1-\frac 1{\mathbf{p}_e(t)}\phi_{\bbT_N,\frak p(t)}[\omega_e]\right)\Cov_{\frak p(t)}(\scrC^0_{\bullet}(\calR),\omega_e)\dd t.
\end{align*}

\paragraph{Step 1: Evaluating the order of magnitude of $A^\calR_t$.}
We are going to first evaluate $A^\calR_t$ using the critical mixing rates $\Delta_{p_c} $. Recall that provided $0\leq t \leq T_b(\delta,\rho)$, the $\textrm{RSW}(\delta,N)$ property applied to rectangles of size $2\times 1$ and $1\times 2$ ensures that coupling constants $\mathbf{p}_e(t)$ are bounded away from $0$ and $1$. This ensures in particular that there exists a constant $O(1)$, only depending on $\delta$ chosen small enough such that
\begin{equation}\label{eq:bound-away-0-1}
	\frac1{\mathbf{p}_e(t)(1-\mathbf{p}_e(t))^2} =O_\delta(1), \quad  \left(1-\frac 1{\mathbf{p}_e(t)}\phi_{\bbT_N,\frak p(t),q}[\omega_e]\right) =O_\delta(1).
\end{equation}
One can therefore write the crude estimate
\begin{align*}
	\Bigg|\frac{\dd A^\calR_t}{\dd t}\Bigg|&= \Bigg| \sum_{e\in E_N}\frac1{\mathbf{p}_e(t)(1-\mathbf{p}_e(t))^2}\left(1-\frac 1{\mathbf{p}_e(t)}\phi_{\bbT_N,\frak p(t),q}[\omega_e]\right)\Cov_{\frak p(t)}(\scrC^0_{\bullet}(\calR),\omega_e) \Bigg|\\
	&\lesssim_\delta \sum_{e\in E_N}\Cov_{\frak p(t)}(\scrC^0_{\bullet}(\calR),\omega_e).
\end{align*}
Still for $0\leq t \leq T_b(\delta,\rho)$, the measure $\phi_{\bbT_N,\frak p(t),q}$ belongs to all of the required stability events from Section \ref{sub:near-critical_geometric_estimates}. Therefore, one may use equation (\ref{eq:sum_cov_crossing_edge}), which holds due to Lemma \ref{lem:near-critical_estimates} to get
\begin{equation}\label{eq:covariance-growth-1}
  \sum_{e\in E_N}\Cov_{\frak p(t)}(\scrC^0_{\bullet}(\calR),\omega_e)\lesssim_{\delta,\rho} \textrm W(N)^{-1}.
\end{equation}
Note that \eqref{eq:sum_cov_crossing_edge} is in principle written for the event $\scrC_{\bullet}(\calR)$ and not $\scrC^0_{\bullet}(\calR)$. Still, the exact same proof holds for this modified crossing event, except one has to take care of both the corners of $\calR$ and $\calR'$. Integrating in time $t$ shows that $0\leq t\leq T_b(\delta,\rho)$ one has
\begin{equation}\label{eq:estimate-growth-A}
    |A^\calR_t-A^\calR_0|\lesssim_{\delta,\rho} \frac t{\textrm W(N)}.
\end{equation}

\paragraph{Step 2: Evaluating the order of magnitude of $M^\calR_t$.}
We now evaluate the the rate of growth of local martingale part $M^\calR_t$. Itô's formula gives the following expression for its bracket:
\begin{equation*}
    \dd \langle M^\calR\rangle_t = \sum_{e\in E_N}\left(\frac 1{\mathbf{p}_e(t)(1-\mathbf{p}_e(t))}\Cov_{\frak p(t)}(\scrC^0_{\bullet}(\calR),\omega_e)^2\right)\dd t.
\end{equation*}
Using again \eqref{eq:bound-away-0-1}, one deduces that
\begin{equation}\label{eq:bracket-M-estimate}
    \frac{\dd \langle M^\calR\rangle_t}{\dd t}\lesssim _\delta \sum_{e\in E_N}\Cov_{\frak p(t)}(\scrC^0_{\bullet}(\calR),\omega_e)^2.
\end{equation}
Once again, for $0\leq t\le T_b(\delta,\rho)$, there exist some $c=c(q)>0$ such that 
\begin{equation*}
     \frac{\dd \langle M^\calR\rangle_t}{\dd t}\lesssim_{\delta,\rho} \Xi(N)^{-1}\lesssim_{\delta,\rho} N^{-c}\textrm W(N)^{-1},
\end{equation*}
where the first inequality comes from \eqref{eq:sum_cov_crossing_edge_squared} and the second one from Lemma \ref{lem:window-inequalities}. Integrating this equation up to some time $0\leq t \leq T_b(\delta,\rho)$ gives
\begin{equation}\label{eq:rate-growth-M}
 \langle M^\calR\rangle_t\lesssim_{\delta,\rho} N^{-c}\textrm W(N)^{-1}\cdot t.
\end{equation}

\paragraph{Step 3: Concluding via large deviations for $M^\calR_t$.} 
We are now in position to conclude. Informally speaking, if there is an instant $t \ll \sigma_N^2$ such that ${\phi_{\bbT_N,\frak{p}(t)}\not\in \textrm{RSW}(2\delta,N)}$, then there is at least one $2$ by $1$ rectangle $\calR\subset \bbT_N$ such that the local martingale $M^{\calR}_t$ achieved a much greater value than what its bracket should allow. Let us clarify the reasoning, highlighting where it becomes crucial to have $\delta$ small enough.
\begin{enumerate}[label=(\arabic*)]
    \item First pick $\delta>0$ small enough and uniform in $N$ such that $\phi_{\bbT_N,p_c}[\scrC^0_{\bullet}(\calR)]\ge4\delta$ for all $2$ by $1$ rectangles $\calR\subset \bbT_N$. Crucially, this $\delta$ only depends on $q$.
     \item Using \eqref{eq:estimate-growth-A}, one gets that there is a constant $C=C(\delta,\rho)>0$ such that for any $2$ by $1$ rectangle $\calR\subset  \bbT_N$ and any $0\leq t\leq \sigma_N^2\wedge T_{b}(\delta,\rho)$ one has 
    \begin{equation*}
    	|A^\calR_{t}-A^\calR_0| \le \frac{C\cdot t}{\textrm W(N)} \le \frac{C\cdot \sigma_N^2}{\textrm W(N)}
    \end{equation*}
Fixing $c_1:=\sqrt{\frac{\delta}{C}}$, this implies that for $\sigma_N\le c_1\cdot \textrm W(N)^{1/2}$, one has (deterministically)  
\begin{equation*}
        |A^\calR_{t}-A^\calR_0|\le \delta
    \end{equation*}
    \item Assuming that the event $\Big\{ \exists \textrm{ } 0\leq  s\leq \Big( \sigma_N^2\wedge T_b(\delta,\rho)\Big), \textrm{ } \phi_{\bbT_N,\frak p(s)}\not\in  \textrm{RSW}(2\delta,N) \Big\} $ holds, there exists (if $\sigma_N\le c_1\cdot \textrm W(N)^{1/2}$) at least one $2$ by $1$ rectangle $\calR\subset \bbT_N$ such that 
\begin{equation}
	\sup_{0\leq s\le \sigma_N^2\wedge T_b(\delta,\rho)} \Big|\phi_{\bbT_N,\frak p(s)}[\scrC^0_{\bullet}(\calR)]-\phi_{\bbT_N,p_c}[\scrC^0_{\bullet}(\calR)]\Big|>2\delta.
\end{equation}
Combining with the previous inequalities for that specific rectangle $\calR$ implies that provided  $\sigma_N\le c_1\cdot \textrm W(N)^{1/2}$ one has
\begin{equation*}
    \sup_{0\leq s\le \sigma_N^2\wedge T_b(\delta,\rho)}\Big|M^\calR_s\Big|>\delta
\end{equation*}
\item One can now use standard large deviation estimates for local martingales. Using \eqref{eq:rate-growth-M}, there exists some constant $c=c(q,\delta)>0$ such that for any $2$ by $1$ rectangle $\calR \subset \bbT_N$ and any  $0\leq t\leq \sigma_N^2\wedge T_b(\delta,\rho)$ one has
\begin{equation}\label{eq:rate-growth-M-2}
    \langle M^\calR\rangle_t \leq c^{-1}N^{-c}\frac{t}{\textrm W(N)} \leq c^{-1}N^{-c}\frac{\sigma_N^2}{\textrm W(N)}.
\end{equation}
Using standard large deviation principles, one can control the supremum and the infimum of a local martingale started at $0$ via its bracket, which reads (choosing if needed $c$ smaller than its value in \eqref{eq:rate-growth-M-2}) as
\begin{equation*}
    \mathbf P_N\left[\sup_{0\leq s\le \sigma_N^2 \wedge T_b(\delta,\rho)} |M^\calR_s|>\delta\right]\lesssim_{\delta,\rho} \exp\Bigg[ -cN^{c} \frac{\textrm W(N)}{\sigma_N^2}  \Bigg].\end{equation*}
Applying the previous remark to each of the $O(N^{3})$ $2$ by $1$ rectangles inside $\bbT_N$, a union bound ensures that
\begin{equation*}
	\mathbf P_N\left[\inf_{\calR \subset \bbT_N}
    \inf_{\bullet=h,v}\inf_{0\leq s\leq \sigma_N^2\wedge T_b(\delta,\rho)} \phi_{\bbT_N,\frak p(s)}[\scrC^{
    0}_{\bullet}(\calR)]<2\delta\right]\lesssim_{\delta,\rho}N^3 \exp\Bigg[ -cN^{c} \frac{\textrm W(N)}{\sigma_N^2}  \Bigg].
\end{equation*}
The complement of the event inside $\mathbf{P}_N$ in the above equation is the event that for any $0\leq s\le \sigma_N^2\wedge T_b(\delta,\rho)$ and any $2$ by $1$ rectangle $\calR\subset \bbT_N$, the probability of any primal crossing $\phi_{\bbT_N,\frak p(s)}[\scrC^{0}_{\bullet}(\calR)]$ is at least $2\delta$. Doing the same reasoning for dual crossings and replacing $c$ by $\frac{c}{2}$ concludes the proof.
\end{enumerate}

\paragraph{Adapting the proof for arm events.}
When proving the stability of the arm events $\mathcal A^{(e)}_{\#}(R)$, one can run almost verbatim the proof made for the crossing event $\scrC^0_{\bullet}(\calR)$, replacing \eqref{eq:sum_cov_crossing_edge} and \eqref{eq:sum_cov_crossing_edge_squared} by respectively
\eqref{eq:sum_cov_arm} and \eqref{eq:sum_cov_arm_squared}. In this case, the probabilities of arm events $\phi_{\bbT_N,\frak p(t)}[\mathcal A^{(e)}_{\#}(R)]$ are not bounded away from $0$ and $1$ uniformly in $N$, so one needs to keep track of the additional pre-factors $\pi_{\#,p_c}(R)$ that appear in front of the stochastic derivatives, exactly as in Remark \ref{rem:sum-Cov-arms-edge}. 
\end{proof}
We prove now Lemma \ref{lem:stability-StabDelta} with very similar methods, using a slightly more delicate strategy due when controlling the covariances between multiple edges. 
\begin{proof}[Proof of lemma \ref{lem:stability-StabDelta}]
We proceed here similarly to the proof of Theorem \ref{thm:scaling-relations_second}, where the stability of the mixing rates $\Delta_{\frak p(t)}$ is derived from the stability of the edge-covariances of the form $\Cov_{\frak p(t)}(\omega_e,\omega_f)$. A priori, as long as $0\leq t\leq T_{b}(\delta,\rho)$, the covariances  $\Cov_{\frak p(t)}(\omega_e,\omega_f)$ are uniformly comparable with  $\Cov_{p_c}(\omega_e,\omega_f)$, \emph{but with some constants depending on $\rho$ and $\delta$}. The goal of this proof is to carefully underline the fact that $\rho$ is chosen only depending on $\delta$.
\paragraph{Step 0: Canonical decomposition of the semimartingale}
Fix two edges $e,f\in E_N$ at distance $R$ from each other. One can use the identities of Proposition \ref{prop:differentiate-bonds} together with the Ito formula to write $ \Cov_{\frak p(t)}(\omega_e,\omega_f)=M^{e,f}_t+A^{e,f}_t$ where 
 $M^{e,f}$ is a local martingale started at $0$ and $A^{e,f}$ is a bounded variation process adapted to filtration $\mathcal{F}^{N}_s$, whose respective expressions are given by
 \begin{align*}
 	\dd M_t^{e,f} &= \sum_{g\in E_N}\Bigg(\frac{\partial}{\partial p_g}\bigg\rvert_{\underline p = \frak p(t)}\Cov_{\underline p}(\omega_e,\omega_f) \Bigg)\dd B^{(g)}_t \\
 	&= \sum_{g\in E_N}\Bigg( \frac{1}{\textbf p_g(t)(1-\textbf p_g(t))}\kappa^{\frak p(t)}_3(e,f,g) \Bigg)\dd B^{(g)}_t, \\
 	\dd A_t^{e,f}&= \frac{1}{2} \sum_{g\in E_N}\Bigg( \frac{\partial^2}{\partial p_g^2}\bigg\rvert_{\underline p = \frak p(t)}\Cov_{\underline p}(\omega_e,\omega_f) \Bigg) \dd t\\
&=\sum_{g\in E_N}\frac{1}{\textbf p_g(t)(1-\textbf p_g(t))^2}\left(1-\frac 1{\textbf p_g(t)}\phi_{\bbT_N,\frak p(t)}[\omega_g]\right)\kappa_3^{\frak p(t)}(e,f,g) \dd t\\
&\quad - \sum_{g\in E_N}\frac 1{\textbf p_g(t)^2(1-\textbf p_g(t))^2}\Cov_{\frak p(t)}(\omega_e,\omega_g)\Cov_{\frak p(t)}(\omega_f,\omega_g) \dd t.
 \end{align*}

\paragraph{Step 1: Evaluating the order of magnitude of $A^{e,f}_t$.} Using again \eqref{eq:bound-away-0-1} (which states that all the bonds $\mathbf{p}_e(t) $ are bounded away from $0$ and $1$) one has for $0\leq t \leq T_{b}(\delta,\rho)$, the triangular inequality provides the upper bound
\begin{equation}\label{eq:bound-A-via-kappa-3}
    \left|\frac{\dd A^{e,f}_t}{\dd t}\right|\lesssim_{\delta
    } \sum_{g\in E_N}|\kappa_3^{\frak p(t)}(e,f,g)| + \sum_{g\in E_N}\Cov_{\frak p(t)}(\omega_e,\omega_g)\Cov_{\frak p(t)}(\omega_f,\omega_g).
\end{equation}

\paragraph{Step 1.1: Estimating the $\kappa_3^{\frak{p}(t)}$ terms in \eqref{eq:bound-A-via-kappa-3}.} As in the proof of Lemma \ref{lem:stability-RSW}, for $0\leq t\le T_b(\delta,\rho)$, we may use the results of Section \ref{sub:near-critical_geometric_estimates}, with constants depending on both $\delta$ and $\rho$. Using equation (\ref{eq:sum_kappa3}) and recalling that $d(e,f)=R $, one has 
\begin{equation*}
    \sum_{g\in E_N}|\kappa_3^{\frak p(t)}(e,f,g)|\lesssim_{\delta,\rho} \Delta_{p_c}(R)^2\cdot\textrm W(N)^{-1}.
\end{equation*}

\paragraph{Step 1.2: Estimating the $\textrm{Cov}_{\frak{p}(t)}$ terms in \eqref{eq:bound-A-via-kappa-3}.} Let us now estimate the terms of involved in the $ \Cov_{\frak p(t)}(\omega_e,\omega_g)\Cov_{\frak p(t)}(\omega_f,\omega_g)$ sum. This is simply the computation made in Section \ref{sub:near-critical_geometric_estimates} and equation \eqref{eq:sum_cov_g} gives
\begin{equation*}
	\sum_{g\in E_N}|\Cov_{\frak p(t)}(\omega_e,\omega_g)\Cov_{\frak p(t)}(\omega_f,\omega_g)|\lesssim_{\delta,\rho} \Delta_{p_c}(R)^2\cdot\Xi(N)^{-1}
\end{equation*}
\paragraph{Step 1.3: Concluding on the estimate of $\frac{\dd A^{e,f}_t}{\dd t}$.}
Due to Lemma \ref{lem:window-inequalities}, the term given by the sum of products of covariances is polynomially smaller than the estimate for the sum of the $\kappa_3^{\frak p(t)}$ terms. Therefore, one may write
\begin{equation*}
	 \left|\frac{\dd A^{e,f}_t}{\dd t}\right|\lesssim_{\delta,\rho
    } \Delta_{p_c}(R)^2\cdot\textrm W(N)^{-1}.
\end{equation*}

\paragraph{Step 2: Evaluating the order of magnitude of $M^{e,f}_t$.}
We can again evaluate the order of magnitude of the local martingale $M^{e,f}_t$ via its bracket $\langle M^{e,f} \rangle_{t} $. As long as $0\leq t\le \sigma_N^2 \wedge T_b(\delta,\rho)$, one can \eqref{eq:sum_kappa3_squared} from Section \ref{sub:near-critical_geometric_estimates} and deduce that
\begin{align*}
	\Bigg|\frac{\dd \langle M^{e,f}\rangle_t}{\dd t}\Bigg|&\lesssim_{\delta}\sum_{g\in E_N}\Big(\kappa_3^{\frak p(t)}(e,f,g)\Big)^2\\
	 & \lesssim_{\delta,\rho} \Delta_{p_c}(R)^4\cdot\Xi(N)^{-1}.
\end{align*}
Finally, Lemma \ref{lem:window-inequalities} ensures the existence of $c=c(q)>0$ such that
\begin{equation*}
    \forall \ 0\le t\le \sigma_N^2\wedge T_b(\delta,\rho),\quad \Bigg|\frac{\dd \langle M^{e,f}\rangle_t}{\dd t}\Bigg|\lesssim_{\delta,\rho} \Delta_{p_c}(R)^4N^{-c}\cdot\textrm W(N)^{-1}.
\end{equation*}

\paragraph{Step 3: Proving that covariances remain stable via large deviations of $M^{e,f}_t$.}
One can now repeat exactly the large deviation arguments within Step 3 the proof of Lemma \ref{lem:stability-RSW} (with additional $\Delta_{p_c}(R)$ pre-factors) and deduce that for any $\varepsilon>0 $, there exist some constants $c_1=c_1(\delta,\rho,\varepsilon,q)>0$ and $c_2=c_2(\delta,\rho,\varepsilon,q)>0$ such that as long $\sigma_N^2\le c_2\cdot\textrm W(N)$ one has
\begin{multline*}
	\mathbf{P}_{N}\Bigg[ \exists\ e,f\in E_N, \exists\ 0\leq s \leq \sigma_N^2 \wedge T_b(\delta,\rho), \frac{ \Big|\Cov_{\frak p(s)}(\omega_e,\omega_f)-\Cov_{p_c}(\omega_e,\omega_f)\Big|}{\Delta_{p_c}(d(e,f))^2}>\eps \Bigg]\\\lesssim_{\delta}\exp\Bigg[ -c_1 N^{c_1} \frac{\textrm W(N)}{\sigma_N^2}  \Bigg].
\end{multline*}
For the critical measure, one has $\Cov_{p_c}(\omega_e,\omega_f)\asymp \Delta_{p_c}(d(e,f))^2$ uniformly in $e,f\in E_N$. Therefore, there exists a constant $C=C(q)>0$ such that
\begin{multline*}
	\mathbf{P}_{N}\Bigg[ \exists\ e,f\in E_N, \exists\ 0\leq s \leq \sigma_N^2 \wedge T_b(\delta,\rho), \Bigg| \frac{ \Cov_{\frak p(s)}(\omega_e,\omega_f)}{\Cov_{p_c}(\omega_e,\omega_f)}-1\Bigg|>C\eps \Bigg]\\\lesssim_{\delta,\rho} \exp\Bigg[ -c_1 N^{c_1} \frac{\textrm W(N)}{\sigma_N^2}  \Bigg].
\end{multline*}
One can now fix $\eps=\eps(q)$ smaller than $\frac 1{2C}$ such that the event 
\begin{equation}\label{eq:event-proof-L2}
	\Bigg\{ \forall\ e,f\in E_N, \forall\ 0\leq s \leq \sigma_N^2 \wedge T_b(\delta,\rho), \Bigg| \frac{ \Cov_{\frak p(s)}(\omega_e,\omega_f)}{\Cov_{p_c}(\omega_e,\omega_f)}-1\Bigg|\leq \frac{1}{2}\Bigg\}
\end{equation}
happens with high probability assuming $\sigma_N^2\le c_2\cdot\textrm W(N)$. Under this event, one has $\Cov_{\frak p(s)}(\omega_e,\omega_f)\asymp \Cov_{p_c}(\omega_e,\omega_f)$ uniformly in $e,f\in E_N$, for some universal constant $\asymp$ (i.e.\ not depending on $\delta,\rho $). 

\paragraph{Step 4: Concluding that mixing rates remain stable.}
Fix an edge $e\in \bbT_N$ and $R\leq N$. Let $g$ be an edge at distance $R$ from $e$, and $f$ an edge approximating the midpoint of the shortest segment from $e$ to $g$. As long as $0\leq t \leq T_b(\delta,\rho)$, the measure $\phi_{\bbT_N,\frak{p}(t)}$ belongs to $\textrm{RSW}(\delta,N)$, therefore one can write
\begin{equation*}
    \frac{\Cov_{\frak{p}(t)}(\omega_e,\omega_f)\Cov_{\frak{p}(t)}(\omega_e,\omega_g)}{\Cov_{\frak{p}(t)}(\omega_f,\omega_g)}\asymp_\delta \frac{\Delta^{(e)}_{\frak{p}(t)}(R/2)\Delta^{(f)}_{\frak{p}(t)}(R/2)\Delta^{(e)}_{\frak{p}(t)}(R)\Delta^{(g)}_{\frak{p}(t)}(R)}{\Delta^{(f)}_{\frak{p}(t)}(R/2)\Delta^{(g)}_{\frak{p}(t)}(R/2)}\asymp_\delta \Delta^{(e)}_{\frak{p}(t)}(R)^2.
\end{equation*}
Therefore, on the event \eqref{eq:event-proof-L2}, which happens with a probability greater than\newline $1-O_{\delta,\rho} \Big(\exp\Big[ -c_1 N^{c_1} \frac{\textrm W(N)}{\sigma_N^2}  \Big] \Big)$), one can reconstruct (up to constants) the mixing rates via the the covariances using the formula 
\begin{equation*}
    \Delta^{(e)}_{\frak{p}(t)}(R)\asymp_\delta \sqrt{\frac{\Cov_{\frak{p}(t)}(\omega_e,\omega_f)\Cov_{\frak{p}(t)}(\omega_e,\omega_g)}{\Cov_{\frak{p}(t)}(\omega_f,\omega_g)}}.
\end{equation*}
This last estimate \emph{doesn't} depend on $\rho$, as it is only tied to the $\textrm{RSW}(\delta,N)$ property as recalled in the first item of Proposition \ref{prop:generalisation-thm-scaling-relations}. Still on the event \eqref{eq:event-proof-L2}, one has
\begin{equation*}
	   \Delta^{(e)}_{\frak{p}(t)}(R)\asymp_{\delta} \sqrt{\frac{\Cov_{p_c}(\omega_e,\omega_f)\Cov_{p_c}(\omega_e,\omega_g)}{\Cov_{p_c}(\omega_f,\omega_g)}}\asymp_{\delta} \Delta_{p_c}(R).
\end{equation*}
Taking $\rho=\rho(\delta)$ to be half of the implicit $\asymp_\delta$ constant in the previous equation (meaning that the ratio between critical and random mixing rates remains in $[2\rho,1/2\rho]$), one can conclude that with probability greater than $1-O_{\delta,\rho} \Big(\exp\Big[ -c_1 N^{c_1} \frac{\textrm W(N)}{\sigma_N^2}  \Big] \Big)$, all of the measures $(\phi_{\bbT_N,\frak{p}(s)})_{0\leq s \leq \sigma_N^2\wedge T_{b}(\delta,\rho)} $ belong to $\textrm{Stab}(2\rho(\delta),N)$. This concludes the proof.
\end{proof}
One can now combine the two previous lemmas to prove Proposition \ref{prop:L2-case}.

\begin{proof}[Proof of Proposition \ref{prop:L2-case}]
Choose first $\delta$ small enough according to Lemma \ref{lem:stability-RSW} and then ${\rho=\rho(\delta)>0}$ according to Lemma \ref{lem:stability-StabDelta}. There exist $c>0$ such that
\begin{align*}
	\mathbf P_N \Bigg[ \forall \, 0\leq s\leq \big( \sigma_N^2\wedge T_b(\delta,\rho)\big), \phi_{\bbT_N,\frak p(s)}\in \textrm{RSW}(2\delta,N)\cap\Big(\bigcap_{\#}\textrm{Stab}_{\#}(2\delta,N)\Big) \cap \textrm{Stab}_{\Delta}(2\rho,N) \Bigg]\\
	 \geq 1-O_{\delta}\Bigg( \exp\Bigg[-c N^{c}\frac{\textrm W(N)}{\sigma_N^2} \Bigg] \Bigg).
	 \end{align*}
The Brownian trajectories are $\mathbf{P}_N$-almost surely time continuous, which implies the same statement for crossing probabilities and mixing rates. For each realization of the event
\begin{equation*}
	\Bigg\{\forall~ 0 \leq  s \leq \sigma_N^2\wedge T_b(\delta,\rho),~\phi_{\bbT_N,\frak p(s)}\in \textrm{RSW}(2\delta,N) \cap \Big(\bigcap_{\#}\textrm{Stab}_{\#}(2\delta,N)\Big)\cap  \textrm{Stab}_{\Delta}(2\rho,N) \Bigg\},
\end{equation*}
there exists a (realization dependent) $\eps>0$ such that 
\begin{equation*}
    \forall ~ 0\leq  s\leq \big( \sigma_N^2\wedge T_b(\delta,\rho) \big)+\eps, \quad \phi_{\bbT_N,\frak p(s)}\in \textrm{RSW}(\delta,N)\cap\left(\bigcap_{\#}\textrm{Stab}_{\#}(\delta,N)\right) \cap \textrm{Stab}_{\Delta}(\rho,N).
\end{equation*}
On that event, one has $T_b(\delta,\rho)>(\sigma_N^2\wedge T_b(\delta,\rho))+\eps$, hence $\sigma_N^2<T_b(\delta,\rho)$. This reads as
\begin{equation*}
    \mathbf P_N\Big[T_b(\delta,\rho)<\sigma_N^2\Big] \lesssim_\delta  \exp\Bigg[-c N^{c}\frac{\textrm W(N)}{\sigma_N^2} \Bigg],
    \end{equation*}
which allows to conclude via the reformulation \eqref{eq:rewrite_L2-case} of Proposition \ref{prop:L2-case}.
\end{proof}

\subsection{General random variables satisfying $(\star)^{\textmd A}_N$}\label{sub:general-variables-L2}
In Section \ref{sub:centred-gaussian}, we proved Theorem \ref{thm:extension_scaling_window} in the specific case when the random bond parameters are independent Gaussians with a common variance. The role of the present section is to explain how one can extend the proof of this result to a sequence of variables $(\mathbf p_e)_{e\in  E_N}$ satisfying the condition $(\star)^{\text A}_N$. Recall that at each edge $e\in E_N$, the random variable $\mathbf p_e$ has standard deviation $\sigma_e:=\bbE[(\mathbf p_e-p_c)^2]^{1/2}$. Ideally, one would like to apply exactly the same reasoning as the Gaussian case, by continuously moving the bond parameters from the uniform  critical parameter $p_c(q)$ towards $\frak p=(\mathbf p_e)_e$ via some continuous process $\frak p(t)=(\mathbf p_e(t))_{e\in E_N}$. In fact, we would like this continuous process to resemble independent Brownian motions, allowing us to repeat the technology developed in Section \ref{sub:centred-gaussian}. This ideal setup can in fact be achieved by the famous Skorokhod embedding theorem, that allows to rewrite any reasonable centred random variable as some Brownian motion stopped appropriately. This reads in the following theorem.
\begin{theorem}[Skorokhod embedding theorem]
    Let $X$ be a centred real value random variable with a finite second moment. Let $(B_t)_{t\geq 0}$ a standard Brownian motion for some (completed) filtration $(\mathcal{F}_t)_{t\geq 0} $. Then there exist some stopping time $T$ for the filtration $(\mathcal{F}_t)_{t\geq 0} $ such that
 \[
\begin{array}{rcl}
B_T &\overset{(d)}{=}& X, \\
\bbE[T] &=& \bbE[X^2].
\end{array}
\]
\end{theorem}
We now explain how to implement the Skorokhod embedding theorem in the context of the random bond measure $\frak{p}$. There exist a sequence of independent Brownian motions $\big((B_t^{(e)})_{t\geq 0} \big)_{e\in E_N}$ for some probability measure $\mathbb{P}$ and a sequence of associated stopping times $\big(T^{(e)} \big)_{e\in E_N}$ for the Brownian motions $\big((B_t^{(e)})_{t\geq 0} \big)_{e\in E_N}$, such that for each edge $e\in E_N$, the variable $p_c+B^{(e)}_{T^{(e)}}$ has the law of $\mathbf p_e$. One can set
\begin{equation}\label{eq:def-continuous-process-Skorokhod}
    \mathbf p_e(t):= p_c(q)+B^{(e)}_{t\wedge T^{(e)}}.
\end{equation}
The stopping times $T^{(e)}$ coming from Skorokhod's embedding theorem are almost surely finite, which implies that almost surely (with respect to the randomness of the Brownian motions) one has $\mathbf p_e(\infty)\stackrel{(d)}=\mathbf p_e$.  To simplify readability, couple the process $\frak p(t)$ to the random environment $\frak p$ so that this equality in distribution as an equality. 

The stopping times given by the Skorokhod embedding theorem are a priori not unique, and we present here a construction that gives a stopping time with good tails estimates. The following technical lemma is proven in the Appendix via Lemma \ref{app:skorokhod_tails}.
\begin{lemma}\label{lem:skorohod_tails}
    Let $X$ be a random variable with stretched exponential tails
    \begin{equation*}
        \bbP[X>t]\lesssim \exp(-ct^{\gamma})
    \end{equation*}
    for some positive constants $\lesssim,c,\gamma$. Let $(B_t)_{t\geq 0}$ be a standard Brownian motion for the filtration $(\mathcal{F}_t)_{t\geq 0} $. Then there exists a stopping time $T$, for the filtration $(\mathcal{F}_t)_{t\geq 0}$, that satisfies the Skorokhod embedding Theorem and such that
    \begin{equation*}
        \bbP[T>t]\lesssim \exp\left(-c't^{\frac{\gamma}{\gamma+2}}\right)
    \end{equation*}
for some positive constants $\lesssim,c'$ only depending on $c$ and $\gamma$.
\end{lemma}
Applying this result to the random variables $\frac{\mathbf p_e-p_c(q)}{\sigma_e}$ with $\gamma=1$ (using the third condition in the definition of $(\star)^{\text A}_N$ and the Brownian scaling), there exists a stopping time $T^{(e)}$ such that
\begin{equation*}
    \bbP[T^{(e)}>t]\lesssim \exp\left[-c'\left(\frac t{\sigma_e^2}\right)^{1/3}\right],
\end{equation*}
where all constants involved are uniform in $N$ and $e\in E_N $. Recall that we denoted $\Sigma_{\frak p}:=\max_{e\in E_N} \sigma_e$. Therefore, as for any $e\in E_N$ one has $\sigma_e\le \Sigma_{\frak p}$, and so uniformly in $e\in E_N$ and $t>0$,
\begin{equation*}
    \bbP\big[T^{(e)}/\Sigma_{\frak p}^2>t\big]\lesssim \exp\left[-c't^{1/3}\right].
\end{equation*}
We now adapt the proof of Section \ref{sub:centred-gaussian}, keeping the same notations, replacing the Brownian motions $\frak p(t)$ by the Brownian motions stopped at $T^{(e)}$ defined in \eqref{eq:def-continuous-process-Skorokhod}. In this new setting, $\frak p=\frak p(\infty)$, and we want to show that the event $\big\{ T_b = \infty \big\} $ happens with high $\mathbf P_N$-probability. We start by recalling the following elementary lemma, whose proof is given in the Appendix \ref{app:large-deviation-sum}.
\begin{lemma}
\label{lem:tail_sum_stretched_exp}
   Fix a sequence of positive numbers $(a_i)_{i\le M}$ with $ m=\max_{i\leq M} a_i$ and  $S=\sum_{i\leq M}a_i $. Let $(T_i)_{i\le M}$ be a collection of positive random variables with uniform stretched exponential tails for some $0<\gamma<1$
 \begin{equation*}
    \bbP[T_i>t]\lesssim \exp\Big[-ct^{\gamma}\Big].
  \end{equation*}
Then there exist positive constants $C=C(c,\gamma)$ and $\tilde c=\tilde c(c,\gamma)>0$ such that uniformly in $t>C\cdot S$, one has
    \begin{equation*}
        \bbP\Bigg[\sum_{i\leq M}a_iT_i>t\Bigg]\lesssim \exp\Big[-\tilde{c} \cdot  \Big(\frac{t}{m}\Big)^{\gamma} \Big].
    \end{equation*}
\end{lemma}
In the spirit of the previous section, we will prove a pair of lemmas that show that, with high probability, the main features of the critical measure remain stable for a large enough amount of time.

\begin{lemma}
\label{lem:stability-RSW_general}
    Assume that $\delta>0$ is chosen small enough, only depending on $1< q\leq 4$. Then for all $\rho>0$, there are some positive constants $c=c(\delta,\rho,q)>0$ and $O_{\delta,\rho}(\cdot)$ such that if $\Sigma_{\frak p}^2\le c\cdot \textrm W(N)$, one has
\begin{equation*}
    \mathbf P_N \Bigg[ \forall \textrm{ } 0\leq  s\leq T_b(\delta,\rho), \textrm{ } \phi_{\bbT_N,\frak p(s)}\in \textrm{RSW}(2\delta,N) \Bigg]\geq 1-O_{\delta,\rho}\left( \exp\Bigg[-c N^c\left(\frac{\textrm W(N)^{1/2}}{\Sigma_{\frak p}}\right)^{1/2}\Bigg] \right).
\end{equation*}
Moreover, the same statement holds when replacing $\textrm{RSW}(\delta,N)$ by $\textrm{Stab}_{\#}(2\delta,N)$ for any $\#\in\{4,3^+,3^-,3^{\text{rg}},3^{\text{lf}}\}$.
\end{lemma}

\begin{lemma}
\label{lem:stability-StabDelta_general}
Assume that $\delta>0$ is arbitrary. Then there $\rho=\rho(\delta)>0$ small enough and constants $c=c(\delta,q)>0$ and $O_{\delta}(\cdot)$ such that if $\Sigma_{\frak p}^2\leq c\cdot \textrm W(N)$, one has
 \begin{equation*}
     \mathbf P_N \Bigg[ \forall \ 0\leq  s\leq T_b(\delta,\rho),\ \phi_{\bbT_N,\frak p(s)}\in \textrm{Stab}_{\Delta}(2\rho,N) \Bigg] \geq 1-O_{\delta}\left( \exp\Bigg[-c N^c\left(\frac{\textrm W(N)^{1/2}}{\Sigma_{\frak p}}\right)^{1/2}\Bigg] \right).
\end{equation*}
\end{lemma}
The proofs of these lemmas and how they allow to deduce Theorem \ref{thm:extension_scaling_window} are very similar to the previous section. Therefore, we only sketch the proof of Lemma \ref{lem:stability-RSW_general} for crossing events to illustrate how the control of the tails for the spotting times recalled in Lemma \ref{lem:tail_sum_stretched_exp} allows to control the probability of crossing events.

\begin{proof}[Sketch of the proof of Lemma \ref{lem:stability-RSW_general} for crossing events]
Let $\calR\subset \bbT_N$ be a $2$ by $1$ rectangle of width $R\leq N/8$, and apply Itô's formula to the process $t\mapsto \phi_{\bbT_N,\frak p(t)}[\scrC^0_{\bullet}(\calR)]$.

\paragraph{Step 0: Canonical decomposition of the semimartingale} One can make the decomposition $\phi_{\bbT_N,\frak p(t)}[\scrC^0_{\bullet}(\calR)]=M^\calR_t+A^\calR_t$ where $M^\calR$ is a local martingale started at $0$ and $A^\calR$ is an adapted finite variation process, given by
\begin{align*}
    \dd M^\calR_t &= \sum_{e\in E_N}\frac 1{\mathbf p_e(t)(1-\mathbf p_e(t))}\Cov_{\frak p(t)}(\scrC^0_{\bullet}(\calR),\omega_e)\ind(t\le T^{(e)})\dd B^{(e)}_t,\\
    \dd A^\calR_t &= \sum_{e\in E_N}\frac1{\mathbf{p}_e(t)(1-\mathbf{p}_e(t))^2}\left(1-\frac 1{\mathbf{p}_e(t)}\phi_{\bbT_N,\frak p(t)}[\omega_e]\right)\ind(t\le T^{(e)})\Cov_{\frak p(t)}(\scrC^0_{\bullet}(\calR),\omega_e)\dd t.
\end{align*}

\paragraph{Step 1: Evaluating the order of magnitude of $A^\calR_t$} Similarly to the previous section one can write for all $0\leq t \leq T_b(\delta,\rho)$,
\begin{equation*}
     \left|\frac{\dd A^\calR_t}{\dd t}\right|\lesssim \sum_{e\in E_N}\Cov_{\frak p(t)}(\scrC^0_{\bullet}(\calR),\omega_e)\ind(t\le T^{(e)}).
\end{equation*}
Still assuming that $0\leq t\le T_b(\delta,\rho)$, one can apply Lemma \ref{lem:near-critical_estimates} which implies that each $\Cov_{\frak p(t)}(\scrC^0_{\bullet}(\calR),\omega_e)$ term can be estimated via Lemma \ref{lem:Cov-crossing-edge} (more precisely working for crossing events of the form $\scrC^0_{\bullet}(\calR)$ as at the end of the proof of Step 1 of Lemma \ref{lem:stability-RSW}). Set 
\[
\alpha(\calR,e) :=
\begin{cases}
\Delta_{p_c}(R,\ell)\Delta_{p_c}(\ell)
 & \text{if } \textrm{dist}(e,\calR)=\ell \ge R \\
\Delta_{p_c}(n) & \text{if the distance from } e \text{ to the corners of } \calR \text{ is } n\lesssim R.
\end{cases}
\]
One can then apply Lemma \ref{lem:Cov-crossing-edge}, which implies that for  $0\leq t\le T_b$, one has the upper bound $\Cov_{\frak p(t)}(\scrC^0_{\bullet}(\calR),\omega_e)\lesssim \alpha(\calR,e)$. Integrating this information for $0\leq t\le T_b$ gives
\begin{equation*}
    \left|A^\calR_t-A^\calR_0\right|\lesssim \sum_{e\in E_N}\alpha(\calR,e)(t\wedge T^{(e)})\le \sum_{e\in E_N}\alpha(\calR,e)T^{(e)}
\end{equation*}
One can also notice that (\ref{eq:sum_cov_crossing_edge}) ensures that  $\sum_{e\in E_N}\alpha(\calR,e)\lesssim \textrm W(N)^{-1}$, while one can crudely bound $\alpha(\calR,e)\lesssim 1$. We are in position to apply Lemma \ref{lem:tail_sum_stretched_exp} to the family of variables $\big(T^{(e)}/\Sigma_{\frak p}^2\big)_{e\in E_N}$, which shows the existence of universal constants $C,\tilde c,c_1$ such that if $c_1\delta/\Sigma_{\frak p}^2\ge C\cdot\textrm W(N)^{-1}$, one has
\begin{align*}
    \mathbf P_N\Big[\exists \ 0\leq  t\le T_b,\quad \left|A^\calR_t-A^\calR_0\right|>\delta\Big]&\le \mathbf P_N\left[\sum_{e\in E_N}\alpha(\calR,e)T^{(e)}\ge c_1\delta\right]\\
    &\lesssim \exp\left[-\tilde c\left(\frac{c_1\delta}{\Sigma_{\frak p}^2}\right)^{1/3}\right].
\end{align*}
Fix $c_2:=c_1\delta/C$ and $c_3:=\frac12\tilde cc_1^{1/3}\delta^{1/3}$. Taking the union bound over the $O(N^3)$ $2$ by $1$ rectangles $\calR \subset \bbT_N$, one sees that provided $\Sigma_{\frak p}^2\le c_2\textrm W(N)$, one has
\begin{align*}
    \mathbf P_N\Big[\exists\ \calR\subset \bbT_N, \exists \ 0 \leq  t\le T_b,\quad\left|A^\calR_t-A^\calR_0\right|>\delta\Big]&\lesssim N^3\exp\left[-2c_3\Sigma_{\frak p}^{-2/3}\right]\\
    &\lesssim \exp\left[-c_3\Sigma_{\frak p}^{-2/3}\right].
\end{align*}
where in the last line we used that $\Sigma_{\frak p}^2\lesssim \textrm W(N)$ decays polynomially fast in $N$. Note that if $c>0$ is small enough, this probability is much smaller than the one in the statement of the lemma.
In the present context, 
$\left|A^\calR_t-A^\calR_0\right|$ is not deterministically bounded from above. Instead, the present reasoning shows that, only with very low probability, the quantity $\left|A^\calR_t-A^\calR_0\right|$ is large, allowing to easily adapt the previous arguments. 
\paragraph{Step 2: Evaluating the order of magnitude of $M_t^\calR$.} One has for $ 0\leq t \leq T_b$
\begin{align*}
     \left|\frac{\dd \langle M^\calR\rangle_t}{\dd t}\right|&\lesssim \sum_{e\in E_N}\Cov_{\frak p(t)}(\scrC^0_{\bullet}(\calR),\omega_e)^2\ind(t\le T^{(e)})\\
    &\lesssim \sum_{e\in E_N}\alpha(\calR,e)^2\ind(t\le T^{(e)}).
\end{align*}
Integrating the above one gets for $ 0\leq t \leq T_b$
\begin{equation*}
    \langle M^\calR\rangle_t\lesssim \sum_{e\in E_N}\alpha(\calR,e)^2(t\wedge T^{(e)})\le \sum_{e\in E_N}\alpha(\calR,e)^2T^{(e)}.
\end{equation*}
Using (\ref{eq:sum_cov_crossing_edge_squared}) gives
\begin{equation*}
\sum_{e\in E_N}\alpha(\calR,e)^2\lesssim \Xi(N)^{-1}.
\end{equation*}
One can now apply Lemma \ref{lem:tail_sum_stretched_exp} to the random variables $T^{(e)}/\Sigma_{\frak p}^2$ with the trivial bound that each $\alpha(\calR,e)^2$ is smaller than $\sum_{e\in E_N}\alpha(\calR,e)^2$, which provides the existence of positive constants $C,\tilde c,c_1$ such that if $\Sigma_{\frak p}^{1/2}\Xi(N)^{-1/4}\Big/\Sigma_{\frak p}^2\ge C\cdot\Xi(N)^{-1}$ one has 
\begin{align*}
    \mathbf P_N\left[\langle M^\calR\rangle_{T_b}>\Sigma_{\frak p}^{1/2}\Xi(N)^{-1/4}\right]&\le \mathbf P_N\left[\sum_{e\in E_N}\alpha(\calR,e)^2T^{(e)}\ge c_1\Sigma_{\frak p}^{1/2}\Xi(N)^{-1/4}\right]\\
    &\lesssim \exp\left[-2\tilde c\left(\frac{\Sigma_{\frak p}^{1/2}\Xi(N)^{-1/4}}{\Sigma_{\frak p}^2\Xi(N)^{-1}}\right)^{1/3}\right] \\
    &\lesssim \exp\left[-2\tilde c N^{\tilde c}\left(\frac{\textrm W(N)^{1/2}}{\Sigma_{\frak p}}\right)^{1/2}\right],
\end{align*}
where in the last line, we used Lemma \ref{lem:window-inequalities}, potentially changing the constant $\tilde c$.
Note that the condition on $\Sigma_{\frak p}$, required to apply Lemma \ref{lem:tail_sum_stretched_exp}, is equivalent to $\Sigma_{\frak p}\le C^{-2/3}\cdot\Xi(N)^{1/2}$, which, by Lemma \ref{lem:window-inequalities}, is satisfied for $N$ large enough as long as $\Sigma_{\frak p}^2\lesssim \textrm W(N)$. The union bound on $\calR\subset \bbT_N$ ensures that if $\Sigma_{\frak p}^2\lesssim\textrm W(N)$, one has
\begin{align*}
    \mathbf P_N\left[\exists\ \calR\subset\bbT_N,\quad \langle M^\calR\rangle_{T_b}>\Sigma_{\frak p}^{1/2}\Xi(N)^{-1/4}\right]&\lesssim N^3\exp\left[-2\tilde c N^{\tilde c}\left(\frac{\textrm W(N)^{1/2}}{\Sigma_{\frak p}}\right)^{1/2}\right]\\
    &\lesssim \exp\left[-\tilde c N^{\tilde c}\left(\frac{\textrm W(N)^{1/2}}{\Sigma_{\frak p}}\right)^{1/2}\right].
\end{align*}

\paragraph{Concluding via large deviation estimates}
Define the event 
\begin{equation*}
	G:=\Big\{ \forall \, \calR \subset \bbT_N, \, \forall\, 0\leq t \leq T_b,  \;   |A^\calR_t-A^\calR_0|\le \delta \Big\} \bigcap \left\{ \forall \, \calR \subset \bbT_N,  \;   \langle M^\calR\rangle_{T_b}\le \Sigma_{\frak p}^{1/2}\Xi(N)^{-1/4} \right\},
\end{equation*}
which happens with probability at least $1-O\left(\exp\left[-\tilde c N^{\tilde c}\left(\frac{\textrm W(N)^{1/2}}{\Sigma_{\frak p}}\right)^{1/2}\right]\right)$. On this event, one can use large deviation estimates for local martingales and Lemma \ref{lem:window-inequalities}, which ensure the existence of some $c>0$ small enough such that
\begin{align*}
    \mathbf P_N\left[G \bigcap \Big\{ \sup_{\calR\subset \bbT_N}\sup_{0\leq s\le T_b}|M^\calR_s|>\delta \Big\} \right]&\lesssim N^3\exp\left[-2c\Sigma_{\frak p}^{-1/2}\Xi(N)^{1/4}\right]\\
    &\lesssim \exp\left[-c N^{c}\left(\frac{\textrm W(N)^{1/2}}{\Sigma_{\frak p}}\right)^{1/2}\right].
\end{align*}
One may now conclude as in the proof of Proposition \ref{prop:L2-case}.
\end{proof}

\section{Random variables naturally centred near $p_c(q)$}\label{sec:non-centred-gaussian}
In the previous section we perturbed the environment from the critical one, using variables which were naively centred around the critical point, requiring that the expected value of the random bond parameter was $p_c(q)$. The toy example for the analysis was the use of some Brownian motions up to some small time $t\ll 1$, whose typical order of magnitude is $\sqrt{t}$. Looking more carefully,  this approach might look too naive, as one wants a near-critical random environment which remains in the critical phase, therefore neither favouring too much much the primal or the dual model. As our random bond parametrisation is \emph{not} naturally centred around $p_c$ from the perspective of self-duality, it is tempting to add a drift correction to the Brownian motions in order to preserve self-duality of the expected random environment up to its second moment. Note that in the natural isoradial parametrisation of the FK model (see e.g.\ \cite{duminilmanolesculi}), this drift would vanish. To each of the Brownian motions, we add a drift $-\mu(q)\cdot t$, as a natural second order expansion for a random process near $0$. In the present section, the critical window in random environment will be expressed via the quantity $\widetilde{\textrm{W}}(N)=\textrm{W}(N)(\sum_{r\le N}r\Delta_{p_c}(r)^4)^{-1/2}$. As CLE computations assert that the mixing rate $\iota$ is larger than $\frac{1}{2}$ (with equality only at $q=4$), the multiplicative correction to $\textrm{W}(N)$ is conjectured to be of constant order for $1<q<4$ and sub-polynomial for $q=4$, and so the results should be interpreted as involving the normal critical window $\textrm W(N)$ instead of $\widetilde{\textrm W}(N)$. 
\subsection{The special case of Gaussian variables}\label{sub:Gaussian-case-L3}

\subsubsection{Statement of the result and setup of the proof}

In this section, we prove an analog of Theorem \ref{thm:extension_scaling_window_drift} in the special case of a sequence of i.i.d.\ Gaussian variables whose law is given by 
\begin{equation}\label{eq:Gaussian-case-drift}
    \mathbf{p}_e \overset{(d)}{=} \calN(p_c(q)-\mu(q)\sigma_N,\sigma_N^2),
\end{equation}
where the parameter $\sigma_N$ goes to $0$ as the size of $\bbT_N$ goes to infinity, which ensures once again that with high probability (i.e.\ not worsening all the other estimates involved), all the edge weights $\frak{p}=(\mathbf{p}_e)_{e\in E_N}$ remain in $(0,1)$. We use once again i.i.d.\ Gaussian variables aiming to generalise afterwards by applying Skorohod's embedding theorem. We modify the i.i.d.\ Brownian motions from the previous section by adding a linear drift at each edge. Compared to the proof of Lemmas \ref{lem:stability-RSW} and \ref{lem:stability-StabDelta}, adding some well chosen drift allows us to lower (with high probability) the order of magnitude of the finite variation processes of the form $A^{\calR}_t$ and $A^{e,f}_t$, which were identified as being the dominant terms in the respective Itô derivatives. As it will be discussed in Section \ref{sub:optimality-random-window}, the conceptual reason to subtract that exact drift $t\mapsto \mu(q)t $ is to construct a random environment under $\mathbf{P}_N$ which is on average almost self-dual for its first two moments as in scenario $(\star)_N^{\text B}$. We keep the notations of Section \ref{sub:Gaussian-case-L3} and skip most details in the proofs that exactly match those of the previous section. This time, we prove a simplified version of Theorem \ref{thm:extension_scaling_window_drift} when $\mathbf{P}_N$ is given \eqref{eq:Gaussian-case-drift}. Note that Lemma \ref{lem:window-inequalities} ensures that the following result represents a strict improvement of the results obtained in Section \ref{sec:naive-random-bonds}, as there exists $c>0$, depending on $q$, such that 
 $\widetilde{\textrm{W}}(N)^{1/3}\gtrsim \textrm{W}(N)^{1/2}N^{c}$.
\begin{proposition}\label{prop:L3-case}
    Fix $1<q\leq 4$. Assume that the random environment $(\mathbf{p}_e)_{e\in E_N}$ under $\mathbf{P}_N$ is given by i.i.d\ Gaussian satisfying \eqref{eq:Gaussian-case-drift} for some uniform parameter $\sigma_N$. Then there exists some $\delta=\delta(q)>0$ and some constant $c=c(\delta,q)>0$ such that for any $N\geq 1 $ and $ \sigma_N \leq c\cdot \widetilde{\textrm{W}}(N)^{1/3} \log(N)^{-\frac12} $  
  \begin{equation*}
        \mathbf{P}_N\Bigg[ \phi_{\bbT_N,\frak p}\in \textrm{RSW}(\delta,N)\Bigg] >1-O\left(\exp\Bigg[-c\cdot \Bigg(\frac{\widetilde{\textrm{W}}(N)^{1/3}}{\sigma_N}\Bigg)^2\Bigg]\right).
    \end{equation*}
\end{proposition}
In this section, for $t\geq 0$, one can define the stochastic process $\frak p(t):=(\mathbf{p}_e(t))_{e\in E_N}$, formed of independent components given by
\begin{equation*}
    \mathbf{p}_e(t) := p_c(q)+B^{(e)}_t-\mu(q)t,
\end{equation*}
where $t\mapsto B^{(e)}_t$ is a standard Brownian motion attached to the edge $e$ and we recall the expression $\mu(q)=\frac{q-1}{2\sqrt q}$. Again $\frak p(0)=(p_c)_{e\in E_N}$ while $t\mapsto \frak p(t)$ is almost surely continuous and evolves towards the random process $\frak p(\sigma_N^2)\overset{(d)}{=}(\mathbf{p}_e)_{e\in \bbT_N}$. In that case, Proposition \ref{prop:L3-case} reads that if $t\lesssim \widetilde{\textrm{W}}(N)^{2/3}\log(N)^{-1}$, then, with high probability, all the measures $s\mapsto (\phi_{\bbT_N,\frak p(s)})_{s\leq t}$ remain within the $\textrm{RSW}(\delta,N)$ class. We keep the $\textrm{Stab}_{\#}(\delta,N),\textrm{Stab}_{\Delta}(\rho,N)$ and $T_b(\delta,\rho)$ notations of the previous sections. Once again, the building block to prove Proposition \ref{prop:L3-case} is a generalisation of Lemmas \ref{lem:stability-RSW} and \ref{lem:stability-StabDelta}. 

Consider an event $\calS$ and apply the Itô formula to $\phi_{\bbT_N,\frak p(t)}[\calS]$ with this new process $\frak p(t)$. A straightforward computation shows that
\begin{align*}
	\dd\phi_{\bbT_N,\frak p(t)}[\calS]&=\sum_{e\in E_N}\frac1{\mathbf p_e(t)(1-\mathbf p_e(t))}\Cov_{\frak p(t)}(\calS,\omega_e)\dd B^{(e)}_t \\
	& \quad + \sum_{e\in E_N}\frac1{\mathbf p_e(t)(1-\mathbf p_e(t))}\mathbf r_e(t)\Cov_{\frak p(t)}(\calS,\omega_e)\dd t
\end{align*}
where the $\mathbf r_e(t)$ is some process which can be computed using the expressions for $\mu(q)$ and $p_c(q)$, and is given explicitely by
\begin{equation}\label{eq:def-r_e}
    \mathbf r_e(t):=\frac{2p_c-1}{2p_c(1-p_c)}-\frac{2\mathbf{p}_e(t)-1}{2\mathbf{p}_e(t)(1-\mathbf{p}_e(t))}\left(1+\frac{1-2\phi_{\bbT_N,\frak{p}(t)}[\omega_e]}{2\mathbf{p}_e(t)}\right).
\end{equation}
If one wants to apply verbatim the reasoning of the previous section, one can brutally bound $\mathbf{r}_e(t)$ by a constant to obtain the same result. However, the drift is chosen specifically such that for any edge $e$, $\mathbf{r}_e(0)=0$. This hints that one can obtain some better bound than constant for $\mathbf{r}_e(t)$ near $t=0$, which will in fact be of order $\sqrt{t} $. This \emph{lowers} the value of all the finite variation processes involved in the semi-martingale decompositions and allows us to run the stochastic processes (here with drift) for a longer amount of time while remaining within the critical phase (in the sense of Section \ref{sub:near-critical_geometric_estimates}).

\subsubsection{A preliminary estimate on $\mathbf r_e(t)$}\label{sub:estimate-r_e}
As mentioned in the previous lines, self-duality of the critical model on $\bbT_N$ ensures that $\phi_{\bbT_N,\frak p(0)}[\omega_e]=\phi_{\bbT_N,p_c}[\omega_e]=\frac12$ while $\mathbf p_e(0)=p_c$, which ensures that $\mathbf r_e(0)=0$. We now give the following estimate, which is the main improvement from the methods introduced in Section \ref{sec:naive-random-bonds}, which allows us to work with distributions $\frak p$ whose mean deviation is larger than square root of the deterministic one.
\begin{lemma}
\label{lem:error-term-estimate}
There exist a constant $c=c(\delta,\rho,q)>0$ such that for any $\sigma_N $ satisfying ${\sigma_N \leq c\cdot \widetilde{\textrm{W}}(N) \log(N)^{-\frac{1}{2}} }$, one has
    \begin{multline*}
        \mathbf P_N\left[ \forall \, 0 \leq t\le \sigma_N^2\wedge T_b, \forall\, e\in E_N, \: |\mathbf r_e(t)|\le \textrm{W}(N)^{1/3}\Bigg(\sum_{r\le N}r\Delta_{p_c}(r)^4\Bigg)^{1/3}\right]\\ \ge 1-O_{\delta,\rho}\left(\exp\left[-c\left(\frac{\widetilde{\textrm W}(N)^{1/3}}{\sigma_N}\right)^2\right]\right).
    \end{multline*}
\end{lemma}

\begin{proof}
For $0\leq t\le T_b$, all of the bonds $\mathbf p_e$ are bounded away from $0$ and $1$. Therefore, taking the expression of $\mathbf r_e(t)$, one may apply the triangular inequality to get for any $0 \leq t \leq T_b$,
\begin{equation*}
 \mathbf r_e(t)\lesssim \Big|\mathbf p_e(t)-p_c\Big|+ \Big|2\phi_{\bbT_N,\frak p(t)}[\omega_e]-1\Big|.
\end{equation*}
Define the event
\begin{equation*}
	G_1:= \Bigg \{\forall \, 0 \leq  t\le \sigma_N^2\wedge T_b, \: \Big|\mathbf p_e(t)-p_c\Big|\le \eps\cdot\textrm{W}(N)^{1/3}\Bigg(\sum_{r\le N}r\Delta_{p_c}(r)^4\Bigg)^{1/3} \Bigg\}.
\end{equation*}
The term $|\mathbf p_e(t)-p_c|$ is the absolute value of a Brownian motion with  a bounded drift started at $0$. Therefore, standard Brownian estimates ensure that for any $\eps>0$, there exists $c_1=c_1(\eps)>0$ such that for any $e\in \bbT_N$ one has
\begin{equation*}
	\mathbf{P}_N\big[ G_1 \big] \\ \geq 1-O_{\delta,\rho,\eps}\Bigg(\exp\Bigg[ -2c_1 \cdot \Bigg(\frac{\textrm{W}(N)^{1/3}\left(\sum_{r\le N}r\Delta_{p_c}(r)^4\right)^{1/3}}{\sigma_N}\Bigg)^{\!2} \,\Bigg] \Bigg).
\end{equation*}
Recalling that $ \widetilde{\textrm W}(N)\leq \textrm{W}(N)$ and $\sum_{r\le N}r\Delta_{p_c}(r)^4\ge 1$, one can take the union bound over edges $e\in E_N$ and gets
\begin{equation*}
	\mathbf{P}_N\big[ G_1 \big] \geq  1-O_{\delta,\rho,\eps}\Bigg(N^2\exp\Bigg[ -2c_1 \cdot \left(\frac{\widetilde{\textrm W}(N)^{1/3}}{\sigma_N}\right)^2 \Bigg] \Bigg) \geq 1-O_{\delta,\rho,\eps}\Bigg(\exp\Bigg[ -c_1 \cdot \left(\frac{\widetilde{\textrm W}(N)^{1/3}}{\sigma_N}\right)^2 \Bigg] \Bigg),
\end{equation*}
where the second inequality holds as $\sigma_N\le c\cdot \widetilde{\textrm W}(N)^{1/3}\log(N)^{-1/2}$ with $c>0$ small enough.
To complete the reasoning, let us estimate $|2\phi_{\bbT_N,\frak p(t)}[\omega_e]-1|$, by computing the stochastic derivative of the process $t\mapsto \phi_{\bbT_N,\frak p(t)}[\omega_e]$. We proceed again in several steps.

\paragraph{Step 0: Canonical decomposition of the semimartingale.} One can first decompose $\phi_{\bbT_N,\frak p(t)}[\omega_e]:=M^{(e)}_t + A^{(e)}_t$ where $M^{(e)}_t$ is a local martingale started at $0$ and $A^{(e)}_t$ is a bounded variation process started at $\frac 12$. Itô's formula gives
\begin{align*}
	\dd M^{(e)}_t&=\sum_{f\in E_N}\frac 1{\mathbf{p}_f(t)(1-\mathbf{p}_f(t))}\Cov_{\frak p(t)}(\omega_e,\omega_f)\dd B^{(f)}_t, \\
	\dd A^{(e)}_t &=  \sum_{f\in E_N}\frac{1}{\mathbf{p}_f(t)(1-\mathbf{p}_f(t))}\mathbf r_f(t)\Cov_{\frak p(t)}(\omega_e,\omega_f)\dd t.
\end{align*}

\paragraph{Step 1: Estimating the local martingale $M^{(e)}_t$.}

For any $0\leq t \leq T_b$, one may use the stability of covariances and RSW to get that for $0\leq t \leq T_b$, one has
\begin{equation*}
   \langle M^{(e)}\rangle_t\asymp t\sum_{r\le N} r(\Delta_{p_c}(r)^2)^2=\left(\sum_{r\le N}r\Delta_{p_c}(r)^4\right)t.
\end{equation*}
Using once again large deviation estimates for local martingales, one gets that for any $\eps>0$, there exist $c_1>0$ such that for each edge $e\in E_N$,
\begin{multline*}
    \mathbf{P}_N\left[ \forall\, 0\le t\le \sigma_N^2\wedge T_b, \: |M^{(e)}_t|\le \eps\cdot \textrm{W}(N)^{1/3}\Bigg(\sum_{r\le N}r\Delta_{p_c}(r)^4\Bigg)^{1/3} \right]\\
    \geq 1-O_{\delta,\rho,\eps}\Bigg( \exp \Bigg[ -c_1\cdot \Bigg( \frac{\textrm W(N)^{1/3}\big(\sum_{r\le N}r\Delta_{p_c}(r)^4\big)^{1/3}}{\big(\sum_{r\le N}r\Delta_{p_c}(r)^4\big)^{1/2}\sigma_N} \Bigg)^2 \Bigg]  \Bigg)
\end{multline*}
Defining the event 
\begin{equation*}
	G_2:= \Bigg\{ \forall \,0 \leq   t\le \sigma_N^2\wedge T_b,\forall\, e\in E_N, \: |M^{(e)}_t|\le \eps\cdot\textrm{W}(N)^{1/3}\Bigg(\sum_{r\le N}r\Delta_{p_c}(r)^4\Bigg)^{1/3} \Bigg\},
\end{equation*}
recalling the definition of $\widetilde{\textrm{W}}(N)$ and taking union bound over edges $e\in E_N$ one gets
\begin{equation*}
	\mathbf{P}_N \big[ G_2 \big] \geq  1-O_{\delta,\rho,\eps}\Bigg(\exp\Bigg[ -c_1 \cdot \left(\frac{\widetilde{\textrm W}(N)^{1/3}}{\sigma_N}\right)^2 \Bigg] \Bigg).
\end{equation*}

\paragraph{Step 2: Concluding the rate of growth of $\mathbf r_e(t)$.}
One can now estimate the finite variation process $A^{(e)}_t$ by integrating Itô's formula. For any $0\leq t \leq  \sigma_N^2\wedge T_b$, using that $A^\calR_0=\frac 12$, one has
\begin{align*}
    \left|A^{(e)}_t-\frac 12\right|&\lesssim \left(\sup_{0\leq s\le t}\sup_{f\in E_N} |\mathbf r_f(t)|\right)\left(\sum_{r\le N}r\Delta_{p_c}(r)^2\right)t \\
    &\lesssim  \left(\sup_{0\leq s\le (\sigma_N^2\wedge T_b)}\sup_{f\in E_N} |\mathbf r_f(t)|\right)\Xi(N)^{-1}\sigma_N^2,
\end{align*}
where the first inequality uses \eqref{eq:bound-away-0-1}, the comparability of the edge covariances with their critical counterparts (as $0\leq t \leq T_b$) and their sum over $\bbT_N$ for the critical measure recalled in Remark \ref{rem:sum-Cov-arms-edge}. Going back to the bound on $\mathbf{r}_e(t)$, one gets
\begin{align*}
    \sup_{t,e}|\mathbf r_e(t)|&\lesssim \sup_{t,e}|\mathbf p_e(t)-p_c|+\sup_{t,e}\left|\phi_{\frak p(t)}[\omega_e]-\frac 12\right| \\
    &\lesssim \sup_{t,e}|\mathbf p_e(t)-p_c|+\sup_{t,e}|M^{(e)}_t|+\sup_{t,e}\left|A^{(e)}_t-\frac 12\right|
\end{align*}
where the above supremum $\sup_{t,e}$ is taken over times $0\le t\le (\sigma_N^2\wedge T_b(\delta,\rho))$ and edges $e\in E_N$. On the event $G_1 \cap G_2$ one has
\begin{equation*}
    \sup_{t,e}|\mathbf r_e(t)|\lesssim \eps \cdot \textrm{W}(N)^{1/3}\Big(\sum_{r\le N}r\Delta_{p_c}(r)^4\Big)^{1/3}+\Xi(N)^{-1}\sigma_N^2\Big(\sup_{t,e}|\mathbf r_e(t)|\Big).
\end{equation*}
One can then apply Lemma \ref{lem:window-inequalities} which states that
\begin{equation*}
	\Xi(N)^{-1}\sigma_N^2\lesssim \Xi(N)^{-1}\frac{\widetilde{\textrm{W}}(N)^{2/3}}{\log(N)^{\frac{1}{2}}}\lesssim \frac 1{\log(N)^{\frac12}}.
\end{equation*}
For $N$ large enough, one can conclude that
\begin{equation*}
	 \sup_{t,e}|\mathbf r_e(t)|\leq O\Big(\eps\cdot\textrm{W}(N)^{1/3}\Big(\sum_{r\le N}r\Delta_{p_c}(r)^4\Big)^{1/3}\Big)+\frac 12\sup_{t,e}|\mathbf r_e(t)|.
\end{equation*}
Choosing $\eps>0$ small enough such that $1/(2\eps)$ is larger than the implicit constant in the $O$, one gets
\begin{equation*}
	 \sup_{t,e}|\mathbf r_e(t)|\le \textrm{W}(N)^{1/3}\Big(\sum_{r\le N}r\Delta_{p_c}(r)^4\Big)^{1/3}.
\end{equation*}
The estimates on the respective probabilities of $G_1$ and $G_2$ allow to conclude the proof.
\end{proof}
\subsubsection{Completing the proof of Proposition \ref{prop:L3-case}}
We are in position to prove Proposition \ref{prop:L3-case} in the spirit of the proofs made in Section \ref{sub:centred-gaussian}, sorting out natural counterparts to Lemmas \ref{lem:stability-RSW} and \ref{lem:stability-StabDelta}. We will once again show that with high $\mathbf{P}_N$-probability, the breaking time $T_b(\delta,\rho)$ is large enough, finding some $\delta,\rho$ such that
\begin{equation}
    \mathbf P\Big[T_b(\delta,\rho)<\sigma_N^2\Big]\lesssim\exp\Bigg[-c\left(\frac{\widetilde{\textrm{W}}(N)^{1/3}}{\sigma_N}\right)^2\Bigg].
\end{equation}
This reads on the following lemmas.
\begin{lemma}
\label{lem:stability-RSW-L3} Assume that $\delta$ is chosen small enough, only depending on $1< q\leq 4$. Then for any $\rho>0$, there are some positive constants $c=c(\delta,\rho,q)$ and $O_{\delta,\rho}(\cdot)$ such that, provided  $\sigma_N^2\le c\cdot \widetilde{\textrm{W}}(N)^{1/3}\log(N)^{-1}$, one has
\begin{equation*}
    \mathbf P_N \Bigg[ \forall \, 0\leq  s\leq \Big( \sigma_N^2\wedge T_b(\delta,\rho)\Big), \, \phi_{\bbT_N,\frak p(s)}\in \textrm{RSW}(2\delta,N) \Bigg]\geq 1-O_{\delta,\rho}\left(\exp\Bigg[-c\cdot \left(\frac{\widetilde{\textrm{W}}(N)^{1/3}}{\sigma_N}\right)^2\Bigg]\right).
\end{equation*}
The same result holds when replacing $\textrm{RSW}(2\delta,N)$ by the stability events $\textrm{Stab}_{\#}(2\delta,N)$ for $\#\in\{4,3^+,3^{-},3^{\textrm{rg}},3^{\textrm{lf}}\}$.
\end{lemma}

\begin{lemma}
\label{lem:stability-StabDelta-L3}
Assume that $\delta$ is chosen small enough, only depending on $1< q\leq 4$. 
Then there exists a constant $\rho=\rho(\delta)>0$ small enough and some positive constants $c=c(\delta,\rho(\delta),q)>0$ and $O_{\delta,\rho}(\cdot)$ such that provided $\sigma_N^2\le c\cdot \widetilde{\textrm{W}}(N)^{1/3}\log(N)^{-1}$,
 \begin{equation*}
     \mathbf P_N \Bigg[ \forall \, 0\leq  s\leq \Big( \sigma_N^2\wedge T_b(\delta,\rho)\Big),\, \phi_{\bbT_N,\frak p(s)}\in \textrm{Stab}_{\Delta}(2\rho,N) \Bigg] \geq 1-O_{\delta,\rho}\left(\exp\Bigg[-c\cdot \left(\frac{\widetilde{\textrm{W}}(N)^{1/3}}{\sigma_N}\right)^2\Bigg]\right).\end{equation*}
\end{lemma}
We now pass to the proof of the those two lemmas, skipping the details regarding the potential degeneracies between $\delta$ and $\rho(\delta)$, as they can be tracked exactly as in the proof of Lemma \ref{lem:stability-StabDelta}. In particular, to lighten the reading, we also remove the $\delta,\rho$ dependencies in the constants $\lesssim $ and $\asymp $. We keep exactly the structure of the previous proofs, and mostly focus on evaluating how the introduction of the drift lowers the order of magnitude of the finite variation process.
\begin{proof}[Proof of Lemma \ref{lem:stability-RSW-L3}.]
Let us start by proving the announced statement for $\scrC^0_{\bullet}(\calR)$ when $\bullet\in\{h,v\}$. A similar reasoning allows to conclude once again for arm events, in the spirit of the arguments presented at the end of Lemma \ref{lem:stability-RSW}.

\paragraph{Step 0: Canonical decomposition of the semimartingale.}

Fix again a $2$ by $1$ rectangle $\calR\in \bbT_N$ of width smaller than $N/8$ and consider the event $\mathcal{S}=\scrC^0_{\bullet}(\calR)$. We may apply Itô's formula to the Brownian motions with drifts $\big( (B^{(e)}_t-\mu(q)t)_{t\geq 0}\big)_{e\in E_N}$ to get
\begin{align*}
    \dd\phi_{\bbT_N,\frak p(t)}[\mathcal{S}]&= \sum_{e\in E_N}\left(\frac{\partial}{\partial p_e}\bigg\rvert_{\underline p = \frak p(t)}\phi_{\bbT_N,\underline p}[\mathcal{S}]\right)\dd B^{(e)}_t -\mu(q) \sum_{e\in E_N}\left(\frac{\partial}{\partial p_e}\bigg\rvert_{\underline p = \frak p(t)}\phi_{\bbT_N,\underline p}[\mathcal{S}]\right)\dd t\\
    & \quad +\frac12\Bigg(\sum_{e\in E_N}\left(\frac{\partial^2}{\partial p_e^2}\bigg\rvert_{\underline p = \frak p(t)}\phi_{\bbT_N,\underline p}[\mathcal{S}]\right)\Bigg)\dd t,\\
	 &=\dd M_{t}^{\calR} + \dd A^\calR_t,
\end{align*}
where $M^\calR$ is a local martingale started at $0$ and $A^\calR$ is a finite variation process. Using Proposition \ref{prop:differentiate-bonds}, the associated explicit expressions read as
\begin{align*}
	\dd M^{\calR}_t&=\sum_{e\in E_N}\frac 1{\mathbf{p}_e(t)(1-\mathbf{p}_e(t))}\Cov_{\frak p(t)}(\scrC^0_{\bullet}(\calR),\omega_e)\dd B^{(e)}_t \\
	\dd A^\calR_t &=  \sum_{e\in E_N}\frac{1}{\mathbf{p}_e(t)(1-\mathbf{p}_e(t))}\mathbf r_e(t)\Cov_{\frak p(t)}(\scrC^0_{\bullet}(\calR),\omega_e)\dd t.
\end{align*}
We evaluate again the order of magnitude of $M^{\calR}_t $ and $A^{\calR}_t$ for times $0\leq t \leq T_b$ using the results of Section \ref{sub:near-critical_geometric_estimates}.

\paragraph{Step 1: Evaluating the order of magnitude of $A^\calR_t$.}

In this first step, we assume that the event $G_1\cap G_2$ of Lemma \ref{lem:error-term-estimate} holds, which happens with high probability. In that setup, for any $e \in E_N$ and any $0\leq t\le \sigma_N^2\wedge T_b$ one has $|\mathbf{r}_e(t)|\lesssim \textrm{W}(N)^{1/3}\left(\sum_{r\le N}r\Delta_{p_c}(r)^4\right)^{1/3}$. This ensures that for $0\leq t\le \sigma_N^2\wedge T_b$ one has
\begin{equation*}
	   \left|\frac{\dd A^\calR_t}{\dd t}\right|\lesssim \textrm{W}(N)^{1/3}\Big(\sum_{r\le N}r\Delta_{p_c}(r)^4\Big)^{1/3}\sum_{e\in E}\Cov_{\frak p(t)}(\scrC^0_{\bullet}(\calR),\omega_e)\lesssim \widetilde{\textrm W}(N)^{-2/3},
    \end{equation*}
    where we used Lemma \ref{lem:near-critical_estimates} to apply \eqref{eq:sum_cov_crossing_edge} and the exact expression of $\widetilde{\textrm W}(N)$. Integrating the previous equation, still intersecting with the event $G_1\cap G_2$ of Lemma \ref{lem:error-term-estimate}, one has for $0\leq t\le \sigma_N^2\wedge T_b$ and  $\sigma_N\lesssim \widetilde{\textrm W}(N)^{1/3}\log(N)^{-1/2}$ that
\begin{equation*}
	|A^\calR_t-A^\calR_0|\lesssim \widetilde{\textrm W}(N)^{-2/3}\cdot t \lesssim \log(N)^{-1}.
\end{equation*}

\paragraph{Step 2: Evaluating the order of magnitude of $M^\calR_t$.} The estimates here are exactly the same as in the proof of Lemma \ref{lem:stability-RSW}, and give that for all $0\leq t\le \sigma_N^2\wedge T_b$,
\begin{equation*}
	\frac{d}{\dd t}\langle M^\calR\rangle_t \lesssim \Xi(N)^{-1}.
\end{equation*}
Applying Lemma \ref{lem:window-inequalities} to get $\widetilde{\textrm W}(N)$ to appear and integrating gives that
\begin{equation*}
 \forall \, 0 \leq t\le \sigma_N^2\wedge T_b,\quad \langle M^\calR\rangle_t\lesssim \left(\frac{\sigma_N}{\widetilde{\textrm W}(N)^{1/3}}\right)^2.
\end{equation*}
Once the bounds of Step $1$ and $2$ on the two components of $\phi_{\bbT_N,\frak p(t)}[\scrC^0_{\bullet}(\calR)]$, one can apply verbatim the same large deviation principle, asserting that if $T_b$ is too small, then at least one of the local martingales associated to a crossing deviated from its initial value too fast compared to what its bracket typically allows, which happens with very small probability. Finally, as in the proof of Lemma \ref{lem:stability-RSW}, the previous proof adapts exactly to the case of stability of arm events.
\end{proof}
We now move to the proof of Lemma \ref{lem:stability-StabDelta-L3}, inspired by the one of Lemma \ref{lem:stability-StabDelta}.
\begin{proof}[Proof of Lemma \ref{lem:stability-StabDelta-L3}]
Once again, one should in principle specify all the $\delta,\rho$ dependencies in the constants of the proof, but one can replicate the reasoning at the end of the proof of Lemma \ref{lem:stability-StabDelta} to show that one can always choose $ \rho=\rho(\delta)>0$ small enough provided $\delta >0 $ and still run all the arguments.
\paragraph{Step 0: Canonical decomposition of the semimartingale.} Fix two edges $e,f\in E$ at distance $R$ from each other. Applying Itô's formula gives $\Cov_{\frak p(t)}(\omega_e,\omega_f)=M^{e,f}_t+A^{e,f}_t$ where $M^{e,f}$ is a local martingale started at $0$ and $A^{e,f}$ is an adapted process with bounded variation, whose stochastic derivatives are given by
\begin{align*}
    \dd M^{e,f}_t &= \sum_{g\in E_N}\frac1{\mathbf p_g(t)(1-\mathbf p_g(t))}\kappa^{\frak p(t)}_3(e,f,g)\dd B^{(g)}_t,\\
    \frac{\dd}{\dd t}A^{e,f}_t &=\sum_{g\in E_N}\frac1{\mathbf p_g(t)(1-\mathbf p_g(t))}\mathbf r_g(t)\kappa^{\frak p(t)}_3(e,f,g) \\
    &\quad -\sum_{g\in E_N}\frac1{\mathbf p_g(t)^2(1-\mathbf p_g(t))^2}\Cov_{\frak p(t)}(\omega_e,\omega_g)\Cov_{\frak p(t)}(\omega_f,\omega_g).
\end{align*}

\paragraph{Step 1: Evaluating the order of magnitude of $A^{e,f}_t$.}

Once again we assume that the high probability event $G_1\cap G_2$ of Lemma \ref{lem:error-term-estimate} holds, allowing us to assume that for any $g\in E_N $ one has $|\mathbf r_g(t)|\lesssim \textrm W(N)^{1/3}\left(\sum_{r\le N}r\Delta_{p_c}(r)^4\right)^{1/3}$. We can then run the computations as in the proof of Lemma \ref{lem:stability-StabDelta}, while adding the \emph{improved} bound on the $\mathbf r_g(t)$. One gets for $0\leq t\le \sigma_N^2\wedge T_b$,
\begin{align*}
	\left|\frac{\dd}{\dd t}A^{e,f}_t\right|&\lesssim \textrm W(N)^{1/3}\Big(\sum_{r\le N}r\Delta_{p_c}(r)^4\Big)^{1/3}\Bigg(\sum_{g\in E_N}\kappa^{\frak p(t)}_3(e,f,g)\Bigg)\\
	& \hspace{3cm} + \sum_{g\in E}\Cov_{\frak p(t)}(\omega_e,\omega_g)\Cov_{\frak p(t)}(\omega_f,\omega_g)\\
	& \lesssim \textrm W(N)^{1/3}\Big(\sum_{r\le N}r\Delta_{p_c}(r)^4\Big)^{1/3} \Delta_{p_c}(R)^2 \textrm{W}(N)^{-1}+\Delta_{p_c}(R)^2\cdot\Xi(N)^{-1}\\
	&\lesssim \Delta_{p_c}(R)^2\cdot\widetilde{\textrm W}(N)^{-2/3}.
\end{align*}
Passing to the last line one uses the definition of $\widetilde{\textrm W}(N)$ and Lemma \ref{lem:window-inequalities}. Integrating the previous bound on the event $G_1\cap G_2$ of Lemma \ref{lem:error-term-estimate} ensures that provided $0\leq t\le \sigma_N^2 \wedge T_b$, one has
\begin{equation*}
	|A^{e,f}_t-A^{e,f}_0|\lesssim \Delta_{p_c}(R)^2\log(N)^{-1}.
\end{equation*}

\paragraph{Step 2: Evaluating the order of magnitude of $M^{e,f}_t$.}
One can now compute the bracket of the local martingale $M^{e,f}$, which is given by
\begin{equation*}
    \frac{\dd}{\dd t}\langle M^{e,f}\rangle_t = \sum_{g\in E}\frac1{\mathbf p_g(t)^2(1-\mathbf p_g(t))^2}\kappa^{\frak p(t)}_3(e,f,g)^2.
\end{equation*}
Once again, we can use \eqref{eq:sum_kappa3_squared} to get
\begin{equation*}
	\forall \, 0 \leq  t\le \sigma_N^2\wedge T_b,\quad \frac{\dd}{\dd t}\langle M^{e,f}\rangle_t \lesssim \Delta_{p_c}(R)^4\cdot\Xi(N)^{-1}.
\end{equation*}
Using Lemma \ref{lem:window-inequalities} to make $\widetilde{\textrm W}(N)$ appear, and integrating in $t$ gives
\begin{equation*}
    \forall \, 0 \leq  t\le \sigma_N^2\wedge T_b,\quad\langle M^{e,f}\rangle_t \lesssim \Delta_{p_c}(R)^4\left(\frac{\sigma_N}{\widetilde{\textrm{W}}(N)^{1/3}}\right)^2
\end{equation*}
Once the bounds of Step $1$ and $2$ on the two components of $\Cov_{\frak p(t)}(\omega_e,\omega_f)$, one can apply again the same large deviation principle, asserting that if $T_b$ is too small, again some local martingale deviated too fast compared to what its bracket allows, which concludes the proof.
\end{proof}

\subsection{General random variables satisfying $(\star)^{\textmd B}_N$}\label{sub:general-variables-L3}

Once the special case of Gaussian variables with natural centring around $p_c(q)$ is settled, one can prove Theorem \ref{thm:extension_scaling_window_drift} in the general case, in the spirit of proof made in Section \ref{sub:general-variables-L2}. We start by fixing some random environment $\frak p$ over $\bbT_N$ that satisfies $(\star)^{\text B}_N$, where the standard deviations also satisfy $\Sigma_{\frak p}\le c\cdot \widetilde {\textrm W}(N)^{1/3}\log(N)^{-2}$ for some small enough constant $c>0$. The strategy of the proof can be summarised as follows. We still wish to construct a continuous family of random environment $\frak{p}(t)$, starting from the critical one and traveling towards $\frak{p}$. We will this time generalise \eqref{eq:def-continuous-process-Skorokhod}, setting 
\begin{equation*}
    \mathbf p_e(t)=p_c+B^{(e)}_{t\wedge T^{(e)}}-\mu^{(e)}\cdot(t\wedge T^{(e)}),
\end{equation*}
where $\big((B^{(e)}_t)_{t\geq 0}\big)_{e\in E_N}$ are i.i.d\ standard brownian motions, the variables $(T^{(e)})_{e\in E_N}$ are some almost-surely finite stopping times and the \emph{deterministic} constants $\mu^{(e)}$ will be cleverly chosen (via $\frak{p}$) to that $\mathbf p_e(\infty)=\mathbf{p}_e$. In this formalism, Section \ref{sub:Gaussian-case-L3} can be viewed as a special case where the spotting times $(T^{(e)})_{e\in E_N}$ are all equal to some deterministic $T>0$. Before diving into the construction of the variables $(T^{(e)})_{e\in E_N}$, le us run some preliminary computations. 
We claim that, for each edge $e\in E_N$, the exist $\mu^{(e)}>0$ such that
\begin{equation*}
    \bbE[e^{2\mu^{(e)}\cdot(\mathbf{p}_e-p_c(q))}]=1.
\end{equation*}
Let us first prove the existence of $\mu^{(e)}$ and derive its expansion. 
Fix $e\in E_N$ and consider the function $\mu\in \bbR_{\ge 0}\mapsto \bbE[e^{2\mu\cdot(\mathbf{p}_e-p_c(q))}]$. The function is equal to $1$ for $\mu=0$, has a derivative $2\bbE[\mathbf{p}_e-p_c(q)]<0$ at $0$, is convex, and tends to $+\infty$ as $\mu \to +\infty$. Therefore, there exist a unique $\mu^{(e)}>0$ such that $\bbE[e^{2\mu^{(e)}\cdot(\mathbf{p}_e-p_c(q))}]=1$. One can use the exponential tail bonds on $|\mathbf{p}_e-p_c(q)|$, which allows us to expand for $\mu\lesssim 1$,
\begin{equation*}
    \bbE[e^{2\mu(\mathbf{p}_e-p_c(q))}]=1+2\mu\bbE[\mathbf{p}_e-p_c(q)]+2\mu^2\bbE[(\mathbf{p}_e-p_c(q))^2]+O(\mu^3\sigma_e^3)
\end{equation*}
Notice that $2\mu\bbE[\mathbf{p}_e-p_c(q)]\asymp-\mu\sigma_e^2$ while $2\mu^2\bbE[(\mathbf{p}_e-p_c(q))^2]\asymp \mu^2\sigma_e^2$. Since $\sigma_e\to 0 $ uniformly in edges $e\in \bbT_N$ as $N\to \infty$, the asymptotics give that if $\mu$ is large enough (but still of constant order),then $\bbE[e^{2\mu\cdot(\mathbf{p}_e-p_c(q))}]>1$, and hence $\mu^{(e)}\lesssim 1$. Plugging the definition of $\mu^{(e)}$ in the above expansion one gets
\begin{equation*}
    1=\bbE[e^{2\mu^{(e)}\cdot(\mathbf{p}_e-p_c(q))}]=1+2\mu^{(e)}\bbE[\mathbf{p}_e-p_c(q)]+2(\mu^{(e)})^2\bbE[(\mathbf{p}_e-p_c(q))^2]+O((\mu^{(e)})^3\sigma_e^3)
\end{equation*}
which implies that
\begin{equation*}
    -\mu^{(e)}\mu(q)\sigma_e^2+(\mu^{(e)})^2\sigma_e^2=O((\mu^{(e)})^3\sigma_e^3)
\end{equation*}
and finally (dividing by $\mu^{(e)} \lesssim 1$) ensures that
\begin{equation}
\label{eq:mu-estimate}
    \mu^{(e)}=\mu(q)+O(\sigma_e).
\end{equation}
Define the time parametrised random environment 
\begin{equation*}
    \mathbf p_e(t):=p_c(q) + B_{t\wedge T^{(e)}}^{(e)}-\mu^{(e)}\cdot(t\wedge T^{(e)}),
\end{equation*}
where the stopping time $T^{(e)}$ will be constructed below to satisfy $\mathbf{p}_e=\mathbf p_e(\infty)$. One can now prove Theorem \ref{thm:extension_scaling_window_drift} exactly as we did for Proposition \ref{prop:L3-case}, modulo the ability to derive some extra estimates on the tails of the collection $(T^{(e)})_{e\in E_N}$. We construct explicitly the family $(T^{(e)})_{e\in E_N}$, which allows us to control their tails via the control of the tails of the variables $\frak{p} $, with the minor downside that our stopping times may now be infinite with positive probability. The result we obtain is the following.  
\begin{lemma}
\label{lem:skorokhod_drift_tails}
    Let $X$ be some real valued random variable such that $|X|\le 1$ and $\mu$ such that $\bbE[e^{2\mu X}]=1$. We also assume that $X$ has a tail
 \begin{equation*}
        \forall t>0,\quad\bbP[X>t]\lesssim \exp\Big[ -c\cdot(t/\sigma)^{\gamma}\Big],
 \end{equation*}
 for some positive constants $\sigma\le 1$ and $\gamma$. Let $(B_t)_{t\geq 0}$ be a standard Brownian motion for the filtration $\mathcal{F}_t$. Then there exists a stopping time $T$ for $(B_t)_{t\geq 0}$ such that:
    \begin{itemize}
        \item One may couple the pair $\big((B_t)_{t\geq 0},T\big)$ and $X$ such that, on the event $\{T<\infty\}$, one has 
         $B_T-\mu T \overset{(d)}{=} X$. 
        \item The exists positive constants $\lesssim,\tilde c$, only depending on $c,\gamma$ and the fact that $\mu$ is uniformly bounded away from $0$ and $\pm\infty$,  such that for all $t>0$
        \begin{equation*}
            \bbP\Big[T>t\Big]\lesssim \exp\Big[-\tilde c \cdot \sigma^{-\frac{\gamma}{\gamma+1}}\Big]+\exp\Big[-\tilde c \Big(\frac{t}{\sigma^2}\Big)^{\frac{\gamma}{\gamma+2}}\Big].
        \end{equation*}
    \end{itemize}
\end{lemma}
In the setup of variables satisfying $(\star)_N^{\text B}$ and for $X^{(e)}$ a random variable of distribution $\mathbf{p}_e-p_c$, one can take $\gamma=1$, while the variables $\mu^{(e)}$ remain bounded away from $0$ and $\pm\infty$ (as a consequence of the preliminary computations). Denoting by $T^{(e)}$ the stopping time constructed for the variable $X^{(e)}$ and the Brownian motion $(B^{(e)}_t)_{t\geq 0}$, one has
\begin{equation*}
    \forall t>0,\quad \mathbf P_N\Bigg[T^{(e)}>t\Bigg]\lesssim \exp\left[-c'\cdot \left(\frac{t}{\sigma_e^2}\right)^{1/3}\right]+\exp\Bigg[-c'\sigma_e^{-\frac12}\Bigg].
\end{equation*}
In general, using Lemma \ref{lem:tail_sum_stretched_exp} to bound $\sum_{e\in E_N}\alpha_e T^{(e)}$ is very inefficient, as one single coefficient $\alpha_e$ could be of the order of magnitude of $\sum_{e\in E_N}\alpha_e$. Instead, we use the most naive union bound on the supremum of the $(T^{(e)})_{e\in E_N}$. This reads as
\begin{align*}
\mathbf P_N\Bigg[\sup_{e\in E_N} T^{(e)}>\Sigma_{\frak p}^{1/2}\widetilde{\textrm W}(N)^{1/2}\Bigg]&\lesssim N^2\left(\exp\left[-c'\cdot \left(\frac{\Sigma_{\frak p}^{1/2}\widetilde{\textrm W}(N)^{1/2}}{\Sigma_{\frak p}^2}\right)^{1/3}\right]+\exp\Bigg[-c'\cdot \Sigma_{\frak p}^{-1/2}\Bigg]\right) \\
&\lesssim  \exp\left[-\frac12c'\cdot \left(\frac{\widetilde{\textrm W}(N)^{1/3}}{\Sigma_{\frak p}}\right)^{1/2}\right]+\exp\Bigg(-\frac12c'\cdot \Sigma_{\frak p}^{-1/2}\Bigg)
\end{align*}
where the passage to the last line uses that $\Sigma_{\frak p}\le c\widetilde{\textrm W}(N)^{1/3}\log(N)^{-2}$. Conditioning on the event
\begin{equation*}
	\Bigg\{ \sup_{e\in E_N} T^{(e)}\le\Sigma_{\frak p}^{1/2}\widetilde{\textrm W}(N)^{1/2} \Bigg\},
\end{equation*}
one can reproduce almost verbatim the proof of Proposition \ref{prop:L3-case} using $\left(\Sigma_{\frak p}^{1/2}\widetilde{\textrm W}(N)^{1/2}\right)^{1/2}$ to replace $\sigma_N$. The only notable difference in the computation is the appearance of some additional term in the Itô derivatives, which we detail now. For an event $\mathcal{S}$, the presence of the drift gives that
\begin{align*}
	 \dd\phi_{\bbT_N,\frak p(t)}[\calS]&=\sum_{e\in E_N}\frac1{\mathbf p_e(t)(1-\mathbf p_e(t))}\Cov_{\frak p(t)}(\calS,\omega_e)\dd B^{(e)}_t \\
	 & \quad + \sum_{e \in E_N}\frac1{\mathbf p_e(t)(1-\mathbf p_e(t))}\mathbf r_e(t)\Cov_{\frak p(t)}(\calS,\omega_e)\dd t\\
	 &\quad +\sum_{e\in E_N}(\mu^{(e)}-\mu(q))\frac1{\mathbf p_e(t)(1-\mathbf p_e(t))}\Cov_{\frak p(t)}(\calS,\omega_e)\dd t
\end{align*}
Recalling that \eqref{eq:mu-estimate}, one can bound $|\mu^{(e)}-\mu(q)|=O(\sigma_e)\lesssim \Sigma_{\frak p}$, and this upper bound is smaller than the one obtained for $\mathbf r_e(t)$ in Section \ref{sub:estimate-r_e}, so one can apply the same proof. Therefore, there exists $c>0$ small enough such that such that provided $\Sigma_{\frak{p}}$ is chosen to satisfy 
$\big(\Sigma_{\frak p}^{1/2}\widetilde{\textrm W}(N)^{1/2}\big)^{1/2}\le c\cdot\widetilde{\textrm W}(N)^{1/3}\log(N)^{-1/2}$, or equivalently ${\Sigma_{\frak p}\le c^4\cdot\widetilde{\textrm W}(N)^{1/3}\log(N)^{-2}}$,
\begin{align*}
    \mathbf{P}_N\Bigg[ \phi_{\bbT_N,\frak p}\in \textrm{RSW}(\delta,N)\Bigg] &>1-O\left(\exp\left[-c\cdot \left(\frac{\widetilde{\textrm{W}}(N)^{1/3}}{\left(\Sigma_{\frak p}^{1/2}\widetilde{\textrm W}(N)^{1/2}\right)^{1/2}}\right)^{\!2}\,\right]\right)\\
    &\ge 1-O\left(\exp\left[-c\cdot \left(\frac{\widetilde{\textrm{W}}(N)^{1/3}}{\Sigma_{\frak p}}\right)^{\!1/2}\right]\right).
\end{align*}
Combining all observations made above allows us to conclude the proof of Theorem \ref{thm:extension_scaling_window_drift}.

\subsection{Optimality of the random bond near-critical window}\label{sub:optimality-random-window}

In the previous section, we showed that strengthening the condition $(\star)^{\text A}_N$ to $(\star)^{\text B}_N$ allows us to use a random a random environment where the bonds have a much larger standard deviation. We now explain that a critical window in a random environment of order $\textrm W(N)^{1/3}$ is ``essentially optimal'', in the sense that finding a condition on the (independent) random bond environment $\frak p$ to enlarge the critical window would require some much more refined understanding than simply the up to constant behaviour of crossing estimates, mixing rates and arm exponents.

Let us start by discussion how the $O(\sigma_e^3)$ correction term in condition $(\star)^{\text B}_N$ naturally appears. In the special case $q=1$, the results of 
Section \ref{sec:percolation} show that the natural condition on the independent random bond parameters is that $\bbE[\mathbf p_e]=\frac 12$, which allows us to preserve the critical phase of the model (at large scale) even with random variables whose standard variation does not decay to $0$ as $N \to \infty$. When remaining in the near-critical random setup, one may expect that this ``criticality condition" generalises to $q> 1$, in the sense that there should exist  $f=f_q$ some smooth function depending on $q$ such that the natural condition on the parameters $\mathbf p_e$ to preserve criticality at large scales is given by $\bbE[f(\mathbf p_e)]=0$. While $f_1$ is easy to guess as it corresponds to asking that the annealed model is critical Bernoulli percolation, this requires some more evolved reasoning when $q>1$. For $q>1$, the critical point $p_c(q)$ is determined by the fact that it is the self-dual point of the model. In this spirit, the condition we require on $f_q$ is that $\bbE[f_q(\mathbf p)]=0$ for any random variable $\mathbf p$ which is self-dual for the FK percolation model. By self-dual, we mean that $\mathbf p$ and $\mathbf p^*$ have the same law, where
\begin{equation*}
    \mathbf p^* := \frac{q(1-\mathbf p)}{\mathbf p+q(1-\mathbf p)}
\end{equation*}
is the dual parameter of $\mathbf p$. Taking $\mathbf p:=\frac12\delta_p+\frac12\delta_{p^*}$ for $p\in [0,1]$, this implies that $f_q$ must satisfy
\begin{equation}\label{eq:condition-dualite-f_q}
    \forall p\in[0,1],\quad f_q(p^*)=-f_q(p).
\end{equation}
In the near-critical setup $\varepsilon \to 0 $, one can expand near the critical point $p_c=p_c(q)$ the relation $ f_q((p_c+\eps)^*)=-f_q(p_c+\eps)$, which implies that 
\begin{equation*}
    f_q\left(p_c-\eps-\frac{q-1}{\sqrt q}\eps^2-\frac{(q-1)^2}{q}\eps^3-\ldots\right)=-f_q(p_c)-\eps f_q'(p_c)-\frac{\eps^2}2 f_q''(p_c)-\frac{\eps^3}6f^{(3)}_q(p_c)+\ldots,
\end{equation*}
which reads as
\begin{equation}\label{eq:condition-f_q}
    2f_q(p_c)+0\cdot\eps+\left(f_q''(p_c) - \frac{q-1}{\sqrt q}f_q'(p_c)\right)\eps^2+\left(\frac{q-1}{\sqrt q}f_q''(p_c)-\frac{(q-1)^2}q f_q'(p_c)\right)\eps^3+\ldots=0
\end{equation}
The above equation shows that if one wants to cancel all the coefficients in the expansion, although there are infinitely many conditions on the derivatives of  $f_q$ at $p_c$, these conditions are \emph{not} sufficient to uniquely determine $f_q$. For every $n\geq 1$, the condition \eqref{eq:condition-f_q} imposes an equation linking the first  $2n$ derivatives at $p_c$, but with \emph{infinitely} many degrees of freedom for $f_q$. Indeed, one can freely choose the \emph{odd degree}  derivatives of $f_q$ at $p_c$, which is enough to fix the derivatives of even degree. This can be seen by interpreting \eqref{eq:condition-dualite-f_q} as asking, up to conjugation by some smooth map, that $f_q$ is an \emph{odd function}, whose odd derivatives at zero (corresponding to $p_c$ under the conjugation) can be chosen arbitrarily, and even derivatives at zero are forced to cancel.

Let us now specify those conditions for the first derivatives of $f_q$. This clearly imposes that $f_q(p_c)=0$. One can always replace $f_q$ by a constant multiple and still satisfy \eqref{eq:condition-dualite-f_q}, which allws us to arbitrarily fix the value of $f_q'(p_c)$. Then, \eqref{eq:condition-f_q} imposes that $f_q''(p_c)=\frac{q-1}{\sqrt q}f_q'(p_c)$. Assuming that $f_q$ is at least three times differentiable at $p_c$, plugging in some variable $\mathbf p$ well concentrated around $p_c$ in the equation $\bbE[f_q(\mathbf p)]=0$, one gets
\begin{equation*}
    f_q'(p_c)\bbE[\mathbf p-p_c]+\frac{f_q''(p_c)}2\bbE[(\mathbf p-p_c)^2]+O(\bbE[(\mathbf p-p_c)^2]^{3/2})=0
\end{equation*}
The link between $f_q'(p_c)$ and $f_q''(p_c)$ precisely gives item $2\textrm B)$ in Definition \ref{def:proprietes_pe}. 

Now although we do not claim that there does not exist some function $f_q$ allowing us to have environments which stay critical in a window of size larger than $\textrm W(N)^{1/3}$, we explain why a generic choice of a function $f_q$ satisfying \eqref{eq:condition-dualite-f_q} cannot allow us to go further than $\textrm W(N)^{1/3}$. Pick a generic $f_q$ satisfying \eqref{eq:condition-dualite-f_q} and an environment $\frak p=(\mathbf p_e)_{e\in e_N}$ of i.i.d.\! bond parameters such that $\mathbb{E}[f_q(\mathbf p_e)]=0$, and assume by contradiction that the environment can be chosen such that $\sigma_e\gg \textrm W(N)^{1/3}$. Now consider the deterministic modification of the environment $\widetilde{\mathbf p}_e=\mathbf p_e+\sigma_e^3$. Since $\bbE[f_q(\mathbf p_e)]=0$, one clearly has $\bbE[f_q(\widetilde{\mathbf p}_e)]=O(\sigma_e^3)$. It is not hard to see that it is possible to construct a function $\widetilde{f}_q$, satisfying \eqref{eq:condition-dualite-f_q} as $\eps\to 0$ in the following way: under the conjugation map where \eqref{eq:condition-dualite-f_q} becomes the condition that $f_q$ is odd, one may pick $\widetilde f_q(x)=f_q(x)-\alpha x^3$ for some well chosen $\alpha$. In this construction, $\widetilde f_q$ obviously depends on the law of the $\mathbf p_e$ (and so on $N$ as well), but one may check that $\alpha$ is bounded uniformly in $\sigma_e$, so that even as $N\to\infty$ and $\sigma_e\to 0$, one may still consider $\widetilde f_q$ to be a generic function satisfying \eqref{eq:condition-dualite-f_q}. Now since $\widetilde f_q$ is a generic function satisfying \eqref{eq:condition-dualite-f_q}, under our assumption, the environment $\widetilde {\frak p}=(\widetilde{\mathbf p}_e)_{e\in E_N}$ should still have some $\textrm{RSW}(\delta,N)$ property on the torus $\bbT_N$. However, between the environments $\mathbf p_e$ and $\widetilde{\mathbf p}_e$, the coupling constant differ \emph{deterministically} by a positive constant $\sigma_e^3\gg \textrm W(N)$. In particular, assuming $\mathbf p_e$ remains in the critical phase (with high $\mathbf{P}_N$-probability), this deterministic modification, way above the critical window, imposes that the model with environment $\widetilde {\frak p}$ should be super-critical (with high $\mathbf{P}_N$-probability), by simply repeating all the arguments in the proof Theorem \ref{thm:scaling-relations_second} starting from $\frak p$ instead of $p_c$, and keeping track of the reverse inequalities (see also \cite[Section 5]{FK_scaling_relations}). This is a contradiction, and so a generic choice of $f_q$ could not lead to critical windows larger than $\textrm W(N)^{1/3}$ for the random environment.

\subsection{Non-independent bonds and extension of the critical window}
In this section, we explore the less restrictive setup of bond parameters at each edge that may not be independent. On $\bbT_N$, we work in a setup where the bonds can be seen as the sum of an i.i.d\ process at each edge and some (much smaller) correction drift involving the entire random environment in $\bbT_N$. Consider stochastic processes of the form $\phi_{\bbT_N,\frak{p}(t)}[\calS]$, similar to the one treated in Sections \ref{sub:centred-gaussian} and \ref{sub:Gaussian-case-L3}. One of the key observation when computing $\dd\phi_{\bbT_N,\frak{p}(t)}[\calS]$ is that, with high $\mathbf{P}_N$ probability, its dominant term is its finite variation part. Our specific choice of the drift is made here to \emph{cancel exactly that drift}, making of the processes of the form $\phi_{\bbT_N,\frak{p}(t)}[\scrC(S)]$ pure local martingale, whose stochastic derivatives (estimated via their brackets) are expected to be much smaller, allowing to run some controlled at each scale deformation for a much longer amount of time. In this section, recall the quantity $\Xi(N)=(\sum_{r\le N}r\Delta_{p_c}(r)^2)^{-1}$. Let us note that the inequality $\Xi(N)\gtrsim \widetilde {\textrm W}(N)^{2/3}$ is essentially sharp (up to subpolynomial factors) when $\Delta_{p_c}(R)\gtrsim R^{-1/2+o(1)}$, but assuming that $\Delta_{p_c}(R)\lesssim R^{-1/2-c}$ for some $c>0$ (which should be the case for $q<4$), the inequality should hold with an additional polynomial correction $\Xi(N)\gtrsim N^{c'}\cdot\widetilde {\textrm W}(N)^{2/3}$, and so the window reached in the following proposition is much larger than that of the previous section. Once again, we start by presenting a simplified version of the Theorem in the case where the random environment $\frak{p}$ is Gaussian.
\begin{proposition}\label{prop:non-independent-bonds_gaussian}
Fix $1<q\le 4$ and let $\frak p=(\mathbf p_e)_{e\in E}$ a random environment satisfying \eqref{eq:Gaussian-case-drift}. There exist some positive constants $\delta,O,c_{1,2}$ such that if $\sigma_N\le c_1\cdot \Xi(N)^{1/2}\log(N)^{-1/2}$, one can construct some random environment $\widetilde{\frak p}=(\widetilde{\mathbf p}_e)_{e\in E_N}$ such that
\begin{itemize}
	\item Under $\mathbf P_N$, for each $e\in E_N$, the difference $|\widetilde{\mathbf p}_e-\mathbf p_e|$ is of order $\Xi(N)^{-1/2}\sigma_N^3$, in the sense that
	\begin{equation*}
        \forall \alpha>0,\quad \mathbf P_N\Big[|\widetilde{\mathbf p}_e-\mathbf p_e|>\alpha\cdot\Xi(N)^{-1/2}\sigma_N^3\Big]\lesssim \exp(-c\alpha^2).
    \end{equation*}
	\item  One has \begin{equation*}
        \mathbf P_N\Big[\phi_{\bbT_N,\widetilde{\frak p}}\in \textrm{RSW}(\delta,N)\Big]>1- O\Bigg(   \exp \Bigg[ -c_2\cdot \left(\frac{\Xi(N)^{1/2}}{\Sigma_{\frak p}}\right)^{2} \Bigg] \Bigg).
    \end{equation*}
\end{itemize}
\end{proposition}
Let us comment on the implications of this statement. Informally speaking, Theorem \ref{thm:non-independent-bonds} states that, one can correct at each edge the i.i.d\ Gaussian, whose typical order of magnitude is $\sigma_N$, by some random factor, depending on entire environment $\frak{p} $, whose typical order of magnitude is $\Xi(N)^{-1/2}\cdot \sigma_N^3 \ll \sigma_N$, while enlarging the critical window in random environment. Using the CLE predictions (asserting that $\Delta_{p_c}(R) \lesssim R^{-\frac{1}{2}-c}$ for $q<4$), one expects this enlargement to be at least of polynomial order $1<q<4$. For $1<q \leq 2$, one expects the critical window in the random environment to be, with high probability,  of logarithmic order. Indeed, it is conjectured that $\Delta_{p_c}(R)\lesssim R^{-1-c}$ for $q<2$ ensuring that $\Xi(N) \asymp 1$. For the limiting case $q=2$, it is known that $\Delta_{p_c}(R)\asymp R^{-1}$, so that $\Xi(N)\asymp \log(N)^{-1}$.
\begin{proof}[Proof of Proposition \ref{prop:non-independent-bonds_gaussian}]
We will once again embed $\mathbf p_e$ as a Brownian motion with some drift, given by
\begin{equation*}
    \mathbf p_e(t)=p_c+B^{(e)}_{t}-\mu(q)t
\end{equation*}
and couple the Brownian motions with $\frak p$ such that for each $e \in E_N $ one has $\mathbf p_e(\sigma_N^2)=\mathbf p_e$. From now on, define, on the same probability space, the process $\widetilde{\frak p}(t)=(\widetilde{\mathbf{p}}_e(t))_{e\in E_N}$ via
\begin{equation*}
    \widetilde{\mathbf{p}}_e(t)=\mathbf p_e(t)+\eps^{(e)}(t)
\end{equation*}
where $t\mapsto(\eps^{(e)}(t))_{e\in E_N}$ is an (adapted) continuously differentiable process such that for any $e\in E_N$, one has $\varepsilon^{(e)}(0)=0$ and $(\eps^{(e)}(t))_{e\in E_N}$ satisfies the SDE
\begin{equation}
\label{eq:eps_SDE}
\frac{\dd}{\dd t}\eps^{(e)}(t)=\mu(q)\cdot t-\frac1{1-\widetilde{\mathbf p}_e(t)}\left(1-\frac 1{\widetilde{\mathbf p}_e(t)}\phi_{\bbT_N,\widetilde{\frak p}(t)}[\omega_e]\right).
\end{equation}
The stochastic differential system satisfied by $(\eps^{(e)}(t)_{t\geq 0})_{e\in E_N}$ involves \emph{the entire environment via} $\phi_{\bbT_N,\widetilde{\frak p}(t)}[\omega_e]$, making the bonds parameters in $\widetilde{\mathbf{p}}_e=\widetilde{\mathbf{p}}_e(\sigma_N^2)$ not independent from each other. In this setup, we once again use the notation $T_b=T_b(\delta,\rho)$ from the previous sections.
For an event $\calS$, one can write the random process $\phi_{\bbT_N,\widetilde{\frak p}(t)}[\calS]=M^{\calS}_t+A^{\calS}_t$ where $M^{\calS}$ is a local martingale started at $0$ and $A^{\calS}$ is an adapted process with bounded variation. The (a priori unnatural) choice of $(\eps^{(e)}(t)_{t\geq 0})_{e\in E_N}$ is made to leave $A^{\calS}_t$ constant, meaning that
\begin{align*}
    \frac{d}{\dd t}A^{\calS}_t &= \sum_{e\in E_N}\frac1{\widetilde{\mathbf p}_e(t)(1-\widetilde{\mathbf p}_e(t))^2}\left(1-\frac 1{\widetilde{\mathbf p}_e(t)}\phi_{\bbT_N,\widetilde{\frak p}(t)}[\omega_e]\right)\Cov_{\widetilde{\frak p}(t)}(\calS,\omega_e) \\
    &\quad + \sum_{e\in E_N}\frac1{\widetilde{\mathbf p}_e(t)(1-\widetilde{\mathbf p}_e(t))}\left(-\mu^{(e)}+\frac{\dd}{\dd t}\eps^{(e)}(t)\right)\Cov_{\widetilde{\frak p}(t)}(\calS,\omega_e) \\
    &= 0.
\end{align*}
Therefore, all the $\phi_{\bbT_N,\widetilde{\frak p}(t)}[\calS]=M^{\calS}_t+A^{\calS}_t$ are pure local martingales, whose bracket is given by
\begin{equation*}
    \frac{\dd}{\dd t}\langle\phi_{\bbT_N,\widetilde{\frak p}(t)}[\calS]\rangle_t = \sum_{e\in E}\frac1{\widetilde{\mathbf p}_e(t)^2(1-\widetilde{\mathbf p}_e(t))^2}\Cov_{\widetilde{\frak p}(t)}(\calS,\omega_e)^2.
\end{equation*}
To conclude, it is enough to prove that one can chose positive constants $\delta,\rho$ such that the event $T_b>\sigma_N^2$ happens with high $\mathbf{P}_N $ probability. The proofs are very close to those of Propositions \ref{prop:L2-case} and \ref{prop:L3-case}, allowing us to only make a sketch.

\paragraph{Step 1: Estimating the quadratic variation of $\phi_{\bbT_N,\widetilde{\frak p}(t)}[\scrC^0_{\bullet}(\calR)]$ and $\phi_{\bbT_N,\widetilde{\frak p}(t)}[\calA^{(e)}_{\#}(R)]$.}
For each $2$ by $1$ rectangle $\calR\subset \bbT_N$ and any $\bullet\in\{h,v\}$, the process $\phi_{\bbT_N,\widetilde{\frak p}(t)}[\scrC^0_{\bullet}(\calR)]$ is a local martingale, whose quadratic variation is given by
\begin{equation*}
    \frac{\dd}{\dd t}\langle \phi_{\bbT_N,\widetilde{\frak p}(\cdot)}[\scrC^0_{\bullet}(\calR)]\rangle_t=\sum_{e\in E_N}\frac1{\widetilde{\mathbf p}_e(t)^2(1-\widetilde{\mathbf p}_e(t))^2}\Cov_{\widetilde{\frak p}(t)}(\scrC^0_{\bullet}(\calR),\omega_e)^2
\end{equation*}
For $0\leq t\le \sigma_N^2\wedge T_b$, we may use Lemma \ref{rem:sum-Cov-crossing-edge} and then integrate in time to deduce that
\begin{equation*}
     \langle\phi_{\bbT_N,\widetilde{\frak p}(\cdot)}[\scrC^0_{\bullet}(\calR)]\rangle_t\lesssim \Xi(N)^{-1}\cdot t.
\end{equation*}
Similarly, using Remark \ref{rem:sum-Cov-arms-edge}, one gets that for any $0\leq t\le \sigma_N^2\wedge T_b$, any $e\in E_N$, $R\le N/4$ and any $\#\in\{4,3^+,3^-,3^{\textrm{lf}},3^{\textrm{rg}}\}$ one has
\begin{equation*}
   \langle\phi_{\bbT_N,\widetilde{\frak p}(\cdot)}[\calA^{(e)}_{\#}(R)]\rangle_t\lesssim \pi_{\#,p_c}(R)^2\Xi(N)^{-1}\cdot t.
\end{equation*}

\paragraph{Step 2: Estimates for $\Cov_{\widetilde{\frak p}(t)}(\omega_e,\omega_f)$.}
Fix two edges $e,f\in E_N$ at distance $R$ from each other. The process $\Cov_{\widetilde{\frak p}(t)}(\omega_e,\omega_f)$ is not a local martingale, so we decompose it into $M^{e,f}_t+A^{e,f}_t$ where $M^{e,f}_t$ is a local martingale started at $0$ and $A^{e,f}_t$ is an adapted process with bounded variation.

\paragraph{Step 2.1: Estimating the bounded variation term $A^{e,f}_t$.}
Using the Itô's formula gives the derivative of $A^{e,f}_t$. Note that most terms cancel out, and the only term left is
\begin{equation*}
    \frac{\dd}{\dd t}A^{e,f}_t = -\sum_{g\in E}\frac 1{\widetilde{\mathbf p_g}(t)^2(1-\widetilde{\mathbf p_g}(t))^2}\Cov_{\widetilde{\frak p}(t)}(\omega_e,\omega_g)\Cov_{\widetilde{\frak p}(t)}(\omega_f,\omega_g).
\end{equation*}
For $0\leq t\le \sigma_N^2\wedge T_b$, we may estimate this term by \eqref{eq:sum_cov_g} and integrating in $t$, which gives
\begin{equation*}
 |A^{e,f}_t-A^{e,f}_0|\lesssim \Delta_{p_c}(R)^2\frac{t}{\Xi(N)}.
\end{equation*}

\paragraph{Step 2.2: Estimating the local martingale term $M^{e,f}_t$.} Itô's formula gives that
\begin{equation*}
    \frac{\dd}{\dd t}\langle M^{e,f}\rangle_t = \sum_{g\in E}\frac1{\mathbf p_g(t)^2(1-\mathbf p_g(t))^2}\kappa^{\widetilde{\frak p}(t)}_3(e,f,g)^2.
\end{equation*}
Applying \eqref{eq:sum_kappa3_squared} and integrating up to $0\leq t \leq \sigma_N^2\wedge T_b$ gives that
\begin{equation*}
    \langle M^{e,f}\rangle_t\lesssim \Delta_{p_c}(R)^4\frac{t}{\Xi(N)}.
\end{equation*}
\paragraph{Step 3: Concluding via Gronwall's lemma.} The above estimates allow to concludes as in Propositions \ref{prop:L2-case} and \ref{prop:L3-case}. Still, one should check that the bond parameters $\widetilde{\mathbf p}_e$ indeed remain close to $\mathbf p_e$, by bounding the terms $\eps^{(e)}(t)$. Assume now that $T_b>\sigma_N^2$. This can be done by noticing that as long as $\mathbf p_e(t)$ remains away from $0$ and $1$, one has for $0\leq t \leq \sigma_N^2$
\begin{align*}
    \left|\frac{\dd}{\dd t}\eps^{(e)}(t)\right|&\lesssim |\mathbf p_e(t)-p_c|+\left|\phi_{\widetilde{\frak p}(t)}[\omega_e]-\frac12\right| \\
    &\lesssim |B^{(e)}_t|+t+|\eps^{(e)}(t)|+\left|\phi_{\widetilde{\frak p}(t)}[\omega_e]-\frac12\right|
\end{align*}
One can bound the linear term $t\le \sigma_N^2$, while, with probability at least $\exp(-c\alpha^2)$, one can bound $\sup_{t\le \sigma_N^2}|B_t^{(e)}|\le \alpha\cdot \sigma_N$ for all edges $e \in E_N$. Repeating Step $1$ of the present proof ensures that with probability at least $1-\exp(-c\alpha^2)$, we can bound
$\sup_{t\le \sigma_N^2}\left|\phi_{\widetilde{\frak p}(t)}[\omega_e]-\frac12\right|$ by $\alpha\cdot \Xi(N)^{-1/2}\sigma_N$. Therefore, with probability at least $1-\exp(-c\alpha^2)$, one has for any $0\leq t \leq \sigma_N^2 $,
\begin{equation*}
     \left|\frac{\dd}{\dd t}\eps^{(e)}(t)\right|\lesssim |\eps^{(e)}(t)|+ \alpha\cdot\Xi(N)^{-1/2}\sigma_N.
\end{equation*}
Using Gronwall's lemma up to time $t=\sigma_N^2\ll 1$ ensures that, with probability at least $1-\exp(-c\alpha^2)$, for any $0\leq t \leq \sigma_N^2$,
\begin{equation*}
    |\widetilde{\mathbf p}_e-\mathbf p_e|=|\eps^{(e)}(t)|\lesssim \frac{\sigma_N^3}{\Xi(N)^{1/2}}.
\end{equation*}
\end{proof}

\begin{proof}[Proof of Theorem \ref{thm:non-independent-bonds}]
As in the proof of Theorem \ref{thm:extension_scaling_window_drift}, we embed the $\mathbf p_e$ into Brownian motions with drift $\mathbf p_e(t)$, using the stopping times $T^{(e)}$. We then define the $\widetilde{\mathbf p}_e(t)$ by setting
\begin{equation*}
    \widetilde {\mathbf p}_e(t)=\mathbf p_e(t)+\eps^{(e)}(t\wedge T^{(e)})
\end{equation*}
where the $\eps^{(e)}$ are determined by $\eps^{(e)}(0)=0$ and the stochastic differential equation \eqref{eq:eps_SDE}. We can then run the proof as in that of Theorem \ref{thm:extension_scaling_window_drift}, using the crude estimate

\begin{align*}
    \mathbf P_N\Big[\sup_e T^{(e)}>\Sigma_{\frak p}^{1/2}\Xi(N)^{3/4}\Big]&\lesssim N^2\left(\exp\left(-c'\left(\frac{\Sigma_{\frak p}^{1/2}\Xi(N)^{3/4}}{\Sigma_{\frak p}^2}\right)^{1/3}\right)+\exp\Big(-c'\Sigma_{\frak p}^{-1/2}\Big)\right)\\
    &\lesssim \exp\left(-\frac 12c'\left(\frac{\Xi(N)^{1/2}}{\Sigma_{\frak p}}\right)^{1/2}\right),
\end{align*}
where we used the fact that $\Xi(N)\le 1$ and that $\Sigma_{\frak p}\le c\cdot \Xi(N)^{1/2}\log(N)^{-2}$ in the second line. Finally, the bound on $|\eps^{(e)}(t)|$ is carried out as in the previous proof, except here we are only considering one random sum of the $T^{(f)}$, so we use Lemma \ref{lem:tail_sum_stretched_exp} to bound
\begin{equation*}
    \mathbf P_N\left[\sum_{f\in E_N}\Cov_{p_c}(\omega_e,\omega_f)^2T^{(f)}>\alpha\cdot\Xi(N)^{-1}\Sigma_{\frak p}^2\right]\lesssim \exp\left(-\tilde c\cdot \alpha^{1/3}\right)
\end{equation*}
and conclude the proof from there.
\end{proof}

\begin{remarks}
    Recall that assuming $\Delta_{p_c}(R)\lesssim R^{-1/2-c}$ allows a polynomial improvement in the corrected environment compared to the independent one. We would like to argue here that in this case, the modified variables $(\widetilde {\mathbf p}_e)_{e\in E_N}$ are actually decorrelated in space, in the sense that if $e$ and $f$ are two edges far apart in $\bbT_N$, the variables $\widetilde {\mathbf p}_e$ and $\widetilde {\mathbf p}_f$ are essentially independent. We do not state or prove any result rigorous here. The variable $\widetilde {\mathbf p}_e$ can be written as $\mathbf p_e+\eps^{(e)}_{T^{(e)}}$, and we now focus on the impact of the correction term $\eps^{(e)}_{T^{(e)}}$, coming from the SDE \eqref{eq:eps_SDE}. In that equation, the correlations between different $\widetilde{\mathbf p}_e$ variables come from the term $\phi_{\bbT_N,\widetilde{\frak p}(t)}[\omega_e]$. But this term is simply a local martingale, given explicitely by Itô's formula
    \begin{equation*}
        \dd\phi_{\bbT_N,\widetilde{\frak p}(t)}[\omega_e]=\sum_{g\in E_N}\frac1{\widetilde{\mathbf p}_g(1-\widetilde{\mathbf p}_g)}\Cov_{\widetilde{\frak p}(t)}(\omega_e,\omega_g)\dd B^{(g)}_t.
    \end{equation*}
    Now assuming that some RSW property and all stability of mixing rates and events hold, the coefficient in front of $\dd B^{(g)}_t$ is of order $\Delta_{p_c}(d(e,g))^2$. Then the derivative of the bracket of this local martingale is of order $\sum_{r\le N}r\Delta_{p_c}(r)^4\asymp 1$ (this last estimate assumes that $\Delta_{p_c}(R)\lesssim R^{-1/2-c}$). Furthermore, the contribution from edges $g$ at distance at least some constant $R$ from $e$ to this sum is of order $\sum_{R\le r\le N}r\Delta_{p_c}(r)^4\asymp R^2\Delta_{p_c}(R)^4$, which goes to $0$ as $R\to \infty$. Therefore, the local martingale $\phi_{\bbT_N,\widetilde{\frak p}(t)}[\omega_e]$ will mainly depend on edges close to $e$, and in the same way the local martingale $\dd\phi_{\bbT_N,\widetilde{\frak p}(t)}[\omega_f]$ will mainly depend on edges close to $f$. Hence, if $e$ and $f$ are far apart, the process $\phi_{\bbT_N,\widetilde{\frak p}(t)}[\omega_e]$ is essentially independent of $\mathbf p_f(t)$ and of $\phi_{\bbT_N,\widetilde{\frak p}(t)}[\omega_f]$. Plugging this back into \eqref{eq:eps_SDE}, one gets that the processes $\widetilde{\mathbf p}_e(t)$ and $\widetilde{\mathbf p}_f(t)$ are essentially independent.
\end{remarks}

\section{The special case of percolation}\label{sec:percolation}
\subsection{Near-critical random bond percolation on the square lattice}\label{sub:percolation-square}
Throughout this section, we will be working with $q=1$. We will prove Theorem \ref{thm:percolation-square} for bond percolation, first in the special case of i.i.d.\ Gaussian variables whose law is given by 
\begin{equation}\label{eq:Gaussian-case-perco}
    \mathbf{p}_e \overset{(d)}{=} \calN\left(\frac{1}{2},\sigma_N^2\right),
\end{equation}
and then in the general case using once again the Skorokhod embedding theorem. The reason for treating independent bond percolation separately (i.e.\ separating the proofs for FK percolation of parameter $q=1$ and the proofs for $1<q\leq 4$ treated in Sections \ref{sec:naive-random-bonds} and \ref{sec:non-centred-gaussian}) is that the near-critical window can be extended up to $O(\log(N)^{-2})$, almost reaching macroscopic deformations. The previous section showed that for $1<q\leq 4$, one can pass from a critical window $\textrm{W}(N,q)$, which is expected to decay like $N^{-\nu+o(1)}$ to near-critical random environments in the larger window $\widetilde{\textrm{W}}(N,q)^{\frac{1}{3}}=N^{-\frac{\nu}{3}+o(1)}$ with high $\mathbf{P}_N$-probability. For Bernoulli percolation, the CLE(6) conjecture implies that for $q=1$, the critical window is $\textrm{W}(N,q=1)=N^{-\frac{3}{4}+o(1)}$ (which has been rigorously proven for the triangular lattice in \cite{WerSmi,camia2006sle}). The fact that the near-critical window in a random environment is much larger comes as no surprise, as independence shows that \emph{annealed} random bond percolation on a random environment where each bond variable is centred at $p_c(1)=\frac{1}{2} $ has \emph{exactly} the law of the critical percolation of parameter $p_c(1)=\frac{1}{2} $. Moreover, using some hands-on noise sensitivity argument, we prove in Section \ref{sub:noise-sensitivy} that one can use even larger random variables centred at $\frac{1}{2}$ and whose variance is \emph{is not bounded from above as} $N\to\infty$, while keeping the large scale strong box-crossing property with high probability. Unlike the approach used up to here, this result doesn't assert that crossing probabilities are preserved at every scale, but rather that they are preserved in the large scale regime. 

In view of the next proposition, we remark that in the case of Gaussian variables $\mathbf p_e$, the scaling $O(\log(N)^{-1/2})$ corresponds exactly to what is required to ensure that none of the $N^2$ variables $\mathbf p_e$ leave the interval $[0,1]$. It is remarkable that, up to a constant factor, keeping the $\mathbf p_e$ from deviating too much from the critical value allows us to keep the RSW property at every scale, as detailed in Proposition \ref{prop:log-case-perco} presented below. As we did previously, we first prove Theorem \ref{thm:percolation-square} in a simplified Gaussian case and then pass to the general case using Skorokhod embedding theorem. This reads in the following proposition.
\begin{proposition}\label{prop:log-case-perco}
      Assume that $q=1$, and that the random environment $(\mathbf{p}_e)_{e\in E_N}$ under $\mathbf{P}_N$ is given by i.i.d\ Gaussian variables satisfying \eqref{eq:Gaussian-case-perco} with some uniform parameter $\sigma_N$. Then there exists some $\delta>0$ and some $c>0$ such that for any $N\geq 1 $, if $\sigma_N\le c\cdot\log(N)^{-1/2}$, then 
    \begin{equation*}
        \mathbf{P}_N\Bigg[ \phi_{\bbT_N,\frak p}\in \textrm{RSW}(\delta,N)\Bigg] >1-O\left(\exp\Big(-c\cdot \sigma_N^{-2}\Big)\right).
    \end{equation*}
\end{proposition}

We keep the notations of the previous sections and implement a similar strategy. The main reason behind the improvement from a polynomial to a logarithmic bound lies in the fact that the finite variation processes involved in Proposition \ref{prop:L2-case} vanish for percolation due to the independence between edges. Therefore, as it represented the dominant term in the stochastic derivatives of crossings and edge influences, one should only control here the local martingale term. Set once again for $t\geq 0$ the process $\frak p(t):=(\mathbf{p}_e(t))_{e\in E_N}$ formed of i.i.d\ Brownian motions 
\begin{equation*}
    \mathbf{p}_e(t) := p_c(q)+B^{(e)}_t,
\end{equation*}
and still denote by $\mathcal{F}^{N}_t$ the naturally associated filtration. We intend to prove once again that if $t\lesssim \log(N)$, all the measures $s\mapsto (\phi_{\bbT_N,\frak p(s)})_{s\leq t}$ remain within the $\textrm{RSW}(\delta,N)$ class with high probability. In the case of percolation, the independence of edges implies that the mixing rates $\Delta^{(e)}(r)$  all vanish, making this notion irrelevant. One should therefore go back to the original study of Kesten \cite{kesten1987scaling} who studied the scaling relation for percolation via some refined understanding of the $4$ arm exponent $\pi_4(r)$ at each edge, which encodes the pivotality of edges. We refer to Section \ref{sub:geometric-estimates} where this can be explicitely seen in the estimates. Since there are no mixing rates for percolation, the breaking time $T_b=T_b(\delta)$ is given simply by
\begin{equation*}
    T_b:=\inf\Bigg\{t\ge 0\Bigg|\phi_{\bbT_N,\frak p(t)}\not\in \textrm{RSW}(\delta,N)\cap\left(   \bigcap_{\#}\textrm{Stab}_{\#}(\delta,N)\right)\Bigg\},
\end{equation*}
where once again the intersection is taken on $\#\in\{4,3^+,3^-,3^{\textrm{rg}},3^{\textrm{lf}}\}$. As long as $0\leq t\leq T_b$, the probability $\mathbf{p}_e(t) $ that a the $e$ is open under the measure $\phi_{\bbT_N,\frak p(t)}$ remains bounded away from $0$ and $1$. One can once again reformulate Proposition \ref{prop:log-case-perco} as the existence of some $\delta>0$ and some $c>0$ such that for $\sigma_N\le c\cdot \log(N)^{-1/2}$,
\begin{equation}
    \mathbf P\Big[T_b< \sigma_N^2 \Big]\lesssim \exp\Big(-c\sigma_N^{-2}\Big).
\end{equation}
Once again, Proposition \ref{prop:log-case-perco} will follow from the following lemma, which is a an adaptation of Lemma \ref{lem:stability-RSW}.

\begin{lemma}
\label{lem:stability-RSW-perco}
    Assume that $\delta$ is chosen small enough. There exist some positive constant $c=c(\delta)>0$ and $O_{\delta}$ such that if $\sigma_N^2\le c\cdot \log(N)^{-1}$, one has
\begin{equation*}
    \mathbf P_N \Bigg[ \forall \ 0\leq  s\leq \sigma_N^2\wedge T_b, \ \phi_{\bbT_N,\frak p(s)}\in \textrm{RSW}(2\delta,N) \Bigg]\geq 1-O_\delta\Bigg( \exp\left(-c \cdot \frac{1}{\sigma_N^2}\right) \Bigg),
\end{equation*}
\begin{equation*}
    \mathbf P_N \Bigg[ \forall \ 0\leq  s\leq\sigma_N^2\wedge T_b, \ \phi_{\bbT_N,\frak p(s)}\in \bigcap_{\#}\textrm{Stab}_{\#}(2\delta,N) \Bigg]\geq 1-O_\delta\Bigg( \exp\left(-c \cdot \frac{1}{\sigma_N^2}\right) \Bigg).
\end{equation*}
\end{lemma}
\begin{proof}
We will focus here in the proof of the first estimate, as one can adapt the proof to the second estimate as discussed at the end of the proof of Lemma \ref{lem:stability-RSW-L3}. Note that we do not need the events $\scrC^0_{\bullet}(\calR)$ and $\scrC^{\star1}_{\bullet}(\calR)$ in this section as boundary conditions have no effect in percolation, so that the RSW property is given simply by estimates on the crossing events $\scrC_{\bullet}(\calR)$ and $\scrC^{\star}_{\bullet}(\calR)$

\paragraph{Step 0: Canonical decomposition of the semimartingale.}
Fix a $2$ by $1$ rectangle $\calR\in \bbT_N$. Using the decomposition of semimartingales, one can write $\phi_{\bbT_N,\frak p(t)}[\scrC_{\bullet}(\calR)]=M^\calR_t+A^\calR_t$ where $M^\calR$ is a local martingale started at $0$ and $A^\calR$ is an adapted finite variation process, whose exact expression are given by
\begin{align*}
	\dd M^{\calR}_t&=\sum_{e\in E_N}\frac 1{\mathbf{p}_e(t)(1-\mathbf{p}_e(t))}\Cov_{\frak p(t)}(\scrC_{\bullet}(\calR),\omega_e)\dd B^{(e)}_t \\
	\dd A^\calR_t &= \sum_{e\in E_N}\frac1{\mathbf{p}_e(t)(1-\mathbf{p}_e(t))^2}\left(1-\frac 1{\mathbf{p}_e(t)}\phi_{\bbT_N,\frak p(t)}[\omega_e]\right)\Cov_{\frak p(t)}(\scrC_{\bullet}(\calR),\omega_e)\dd t
\end{align*}

\paragraph{Step 1: Evaluating the order of magnitude of $A^\calR_t$.} 
In the case of percolation, the edges parameter are independent therefore for any $e\in E_N$ one has 
\begin{equation*}
	\phi_{\bbT_N,\frak p(t)}[\omega_e]=\mathbf{p}_e(t),
\end{equation*}
and so $\dd A^{\calR}_t = 0$.

\paragraph{Step 2: Evaluating the order of magnitude of $M^\calR_t$.}
We estimate once again the rate of growth of the local-martingale $M^\calR_t$ via some control of its bracket, which is given by
\begin{equation*}
    \dd \langle M^\calR\rangle_t = \sum_{e\in E_N}\left(\frac 1{\mathbf{p}_e(t)(1-\mathbf{p}_e(t))}\Cov_{\frak p(t)}(\scrC_{\bullet}(\calR),\omega_e)^2\right)\dd t.
\end{equation*}
Using again \eqref{eq:bound-away-0-1}, one deduces that
\begin{equation}
    \frac{\dd \langle M^\calR\rangle_t}{\dd t}\asymp \sum_{e\in E_N}\Cov_{\frak p(t)}(\scrC_{\bullet}(\calR),\omega_e)^2.
\end{equation}
Now Lemma \ref{lem:near-critical_estimates} allows us to use the critical estimates of Lemma \ref{lem:geometric-estimates_perco} to get
\begin{equation}\label{eq:rate-growth-M-perco}
	\forall\ 0\leq t \leq \sigma_N^2\wedge T_b,\quad \frac{\dd \langle M^{\calR}\rangle_t}{\dd t}\lesssim 1.
\end{equation}

\paragraph{Step 3: Concluding via large deviations estimates for $M^\calR_t$.}
We conclude the proof by some straightforward large deviation estimates for the $M^\calR$ as in the previous section, we omit the details. Once again, adapting for arm events poses no difficulties.
\end{proof}

\begin{proof}[Proof of Theorem \ref{thm:percolation-square}]
Once Proposition \ref{prop:log-case-perco} is proven, one can apply once again the Skorokhod embedding theorem as in Section \ref{sec:naive-random-bonds} to pass from Brownian motions to a sequence of general random variables satisfying $(\star)^{q=1}_{N} $. We simply point out the intermediate estimate which is used
\begin{equation*}
    \mathbf P_N\left[\sup_e T^{(e)}>\Sigma_{\frak p}^{1/2}\right]\lesssim \exp(-c\cdot \Sigma_{\frak p}^{-1/2}).
\end{equation*}
\end{proof}

\subsection{Random bond percolation via noise sensitivity}\label{sub:noise-sensitivy}
The goal of this section is to extend, in the large scale limit the RSW box crossing property to deformations from the critical point which are \emph{macroscopic}, meaning that one \emph{doesn't} need to scale the maximal variance $\Sigma_{\frak{p}}$ inside $\bbT_N$ to $0$ as $N\to \infty $, working in \emph{true random bond context}. From the previous section, this theorem comes as no surprise. Indeed, we showed that, in the case of Brownian deformations centred at the critical point, as long as the model stays critical \emph{at every scale}, the stochastic derivative \eqref{eq:rate-growth-M-perco} is \emph{bounded} independently from $N$. The only reason that one needs additional logarithmic factors is to brutally ensure that \emph{none} of the probabilities $\phi_{\frak p}[\scrC_{\bullet}(\calR)]$ have deviated far from their original value, while one can expect that even if few boxes became too off-critical, it will not have a big influence on the criticality of the large scale model. Therefore, one should find some kind of renormalisation argument to bypass the relevance of criticality being preserved at each scale. The idea to use noise sensitivity in our context is not so surprising. In fact, the key estimate in our proof for showing that RSW holds up to windows of logarithmic size depends essentially on the estimate, for all $2$ by $1$ rectangles $\calR$,
\begin{equation*}
    \sum_{e\in \bbZ^2}\Cov_{p_c}(\scrC_{\bullet}(\calR),\omega_e)^2\lesssim 1
\end{equation*}
where we point out that the covariance is zero when $e\not\in \calR$. Using half-plane two arm events (see \cite[Chapter VI.1]{garban2012noise}), one can actually also show with the arguments of Section \ref{sub:geometric-estimates} that as the width of $\calR$ goes to infinity,
\begin{equation}\label{eq:estimate-noise-sensitivity}
    \sum_{e\in \bbZ^2}\Cov_{p_c}(\scrC_{\bullet}(\calR),\omega_e)^2\to 0.
\end{equation}
This is exactly the usual condition used to show that the event $\scrC_{\bullet}(\calR)$ is asymptotically noise sensitive as the side length of $\calR$ goes to infinity, see \cite[Theorem 1.3]{benjamini1999noise}. The main difficulty with our method is that it requires that the RSW property is preserved at all scales and in all boxes, which is not enough to prove that \eqref{eq:estimate-noise-sensitivity} is preserved up to macroscopic deformations. Still, one wants to use this intuition, which hints that the missing renormalisation argument can be derived via some noise sensitivity of Bernoulli percolation, heavily studied in \cite{benjamini1999noise,garban2013pivotal,tassion2023noise}. We end up using a rather hands-on argument to conclude. 
Let us start with a sequence of random variables $\frak{p}=(\mathbf p_e)_{e\in E(\mathbb{Z}^2)}$, chosen under some probability measure $\mathbf{P}$ such that 
\begin{enumerate}
    \item The variables $\mathbf p_e$ are mutually independent.
    \item For each edge $e$, one has $\bbE[\mathbf p_e]=\frac12$
    \item There is some $\eps>0$ such that almost surely, all of the $\mathbf p_e$ belong to $[\eps,1-\eps]$.
\end{enumerate}
We denote by $\phi_{\frak{p}}$ the associated \emph{full-plane} percolation measure. We are now in position to state the asymptotic strong box crossing property for crossings. Denote by $S_N$ a square of side-length $N$ in the plane.
\begin{proposition}
\label{prop:random-bond-perco}
    Assume that the variables $\frak p$ satisfy the above conditions. Then as $N\to\infty$, one has the following convergence in probability
    \begin{equation*}
        \phi_{\frak p}[\scrC_h(S_N)] \xlongrightarrow{(\bbP)}\frac12.
    \end{equation*}
\end{proposition}
We start by proving a particular case of the result above, which is actually just a rewording of noise sensitivity of crossing events for percolation.

\begin{lemma}[A simple consequence of \cite{benjamini1999noise,garban2012noise}]\label{lem:random-bond-perco}
    Assume $\frak p^{\eps}=(\mathbf p_e^{\eps})_e$ has the specific following law: the $\mathbf p_e^{\eps}$ are i.i.d. variables equal to $\eps$ or $1-\eps$, each with probability $\frac 12$. Then the result of proposition \ref{prop:random-bond-perco} holds.
\end{lemma}

\begin{proof}
    Consider a realisation $\frak p^{\eps}$ of the environment. We define a percolation configuration $\mathbf \omega_{env}$ by setting $\omega_{env}(e)$ to be open if $\mathbf p_e^{\eps}=1-\eps$, and closed otherwise. Note that $\omega_{env}(e)$ is a deterministic function of the environment $\frak p^{\eps}$, but its annealed law is that of critical percolation, because each $\mathbf p^{\eps}_e$ is independently $\eps$ or $1-\eps$ with probability $\frac12$.

    Now consider the following way of sampling the percolation configuration $\omega$ from the random bond percolation $\phi_{\frak p^{\eps}}$: one starts with $\omega_{env}$, and for each edge $e$, one flips the state of edge $e$ with probability $\eps$. Indeed, if for example $\omega_{env}(e)=0$, this means that $\mathbf p_e=\eps$, and so $\omega(e)=0=\omega_{env}(e)$ with probability $1-\eps$, and $\omega(e)=1=1-\omega_{env}(e)$ with probability $\eps$, independently at each edge. Then \cite[Theorem 1.2]{benjamini1999noise} states that the crossing events $\scrC_h(S_N)$ are asymptotically noise sensitive as $N\to \infty$. According to \cite[Exercise IV.6]{garban2012noise}, this is equivalent to saying that $\mathbf P[\omega\in \scrC_h(S_N)|\omega_{env}]-\mathbf P[\omega\in\scrC_h(S_N)]$ converges to $0$ in probability. But under the annealed probability measure $\mathbf P$, $\omega$ is simply a Bernoulli percolation with parameter $\frac12$, so $\mathbf P[\scrC_h(S_N)]\to \frac12$ as $N\to\infty$. This exactly gives the desired result.
\end{proof}

We now pass to the proof of Proposition \ref{prop:random-bond-perco} for a general set of independent random variables.
\begin{proof}
    Let $\frak p=(\mathbf p_e)_{e\in E(\mathbb{Z}^2)}$ be random variables satisfying the above condition. We will write the environment $\frak p$ as some averaged version of the environment $\frak p^{\eps}$. Given a realisation $\frak p=(\mathbf p_e)_{e\in E(\mathbb{Z}^2)}$ of the environment, one can define random variables $(\mathbf q_e)_{e\in E(\mathbb{Z}^2)}$ as follows: for each edge $e$, take $\mathbf q_e$ to be
    \begin{equation*}
        \mathbf q_e :=
        \begin{cases}
            \eps & \text{ with probability } \frac{1-\eps-\mathbf p_e}{1-2\eps}\\
            1-\eps & \text{ with probability } \frac{\mathbf p_e-\eps}{1-2\eps},
        \end{cases}
    \end{equation*}
in a way that the variables $\mathbf q_e$ are independent, both unconditionally and conditionally on $\frak p$. Note that this definition makes sense precisely because $\mathbf p_e$ is almost surely in $[\eps;1-\eps]$, and it is straightforward to see that $\bbE[\mathbf q_e|\frak p]=\mathbf p_e$. Consider the random environment $\frak q=(\mathbf q_e)_{e\in E(\mathbb{Z}^2)}$. Not conditioning anymore on the random variables $\frak p$, this environment $\frak q$ has independent bonds taking values in $\{\eps,1-\eps\}$, and it satisfies, for each $e\in E(\mathbb{Z}^2)$, $\bbE[\mathbf q_e]=\bbE[\bbE[\mathbf q_e|\frak p]]=\bbE[\mathbf p_e]=\frac12$. Therefore,   $\frak q$ actually has the law of the environment $\frak p^{\eps}$ given in Lemma \ref{lem:random-bond-perco}. This allows to conclude that for any $\gamma>0$,
    \begin{equation*}
        \mathbf P\Bigg[\left|\phi_{\frak q}[\scrC(S_N)]-\frac12\right|>\gamma\Bigg] \underset{N\to \infty}{\longrightarrow}  0,    \end{equation*}
    which, by Markov's inequality, implies that
    \begin{equation*}
        \mathbf P\Bigg[\mathbf P\Bigg[\left|\phi_{\frak q}[\scrC(S_N)]-\frac12\right|>\gamma\ \Bigg|\frak p\Bigg]>\gamma \Bigg]\underset{N\to \infty}{\longrightarrow}  0. 
     \end{equation*}
To conclude, it is enough to notice that to sample the percolation configuration $\omega$ under $\phi_{\frak p}$, one can first sample $\frak q$ from $\frak p$, and then sample $\omega$ from $\phi_{\frak q}$ (as $\bbE[\mathbf q_e|\frak p]=\mathbf p_e$). Therefore
    \begin{equation*}
        \phi_{\frak p}[\scrC(S_N)]=\bbE[\phi_{\frak q}[\scrC(S_N)]|\frak p],
    \end{equation*}
which allows to write
    \begin{align*}
        \left|\phi_{\frak p}[\scrC(S_N)]-\frac12\right|&\le \bbE\Bigg[\left|\phi_{\frak q}[\scrC(S_N)]-\frac12\right|\ \Bigg| \frak p\Bigg]\\
        &\le \gamma + \mathbf P\Bigg[\left|\phi_{\frak q}[\scrC(S_N)]-\frac12\right|>\gamma\ \Bigg|\frak p\Bigg].
   \end{align*}
This allows to conclude that
    \begin{equation*}
        \mathbf P\Bigg[\left|\phi_{\frak p}[\scrC(S_N)]-\frac12\right|>2\gamma\Bigg]\le \mathbf P\Bigg[\mathbf P\Bigg[\left|\phi_{\frak q}[\scrC(S_N)]-\frac12\right|>\gamma\ \Bigg|\frak p\Bigg]>\gamma \Bigg]\underset{N\to \infty}{\longrightarrow}  0.
    \end{equation*}
 This concludes the proof.
\end{proof}

Let us now move on to proving that Cardy's formula stays true in a random environment.

\begin{proof}[Proof of Theorem \ref{thm:Cardy-random-environment}]
    Fix some $x\in[0,1]$, and consider the crossing event denoted $F_N(x)=\calC([C_N;A_N],[x_N;B_N])$. If $x=0$ or $x=1$, then $\phi_{\frak p}[F_N(x)]$ is identically $0$ or $1$ respectively. Assuming that $x\in(0,1)$, we claim that as $N\to\infty$, the events $F_N(x)$ are asymptotically noise-sensitive. Indeed, one can follow the same computations as in \cite[Chapter VI.1]{garban2012noise} by considering the quadrilateral $A_NB_Nx_NC_N$. The key point is that all of the sides of this quadrilateral are straight lines, so one can use half-plane three arm events, and the corners have an opening angle of $\pi$ at $x_N$ and $\pi/3\le \pi$ at $A_N,B_N,C_N$ (see Figure \ref{fig:cardy} in the introduction), so one can use half-plane two arm events to control pivotality in the corners. Then one can run the proof of Proposition \ref{prop:random-bond-perco} to get that as $N\to\infty$,
    \begin{equation*}
        \phi_{\frak p}[F_N(x)]-\mathbf P[F_N(x)]\xlongrightarrow{(\bbP)}0.
    \end{equation*}
    Once again, the annealed model is simply critical Bernoulli percolation, which ensures via Theorem \ref{thm:cardy_formula} that one has $\bbP[F_N(x)]\to 1-x$ as $N\to \infty$. Therefore, one gets that $\phi_{\frak p}[F_N(x)]$ converges to $1-x$ in probability as $N\to\infty$. 
    
    Now we quickly explain how to pass from this ``pointwise" convergence in probability to the ``uniform" convergence in probability of Theorem \ref{thm:Cardy-random-environment}. This is a basic consequence of the fact that the $x\mapsto\phi_{\frak p}[F_N(x)]$ are increasing. Indeed, fix some $K$ and consider $x_i = \frac iK$ for $i=0,1,\ldots,K$. We may take $N$ large enough so that for each $i$, one has \begin{equation*}
        \mathbf P\big[|\phi_{\frak p}[F_N(x_i)]-x_i|>1/K\big]\le \frac1{K(K+1)}.
    \end{equation*}
    so that by a union bound,
    \begin{equation*}
        \mathbf P\big[\forall\ 0\le i\le K,\quad |\phi_{\frak p}[F_N(x_i)]-x_i|\le 1/K\big]>1- 1/K.
    \end{equation*}
    Under this event, for each $x\in (0,1)$, let $i$ be the index such that $x_i\le x<x_{i+1}$. Then we may write
    \begin{equation*}
        x_i-1/K\le\phi_{\frak p}[F_N(x_i)]\le\phi_{\frak p}[F_N(x)]\le \phi_{\frak p}[F_N(x_{i+1})]\le x_{i+1}+1/K.
    \end{equation*}
    Plugging in the fact that $x_{i+1}\le x+1/K$ and $x_i\ge x-1/K$ gives that ${|\phi_{\frak p}[F_N(x)]-x|\le 2/K}$. Therefore, for $N$ large enough,
    \begin{equation*}
        \mathbf P\big[\forall\ x\in [0,1],\quad|\phi_{\frak p}[F_N(x)]-x|\le 2/K\big]>1-1/K
    \end{equation*}
    Taking $N\to\infty$ and $K\to\infty$ sufficiently slowly gives exactly the statement of Theorem \ref{thm:Cardy-random-environment}.
\end{proof}

\appendix

\section{Computations}\label{sec:appendix}

\subsection{Derivatives in the $p_e$ in non translation invariant environments}\label{app:derivative_formulas}
We prove now Proposition \ref{prop:differentiate-bonds} which generalises the formulae used extensively in \cite{FK_scaling_relations} when differentiating the probability of an event with respect to some uniform  parameter $p$, which is in our context no longer the same at each edge. All results in this section are valid on general graphs $G=(V,E)$ and for all $q$, so we will ignore all dependencies on the graph and on $q$. Recall the weight of a configuration $w_{\underline p}(\omega)$ whose expression is given just above the statement of Proposition \ref{prop:differentiate-bonds}.
\begin{proof}[Proof of Proposition \ref{prop:differentiate-bonds}]
First, using the formula for $w_{\underline p}(\omega)$, we may write, for all edges $e\in E_N$ and all $\omega$
\begin{align*}
    \frac{\partial}{\partial p_e}w_{\underline p}(\omega)&=\frac1{p_e(1-p_e)}\ind(\omega_e=1)w_{\underline p}(\omega),\\
    \frac{\partial^2}{\partial p_e^2}w_{\underline p}(\omega)&=\frac1{p_e(1-p_e)^2}\ind(\omega_e=1)w_{\underline p}(\omega).
\end{align*}
For an event $\calS$, one can write directly that 
\begin{equation}
    \phi_{\underline p}[\calS] = \frac{Z_{\underline p}(\calS)}{Z_{\underline p}}, \quad \textrm{ with } Z_{\underline p}(\calS) := \sum_{\omega\in \calS}w_{\underline p}(\omega)
\end{equation}
and $Z_{\underline p}=Z_{\underline p}(\{0,1\}^E)$ is the partition function of the model. For any edge $e\in E_N$, plugging in the expressions for the derivatives of $w_{\underline p}(\omega)$ gives
\begin{align*}
	 \frac{\partial}{\partial p_e}Z_{\underline p}(\calS) &= \frac 1{p_e(1-p_e)}Z_{\underline p}(\calS\cap\{\omega_e=1\}),\\
	\frac{\partial^2}{\partial p_e^2}Z_{\underline p}(\calS) & = \frac 1{p_e(1-p_e)^2}Z_{\underline p}(\calS\cap\{\omega_e=1\}).\\
\end{align*}
In particular
\begin{align*}
	\frac{\partial}{\partial p_e}\phi_{\underline{p}}[\calS]&=\frac{1}{Z_{\underline p}} \frac{\partial}{\partial p_e}Z_{\underline p}(\calS) - \frac{\frac{\partial}{\partial p_e}Z_{\underline p}}{Z_{\underline p}^2} Z_{\underline p}(\calS)\\
	 &=  \frac1{p_e(1-p_e)}\left[\frac 1{Z_{\underline p}}Z_{\underline p}(\calS\cap\{\omega_e=1\})-\frac{Z_{\underline p}(\omega_e=1)}{Z_{\underline p}^2}Z_{\underline p}(\calS)\right] \\
	&  =\frac1{p_e(1-p_e)}\Big[\phi_{\underline{p}}[\calS\cap \{\omega_e=1\}]-\phi_{\underline{p}}[\omega_e]\phi[\calS]\Big]\\
	&=\frac1{p_e(1-p_e)}\textrm{Cov}_{\underline p }(\calS,\omega_e).
\end{align*}
As a corollary, one deduces that for any edges $e_1,e_2,\ldots,e_n,f$, not necessarily distinct, one has
\begin{equation*}
    \frac{\partial}{\partial p_f}\phi_{\underline{p}}\left[\prod_{k=1}^n\omega_{e_i}\right]=\frac1{p_f(1-p_f)}\left(\phi_{\underline{p}}\left[\omega_f\prod_{k=1}^n\omega_{e_i}\right]-\phi_{\underline{p}}[\omega_f]\phi_{\underline{p}}\left[\prod_{k=1}^n\omega_{e_i}\right]\right).
\end{equation*}
Applying this to $\Cov_{\underline p}(\omega_e,\omega_f)=\phi_{\underline p}[\omega_e\omega_f]-\phi_{\underline p}[\omega_e]\phi_{\underline p}[\omega_f]$, one gets for any $e,f,g\in E_N$,
\begin{align*}
    \frac{\partial}{\partial p_g}\Cov_{\underline p}(\omega_e,\omega_f) &= \frac1{p_g(1-p_g)}\Bigg(\phi_{\underline p}[\omega_g\omega_e\omega_f]-\phi_{\underline p}[\omega_g]\phi_{\underline p}[\omega_e\omega_f]-\phi_{\underline p}[\omega_e]\phi_{\underline p}[\omega_g\omega_f] \\
    &\quad  +\phi_{\underline p}[\omega_e]\phi_{\underline p}[\omega_g]\phi_{\underline p}[\omega_f]-\phi_{\underline p}[\omega_f]\phi_{\underline p}[\omega_g\omega_e]
    +\phi_{\underline p}[\omega_f]\phi_{\underline p}[\omega_g]\phi_{\underline p}[\omega_e]\Bigg)\\
    &=\frac1{p_g(1-p_g)}\kappa_3^{\underline p}(e,f,g)
\end{align*}
We compute now the second derivatives of $\phi_{\underline p}[\calS]$.
\begin{align*}
	\frac{\partial ^2}{\partial p_e^2}\phi_{\underline{p}}[\calS]&=\frac{1}{Z_{\underline p}} \frac{\partial ^2}{\partial p_e^2}Z_{\underline p}(\calS) - 2\frac{\frac{\partial}{\partial p_e}Z_{\underline p}}{Z_{\underline p}^2} \frac{\partial}{\partial p_e}Z_{\underline p}(\calS) + \frac{-Z_{\underline p}\frac{\partial^2}{\partial p_e^2}Z_{\underline p}+2(\frac{\partial}{\partial p_e}Z_{\underline p})^2}{Z_{\underline p}^3}Z_{\underline p}(\calS)\\
	& =\frac 2{p_e(1-p_e)^2}\frac{Z_{\underline p}(\calS\cap \{\omega_e=1\})}{Z_{\underline p}}- \frac 2{p_e^2(1-p_e)^2}\frac{Z_{\underline p}(\omega_e=1)}{Z_{\underline p}}\frac{Z_{\underline p}(\calS\cap \{\omega_e=1\})}{Z_{\underline p}}\\
	& \quad -\frac 2{p_e(1-p_e)^2}\frac{Z_{\underline p}(\omega_e=1)}{Z_{\underline p}}\frac{Z_{\underline p}(\calS)}{Z_{\underline p}}+\frac 2{p_e^2(1-p_e)^2}\frac{Z_{\underline p}(\omega_e=1)^2}{Z_{\underline p}^2}\frac{Z_{\underline p}(\calS)}{Z_{\underline p}} \\
	&= \frac 2{p_e(1-p_e)^2}\left(1-\frac1{p_e}\phi_{\underline{p}}[\omega_e]\right)\textrm{Cov}_{\underline p }(\calS,\omega_e).
\end{align*}
And now differentiating twice covariances with respect to an edge bond $p_g$: 
\begin{align*}
	  \frac{\partial ^2}{\partial p_g^2}\Cov(\omega_e,\omega_f)&=\frac{\partial^2}{\partial p_g^2}\phi_{\underline p}[\omega_e\omega_f]-\left(\frac{\partial^2}{\partial p_g^2}\phi_{\underline p}[\omega_e]\right)\phi_{\underline p}[\omega_f]-\phi_{\underline p}[\omega_e]\left(\frac{\partial^2}{\partial p_g^2}\phi_{\underline p}[\omega_f]\right)\\
      & \hspace{4cm} -2\left(\frac{\partial}{\partial p_g}\phi_{\underline p}[\omega_e]\right)\left(\frac{\partial}{\partial p_g}\phi_{\underline p}[\omega_f]\right)\\
	  & =\frac 2{p_g(1-p_g)^2}\left(1-\frac1{p_g}\phi_{\underline p}[\omega_g]\right)\Big[\phi_{\underline p}[\omega_e\omega_f\omega_g]-\phi_{\underline p}[\omega_e\omega_f]\phi_{\underline p}[\omega_g]\Big] \\
	  & \quad  -\frac 2{p_g(1-p_g)^2}\left(1-\frac1{p_g}\phi_{\underline p}[\omega_g]\right)\Big[\phi_{\underline p}[\omega_e\omega_g]-\phi_{\underline p}[\omega_e]\phi_{\underline p}[\omega_g]\Big]\phi_{\underline p}[\omega_f]\\
	  & \quad -\frac 2{p_g(1-p_g)^2}\left(1-\frac1{p_g}\phi_{\underline p}[\omega_g]\right)\Big[\phi_{\underline p}[\omega_f\omega_g]-\phi_{\underline p}[\omega_f]\phi_{\underline p}[\omega_g]\Big]\phi_{\underline p}[\omega_e] \\
	  &  \quad -\frac 2{p_g^2(1-p_g)^2}\Big[\phi_{\underline p}[\omega_e\omega_g]-\phi_{\underline p}[\omega_e]\phi_{\underline p}[\omega_g]\Big]\Big[\phi_{\underline p}[\omega_f\omega_g]-\phi_{\underline p}[\omega_f]\phi_{\underline p}[\omega_g]\Big]\\
	  &=\frac 2{p_g(1-p_g)^2}\left(1-\frac1{p_g}\phi_{\underline p}[\omega_g]\right)\kappa_3^{\underline p}(e,f,g)\\
	  & \quad -\frac 2{p_g^2(1-p_g)^2}\Cov_{\underline p}(\omega_e,\omega_g)\Cov_{\underline p}(\omega_f,\omega_g).
\end{align*}
\end{proof}
We mention that the computations of the first derivatives can be seen as a special case of the following more general result on cumulants of events. Since we will not be needing this result, we do not recall the definition of cumulants and we skip its proof.

\begin{lemma}
Let $A_1,A_2,\ldots,A_n$ be events, and let $e\in E_N$ be an edge. Then one has the identity for all $\underline p$:

\begin{equation*}
    \frac {\partial}{\partial p_e}\kappa^{\underline p}_n(A_1,A_2,\ldots,A_n)=\frac1{p_e(1-p_e)}\kappa^{\underline p}_{n+1}(A_1,A_2,\ldots,A_n,\omega_e).
\end{equation*}
\end{lemma}

\subsection{Proofs of crossing estimates via arm events}\label{app:geometric_estimates}
We collect here the proofs of geometric estimates stated in Section \ref{sub:geometric-estimates}. In these proofs, we do not detail the geometric estimates based on RSW, and for the sake of readability, we will often identify lengths at the same scale, for example instead of writing ``there is a four-arm event from scales $2r$ to $R/2$", we would write ``there is a four-arm event from scales $r$ to $R$".

\begin{proof}[Proof of Lemma \ref{lem:Cov-crossing-edge}]
The first item was already derived in \cite{FK_scaling_relations}, therefore we focus on the second item. Let $e$ be an edge at distance $n\lesssim R$ from the corners of $\calR$ and set $\ell = \textrm{dist}(e,\partial \calR)$. According to \cite[Lemma 5.3]{FK_scaling_relations}, one may write
\begin{equation*}
    \Cov_{p_c}(\scrC_{\bullet}(\calR),\omega_e)\lesssim \sum_{r=1}^{3R}\frac{\Delta_{p_c}(r)}r\phi_{p_c}[\textrm{Piv}_{r,e}(\scrC_{\bullet}(\calR))]
\end{equation*}
where for an event $\calS$, $\textrm{Piv}_{r,e}(\calS)$ is the event that the event that $\calS$ holds and that the box $\Lambda_r(e)$ is pivotal for the event $\calS$. To conclude, one should evaluate $\phi_{p_c}[\textrm{Piv}_{r,e}(\scrC_{\bullet}(\calR))]$, which is done by dichotomy on the value of $r$. We refer to Figure \ref{fig:pivotal_crossing} for an illustration of the following cases. Since there is no risk of confusion, we will not write the subscript $p_c$ for $\Delta,\pi_4,\pi_{3^+}$ in this proof.

\begin{itemize}
	\item \textbf{Case 1:} $n\leq r \leq 3R$. The pivotality event $\textrm{Piv}_{r,e}(\scrC_{\bullet}(\calR))$  implies the existence of two arms, of different parity, going from $\Lambda_r(e)$ to a distance $R$ inside of $\calR$. By standard RSW estimates, the probability of the latter decays polynomially fast, at a speed $\lesssim (r/R)^c$ for some $c>0$.
	\item \textbf{Case 2:} $\ell \le r \le n$. The pivotality event implies the existence of $3$ arms of alternating parity going from $\Lambda_r(e)$ to distance $n$ in $\calR$, and together with two arms of different parity going from $\Lambda_n(e)$ to distance $R$ in $\calR$. Therefore, Remark \ref{rem:arm-estimates-transation-invariant-environment} together with standard RSW estimates ensures that the pivotality event happens with probability at most $(r/n)^2(n/R)^c$ for some $c=c(\delta)>0$.
	\item \textbf{Case 3:} $r\leq \ell $. The pivotality event implies existence of four alternating parity arms from $\Lambda_r(e)$ to distance $\ell$, together with three alternating parity arms from $\Lambda_{\ell}(e)$ to distance $n$ in $\calR$, and finally two alternating arms from $\Lambda_n(e)$ to distance $R$ in $\calR$. Using again Remark \ref{rem:arm-estimates-transation-invariant-environment} and RSW, the pivotality event happens with probability at most $\pi_4(r,\ell)(\ell/n)^2(n/R)^c$.
\end{itemize}

\begin{figure}
    \centering
    \includegraphics[width=0.34\textwidth]{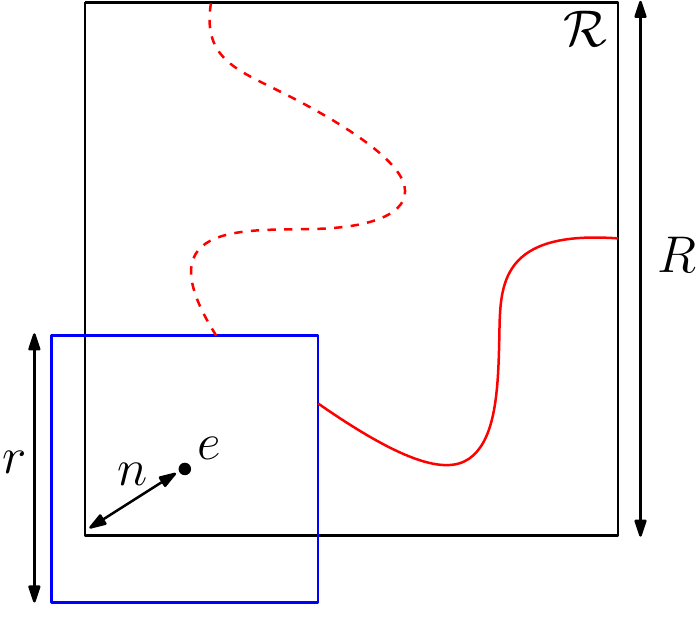}
    \includegraphics[width=0.32\textwidth]{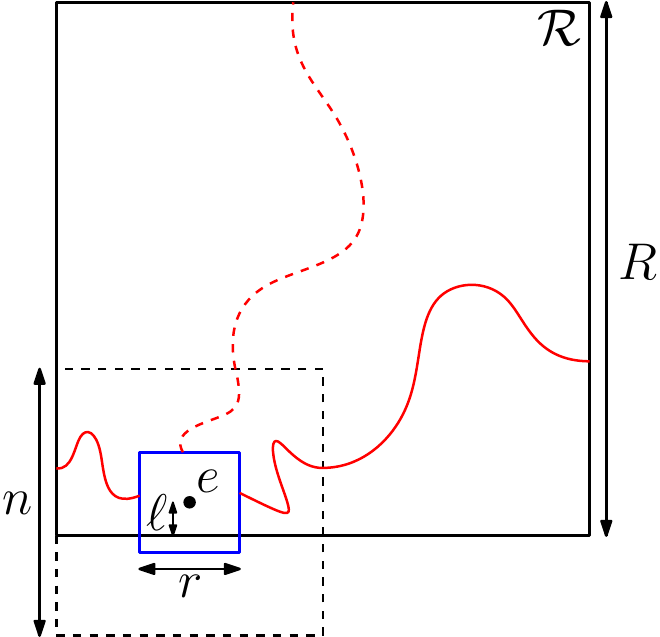}
    \includegraphics[width=0.32\textwidth]{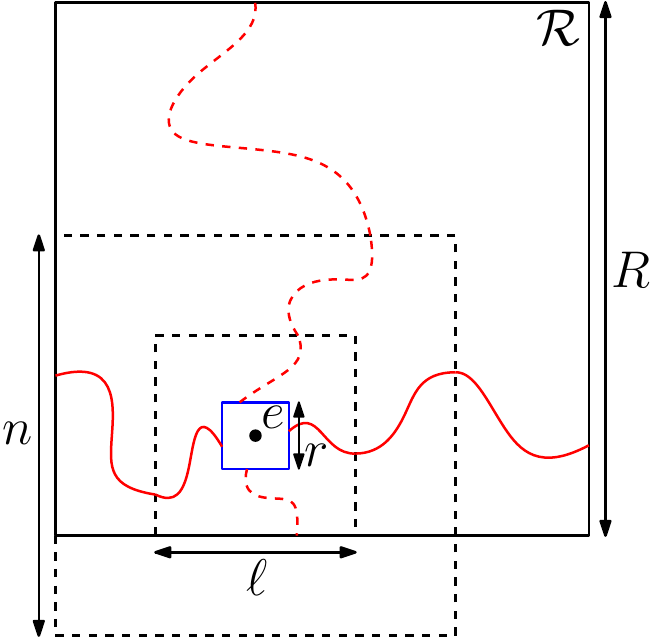}
    \caption{An illustration of the pivotality event $\textrm{Piv}_{r,e}(\scrC_h(\calR))$ in the three cases explicited in the proof. In these figures, all distances denoted by $r,n,\ell$ are to be understood as being of order $r,n,\ell$, and $\calR$ is represented as a square. We only illustrate the case where $e$ is inside $\calR$ but the case when $e$ is outside is analogous.}
    \label{fig:pivotal_crossing}
\end{figure}
This implies that
\begin{align*}
    \Cov_{p_c}(\scrC^{0}_{\bullet}(\calR),\omega_e)&\lesssim\sum_{r=1}^{\ell}\frac{\Delta(r)}r\pi_4(r,\ell)(\ell/n)^2(n/R)^c+\sum_{r=\ell}^n \frac{\Delta(r)}r (r/n)^2(n/R)^c +\sum_{r=n}^R \frac{\Delta(r)}r(r/R)^c \\
    &\asymp\frac{\ell^2\Delta(\ell)}{n^{2-c}R^c}\sum_{r=1}^{\ell}\frac{\pi_4(r,\ell)}{r\Delta(r,\ell)} +\frac 1{n^{2-c}R^c}\sum_{r=\ell}^nr\Delta(r) + \frac 1{R^c}\sum_{r=n}^R \frac{\Delta(r)}{r^{1+c}} \\
    &\lesssim \frac{\ell^2\Delta(\ell)}{n^{2-c}R^c}+\frac {n^c}{R^c}\Delta(n) + \Delta(n) \lesssim \Delta(n).
\end{align*}
Passing from the second to the third line, we used Remark \ref{rem:arm-estimates-transation-invariant-environment} by writing
\begin{itemize}
    \item In the first sum, $\frac{\pi_4(r,\ell)}{\Delta(r,\ell)}\lesssim (r/\ell)^{-c}$.
    \item In the second sum, $\sum_{r=\ell}^nr\Delta(r)\lesssim n^2\Delta(n)$ as $\Delta(r,n)\gtrsim (r/n)^{2-c}$. 
    \item In the third sum $\Delta(r)\le \Delta(n)$ for $r\ge n$.
\end{itemize}
The last inequality uses that $\ell^2\Delta(\ell)\lesssim n^2\Delta(n)$ for $\ell\le n$ and that $n\le R$.
\end{proof}

\begin{proof}[Sketch of proof of Lemma \ref{lem:Cov-arms-edge}]
    We sketch the proof for $\#=3^+$, the case $\#=4$ is analogous (and even simpler). One again one has
 \begin{equation*}
        \Cov_{p_c}(\calA_{3^+}^{(e)}(R),\omega_f)\lesssim \sum_{r=1}^R\frac{\Delta(r)}{r}\phi_{p_c}[\textrm{Piv}_{r,f}(\calA_{3^+}^{(e)}(R))]
 \end{equation*}
Recall (as it was not relevant for the proof of Lemma \ref{lem:Cov-crossing-edge}) that the event $\textrm{Piv}_{r,f}(\calA_{3^+}^{(e)}(R))$ not only requires that the box $\Lambda_r(f)$ is pivotal for the arm event, but also that the associated arm event itself holds, bringing an additional multiplicative factor of small probability. We run once again a dichotomy, depending on the position of $f$, encoded by $n$, and $\ell$, the distance from $f$ to the boundary of the half-plane centred at $e$. This implies that $n\le\ell$. We now estimate $\phi_{p_c}[\textrm{Piv}_{r,f}(\calA_{3^+}^{(e)}(R))]$ depending on $r$ with respect to $\ell$ and $n$. See Figure \ref{fig:pivotal_arms} for an illustration of the following cases. Once again, we do not write the subscript $p_c$ in this proof.
    \begin{itemize}
        \item \textbf{Case 1:} $r\ge n$. One can crudely bound $\phi_{p_c}[\textrm{Piv}_{r,f}(\calA_{3^+}^{(e)}(R))]$ by the probability of $\calA_{3^+}^{(e)}(R)$ arm event, which is (up to constant) $\pi_{3^+}(R)$. This bound is sharp (up to constant) when $f$ is at distance at most $r$ from $e$.
        \item \textbf{Case 2:} $\ell\le r\le n$ while $n=d(e,f)$. In that case, the pivotality of $f$ at scale $r$ implies the existence of a half-plane three arm event from $\Lambda_r(f)$ to distance $\ell$, together with some half-plane three arm event to scale $n$ from $e$, and a half-plane three arm event from scales $n$ to scale $R$ around $e$ and $f$. Therefore, one can bound the probability of the pivotality by $\pi_{3^+}(r,n)\pi_{3^+}(n)\pi_{3^+}(n,R)$.
        \item \textbf{Case 3:} $\ell\le r\le n$ while $n$ is the distance from $f$ to one of the bottom corners of $\Lambda_R(e)\cap \mathbb{H}^{(e)}_N$. Then the pivotality event implies some half-plane three arm event from $\Lambda_r(f)$ to distance $n$, together with a half-plane three arm event from $e$ to distance $R$. Therefore, one can upper bound the probability of the pivotality  by $\pi_{3^+}(r,n)\pi_{3^+}(R)$. By quasi-multiplicativity, this bound is of the same order of magnitude as the one of Case $2$.  Notice that this is of the same order as the previous case.
        \item \textbf{Case 4:} $r\le \ell$. In that case, the pivotality implies the existence a four-arm event from $\Lambda_r(f)$ up to distance $\ell$, together with the event described in case $3$, replacing $r$ by $\ell$. One can then bounds the pivotality probability by $\pi_4(r,\ell)\pi_{3^+}(\ell,n)\pi_{3^+}(n)\pi_{3^+}(n,R)$.
    \end{itemize}

\begin{figure}
    \centering
    \hspace{0.03\textwidth}
    \includegraphics[width=0.45\textwidth]{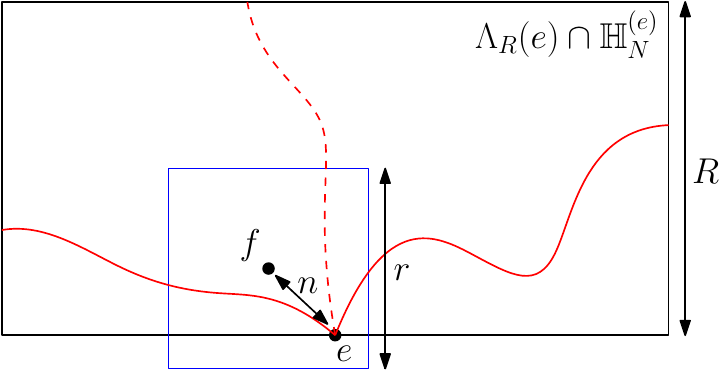}
    \includegraphics[width=0.45\textwidth]{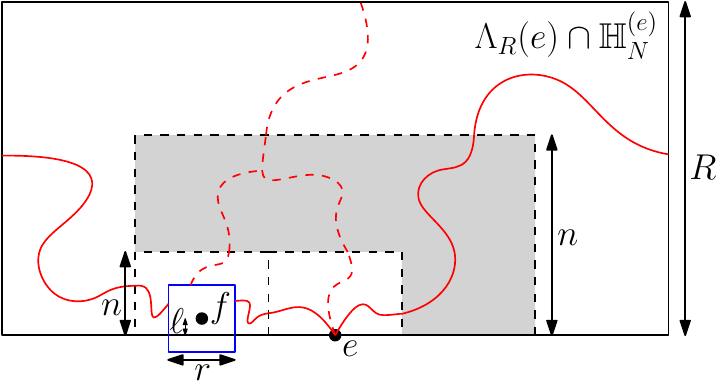}
    \par\bigskip
    \includegraphics[width=0.49\textwidth]{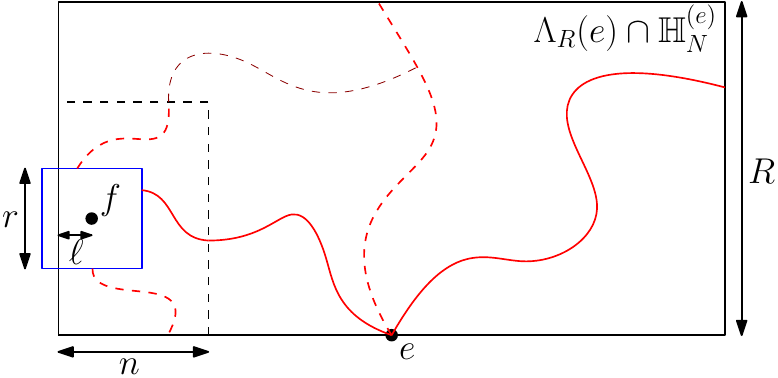}
    \includegraphics[width=0.45\textwidth]{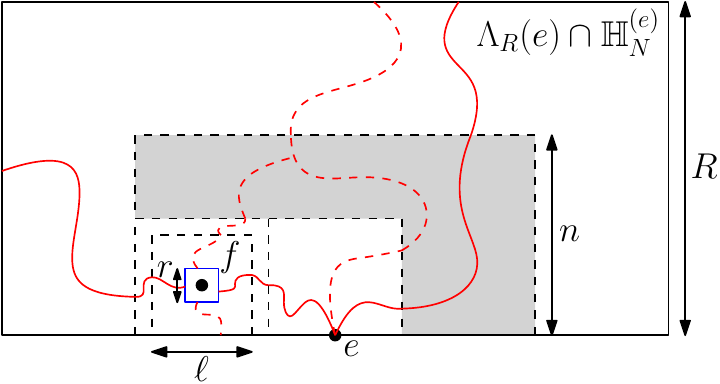}
    \caption{An illustration of the four cases in the proof. Once again, any length denoted by $r,n,\ell$ is to be understood as being of order $r,n,\ell$ (see figure of Case $2$ where two different lengths are denoted $n$). The greyed out parts in Cases $2$ and $4$ and the darker path in Case $3$ are not taken into account in the upper bound on the probability of being pivotal.}
    \label{fig:pivotal_arms}
\end{figure}
 Putting everything together and using quasi-multiplicativity, one gets
 \begin{align*}
 	 \Cov_{p_c}(\calA_{3^+}^{(e)}(R),\omega_f)&\lesssim  \sum_{r=1}^{\ell}\frac{\Delta(r)}r\pi_4(r,\ell)\pi_{3^+}(\ell,n)\pi_{3^+}(n)\pi_{3^+}(n,R)\\
 	 & +\sum_{r=\ell}^{n}\frac{\Delta(r)}r\pi_{3^+}(r,n)\pi_{3^+}(n)\pi_{3^+}(n,R)+\sum_{r=n}^R\frac{\Delta(r)}r \pi_{3^+}(R).
 \end{align*}
This allows to conclude that 
\begin{align*}
        \frac{1}{\pi_{3^+}(R)}\Cov_{p_c}(\calA_{3^+}^{(e)}(R),\omega_f)&\lesssim  \sum_{r=1}^{\ell}\frac{\Delta(r)}{r}\pi_4(r,\ell)\pi_{3^+}(\ell,n)+\sum_{r=\ell}^{n}\frac{\Delta(r)}{r}\pi_{3^+}(r,n)+\sum_{r=n}^{R}\frac{\Delta(r)}r \\
        &\lesssim \Delta(\ell)+\Delta(\ell)+\Delta(\ell),
    \end{align*}
where passing to the last line one uses the existence of $c>0$ (see Remark \ref{rem:arm-estimates-transation-invariant-environment}) such that
    \begin{itemize}
        \item $\pi_{3^+}(d,\ell)\lesssim \Delta(d,\ell)(d/\ell)^c$ and $\pi_4(r,d)\lesssim \Delta(r,d)(r/d)^c$ in the first sum.
	\item $\pi_{3^+}(r,\ell)\asymp (r/\ell)^2\lesssim \Delta(r,\ell)(r/\ell)^c$ in the second sum.
        \item $\Delta(r)\lesssim \Delta(\ell)(\ell/r)^c$ for the last sum.
    \end{itemize}
 This concludes the proof.
\end{proof}

\begin{proof}[Proof of Lemma \ref{lem:geometric-estimates_perco} in the case of crossing events]

In the case of percolation, one may rewrite
\begin{equation*}
    \Cov_{p_c}(\scrC_{\bullet}(\calR),\omega_e)\asymp\phi_{\bbT_N,p_c,1}[\textrm{Piv}_e(\scrC_{\bullet}(\calR))]
\end{equation*}
where $\textrm{Piv}_e(\scrC_{\bullet}(\calR))$ is the event that $e$ is pivotal for the event $\scrC_{\bullet}(\calR)$. Note that if $e$ does not belong to $\calR$, this probability is simply $0$. Assume now that $e$ belongs to $\calR$, and let $\ell\le n$ be respectively the distance from $e$ to $\partial \calR$ and from $e$ to a corner of $\calR$. The event that $e$ is pivotal for $\scrC_{\bullet}(\calR)$ implies the existence of a four-arm event around $e$ from scales $1$ to $\ell$, then a half-plane three arm event from scales $\ell$ to $n$, and finally a quarter-plane two arm event from scales $n$ to $1$. This is exactly what is depicted in Figure \ref{fig:pivotal_crossing} when setting $r=1$. By point (iv) of Remark \ref{rem:arm-estimates-transation-invariant-environment} (lower bound) there exists a constant $c>0$ such that
\begin{align*}
    \phi_{\bbT_N,p_c,1}[\textrm{Piv}_e(\scrC_{\bullet}(\calR))]&\lesssim \pi_{4}(\ell)\pi_{3^+}(\ell,n) \lesssim \pi_{4}(\ell)(\ell/n)^2\lesssim \pi_{4,p_c}(n),
\end{align*}
where we once again choose to not write the subscripts $p_c$. Summing over all edges $e\in \calR$ gives, using point (iv) of Remark \ref{rem:arm-estimates-transation-invariant-environment} (lower bound)
\begin{align*}
    \sum_{e\in \calR}\Cov_{p_c}(\scrC_{\bullet}(\calR),\omega_e)&\lesssim \sum_{n}n\cdot\pi_{4}(n)\asymp \textrm{W}(N)^{-1}.
\end{align*}
Finally, using the upper bound of the fourth item of Remark \ref{rem:arm-estimates-transation-invariant-environment}, one gets
\begin{align*}
    \sum_{e\in \calR}\Cov_{p_c}(\scrC_{\bullet}(\calR),\omega_e)^2&\lesssim \sum_{n}n\cdot\pi_{4}(n)^2\asymp 1.
\end{align*}
\end{proof}

\begin{proof}[Proof of Lemma \ref{lem:window-inequalities}]
    For this proof, we recall that the scaling relations ensures that one has $\textrm W(N)\asymp N^{-2}\Delta_{p_c}(N)^{-1}\asymp (\sum_{r\le N}r\Delta_{p_c}(r))^{-1}$. The inequality on the left without the polynomial improvement comes from the trivial bound $\sum_{r\le N}r\Delta_{p_c}(r)^4\lesssim \sum_{r\le N}r\Delta_{p_c}(r)$, which is of order $\textrm W(N)^{-1}$. In order to get the polynomial improvement, we recall the inequality $(r/R)^c\gtrsim\Delta_{p_c}(r,R)\asymp \frac{\Delta_{p_c}(R)}{\Delta_{p_c}(r)}\gtrsim (r/R)^{2-c}$ for some $c>0$, which allows us to write
\begin{align*}
        \sum_{r\le N}r\Delta_{p_c}(r)^4&\le \sum_{r\le \sqrt N}r\Delta_{p_c}(r) + \Delta_{p_c}\left(\sqrt N\right)^3\sum_{r\le N}r\Delta_{p_c}(r) \\
        &\lesssim \left(\sqrt N\right)^2\Delta_{p_c}\left(\sqrt N\right)+\left(\sqrt N\right)^{-3c}\textrm W(N)^{-1} \\
        &\lesssim \left(\sqrt N/N\right)^{-c}N^2\Delta_{p_c}(N)+\left(\sqrt N\right)^{-3c}\textrm W(N)^{-1} \\
        &\lesssim N^{-c/2}\textrm W(N)^{-1},
 \end{align*}
where in the second to last line, we used that $\textrm W(N)^{-1}\asymp N^2\Delta_{p_c}(N)$. Plugging this into the expression of $\widetilde{\textrm W}(N)$ gives the first inequality of the lemma with constant $c/6$. For the inequality on the right, simply rearranging the terms shows that it is equivalent to
    \begin{equation*}
        \sum_{r\le N}r\Delta_{p_c}(r)^2\lesssim \left(\sum_{r\le N}r\Delta_{p_c}(r)\right)^{2/3}\left(\sum_{r\le N}r\Delta_{p_c}(r)^4\right)^{1/3},
    \end{equation*}
which is just the Hölder inequality applied to $r\Delta_{p_c}(r)^2=(r\Delta_{p_c}(r))^{2/3}(r\Delta_{p_c}(r)^4)^{1/3}$.
\end{proof}

\section{Estimates on Skorokhod embedding stopping times}
\label{app:tail-bounds}

\subsection{Tails for Skorokhod embedding stopping times}
\label{app:skorokhod_tails}

In this section, we prove Lemma \ref{lem:skorohod_tails}. We point out that when the random variable $X$ is bounded, there are simpler arguments that show that $T$ has exponential tails (which is the case $\gamma=\infty$).

\begin{proof}[Proof of Lemma \ref{lem:skorohod_tails}]
The goal will be to use the so-called \emph{Vallois embedding} \cite{skorokhod_vallois} in order to construct some stopping time $T$ adapted to a standard Brownian motion $(B_t)_{t\geq 0}$ such that $B_T \overset{(d)}{=}X$. We denote $\mathcal{F}_t:=\sigma\Big( B_s, 0\leq s \leq t \Big)$ the naturally associated filtration. Let us recall explicitly this construction here, which corresponds to \cite{skorokhod_azema-yor_and_vallois} in the special case where $h(s)=s$ and $F=0$. We mostly keep the notations of \cite{skorokhod_azema-yor_and_vallois} in order to simplify cross-reading between the articles. All the constants involved in the statements depend only on $c$ and $\gamma$. Let $\mu$ be the law of the random variable $X$ and set for $x\in \bbR $ its \emph{potential} denoted $U_{\mu}(x)$ (which is related to the function $c(x)$ in \cite{skorokhod_azema-yor_and_vallois}) and given by
\begin{equation*}
    U_{\mu}(x):=-\int_{\bbR}|x-y|d\mu(y)=-x-2\int_x^{\infty}(y-x)d\mu(y).
\end{equation*}
This allows to define the functions $\theta,\phi,R,S,\Gamma$ as:
\[
\left\{
\begin{array}{l}
\theta(s) := \textrm{argmin}_{x>0}\left\{\frac{-s-U_{\mu}(x)}{x}\right\}, \\
\phi(s): = \textrm{argmax}_{x<0}\left\{\frac{-s-U_{\mu}(x)}{x}\right\},
\end{array}
\right.
\quad
\left\{
\begin{array}{l}
R(s):=\frac{-s-U_{\mu}(\theta(s))}{\theta(s)}, \\
S(s):=\frac{-s-U_{\mu}(\phi(s))}{\phi(s)}, \\
\Gamma(s) := \frac{R(s)-S(s)}2.
\end{array}
\right.
\]
One can then parametrize the time using $ H(s) := \int_0^s \frac{dz}{\Gamma(z)}$ and define the reparametrized versions of $\theta$ and $\phi$ by setting:
\begin{equation*}
    a(\ell) := \theta(H^{-1}(\ell)),\quad b(\ell) := \phi(H^{-1}(\ell)).
\end{equation*}
Let $L_t$ be the local time of $B_t$ at $0$ \cite{le2016brownian}. The Vallois stopping time \cite{skorokhod_vallois} is given by \cite[Equation (2)]{skorokhod_azema-yor_and_vallois} to the filtration $\mathcal{F}_t$ and is defined by
\begin{equation}\label{eq:T_V}
    T_V := \inf\{t\ge 0|B_t \not\in (b(L_t),a(L_t))\}.
\end{equation}
The main goal of this section is to obtain tail estimates for the variable $T_V$. In order to simplify the reading, we assume that the law $\mu$ has a density with respect to the Lebesgue measure (denoted here by $f$), but the announced results hold in full generality. Given the formulation $\theta(s)$ as an $\textrm{argmax}$ on $\bbR_+^{*}$, it is straightforward to see that
\begin{align*}
	 \frac{\dd}{\dd x}\left(\frac{-s-U_{\mu}(x)}x\right)|_{x=\theta(s)}&= \frac{\dd}{\dd x}\left(\frac{-s+2\int_x^{\infty}(y-x)f(y)dy}x+1\right)|_{x=\theta(s)}\\
	 &=\frac{s-2\int_{\theta(s)}^{\infty}yf(y)dy}{\theta(s)^2}=0,
\end{align*}
which ensures that
\begin{equation*}
    s=2\int_{\theta(s)}^{\infty}yf(y)dy\lesssim \exp(-c \cdot\theta(s)^{\gamma}).
\end{equation*}
The above equation reads as $\theta(s)\lesssim \log(1/s)^{1/\gamma}$, and the same estimate holds by symmetry for $-\phi(s)$. This ensures that as long as $s$ remains far away from $0$ and $\infty$, both functions $\theta,-\phi$ remain also bounded away from $0$ and $\infty$. In order to finish the proof, one needs to estimate all three functions $R,S,\Gamma$ in the $s\to 0$ regime. One has
\begin{align*}
	R(s) &= \frac{-s-U_{\mu}(\theta(s))}{\theta(s)}=\frac{\theta(s)+\int_{\theta(s)}^{\infty}(2(y-\theta(s))-2y)f(y)dy}{\theta(s)}\\
&= 1-2\mu[\theta(s),\infty).
\end{align*}
Similarly one gets that  $S(s) = -1-2\mu(-\infty,\phi(s)]$ and therefore
\begin{equation*}
    \Gamma(s) = 1-\big(\mu[\theta(s),\infty)-\mu(-\infty,\phi(s)]\big).
\end{equation*}
It is clear from the definitions of $\theta(s)$ and $\phi(s)$ that when $s\to 0$, the functions $\theta(s),-\phi(s)$ diverge to $+\infty$, which ensure that $\Gamma$ is bounded from below and $H(s)$ grows at linear speed when $s\to 0$. This ensures that for $\ell>0$ small enough, the functions $a(\ell)$ and $b(\ell)$  satisfy
\begin{equation*}
    a(\ell),-b(\ell)\lesssim \log(1/\ell)^{1/\gamma}.
\end{equation*}
When $\ell\gtrsim 1$, we simply bound $a(\ell),b(\ell)$ by some constant $C>0$. Then $T_V$ is bounded by the stopping time $\inf\{t|B_t\not\in(-C,C)\}$ which has exponential tails, plus the stopping time $T^*$ defined by
\begin{equation*}
    T^* := \inf\big\{t\ge 0\big||B_t|\ge C\log(1/L_t)^{1/\gamma}\big\}
\end{equation*}
Fix some $\varepsilon>0$ (which will be adjusted later in the proof). Using the union bound, for any $t>0$, one can write
\begin{align*}
	\bbP \Big[T^*>t\Big]&\leq \bbP\Big[T^*>t,L_{t/2}>\eps\Big]+\bbP\Big[T^*>t,L_{t/2}\le \eps\Big]\\
	&  \le \bbP\left[\sup_{s\le t/2}|B_s|< 2\log(1/\eps)^{1/\gamma}\right]+\bbP\Big[L_{t/2}\le \eps\Big]\\
	& \lesssim \exp\left(\frac{-c}{\log(1/\eps)^{2/\gamma}}t\right)+\frac{\eps}{\sqrt t}.
\end{align*}
Set $\eps=\exp\left(-t^{\frac{\gamma}{\gamma+2}}\right)$, which ensures that for $t\gtrsim 1$ one has, for $c>0$ small enough,
\begin{equation*}
  \bbP\Big[T_V>t\Big] \lesssim \bbP\Big[T^*>t/2\Big]+\exp(-ct) \lesssim   \exp\left(-ct^{\frac{\gamma}{\gamma+2}}\right),
\end{equation*}
since $\frac{\gamma}{\gamma+2}\le 1$, which concludes the proof.    
\end{proof}

\subsection{Tails for Skorokhod embedding stopping times with drift}
\label{app:skorokhod_drift_tails}

In this section, we prove Lemma \ref{lem:skorokhod_drift_tails}. Once again, when $X$ is bounded, one may show that the result holds with $\gamma=\infty$.

\begin{proof}[Proof of Lemma \ref{lem:skorokhod_drift_tails}]
Set $Y_t=B_t-\mu t$. From the process $(Y_t)_{t\geq 0}$, one can naturally construct a martingale $Z_t = e^{2\mu Y_t}-1$ such that $Z_0=0$ almost surely. The Dubins-Schwarz Theorem ensures the existence of a Brownian motion $\widetilde{B}$ adapted to some modified filtration (potentially defined on some extended probability space), such that $Z_t = \widetilde{B}_{\langle Z,Z\rangle_t}$. Applying the Skorokhod embedding theorem recalled above to the centred random variable $e^{2\mu X}-1$, there exists some stopping time $T_V$ such that $\widetilde{B}_{T_V} \overset{(d)}{=} e^{2\mu X}-1$. One can then define the stopping time
\begin{equation*}
    T := \inf\{t\ge 0| \langle Z,Z\rangle_t = T_V\}.
\end{equation*}
Let us see why this $T$ satisfies the first condition of the lemma. One can couple $(Y,\widetilde B,T_V)$ to $X$ in such a way that $\widetilde B_{T_V}=e^{2\mu X}-1$. Then one may write, on the event that $\{T<\infty\}$,
\begin{equation*}
    e^{2\mu X}-1 = \widetilde B_{T_V} = \widetilde B_{\langle Z,Z\rangle_T}=Z_T=e^{2\mu Y_T}-1
\end{equation*}
Since $x\mapsto e^{2\mu X}-1$ is injective, one gets that $X=Y_T$ on the event $\{T<\infty\}$.

Given the definition of $T$, it is not mandatory that $\langle Z,Z\rangle_t\to \infty$ as $t\to\infty$, therefore the event $\{T=\infty \} $ may happen with positive probability. Still, in the regime $\sigma\ll 1$ which is relevant in our setup, this will happen with a small probability depending on $\sigma$. Fix some parameter $\eps\ll 1$, whose value will be fixed later in the reasoning. In order to estimate the tail of the stopping time $T$, one can make the dichotomy
\begin{equation}\label{eq:bound-stopping-time-drifted-process}
    \bbP\Big[T>t\Big]\le \bbP\Big[T_V>\mu t\Big] + \bbP\Big[T>T_V/\mu, T_V\le \mu t\Big].
\end{equation}
Combining the fact that $|X|\le 1$ and that $\mu$ is bounded away from $0$ and $\infty $ ensures that $e^{2\mu X}-1$ also has the same stretched exponential tail as $X$. Therefore
\begin{equation*}
    \bbP\Big[T_V>\mu t\Big]\lesssim \exp\Big(-c' \cdot \Big(\frac{t}{\sigma^2}\Big)^{\frac{\gamma}{\gamma+2}}\Big)
\end{equation*}
for some constant $c'$ only depending on $c$. In order to derive tails bounds for the the stopping time $T$, it remains to evaluate the second term in the RSH of \eqref{eq:bound-stopping-time-drifted-process}. One has
\begin{equation*}
    \bbP\Big[T>T_V/\mu, T_V\le \mu t\Big]=\bbP\Big[\langle Z,Z\rangle_{T_V/\mu}<T_V,T_V\le \mu t\Big]
\end{equation*}
Computing explicitly the bracket of $Z$, one gets
\begin{equation*}
    \langle Z,Z\rangle_{T_V/\mu}=   2\mu\int_0^{T_V/\mu}e^{2\mu Y_t}\dd t\ge 2T_V\cdot  \exp\Big(2\mu \Big(\inf_{0\leq s\le T_V/\mu}Y_s\Big)\Big).
\end{equation*}
Let $\mu_{T_V}$ be the law of the stopping time $T_V$. The previous equation 
implies that
\begin{align*}
	\bbP\Big[\langle Z,Z\rangle_{T_V/\mu}<T_V,T_V \le \mu t\Big] & \leq \bbP\left[\inf_{0\leq s\le T_V/\mu}Y_s<-\frac{\log 2}{2\mu}, T_V\le \mu t\right]\\
	& =   \int_0^{\mu t} \bbP\left[\inf_{0\leq s\le u/\mu}Y_s<-\frac{\log2}{2\mu}\right]d\mu_{T_V}(u)\\
&\le \int_0^{\mu t} \bbP\left[\inf_{0\leq s\le u/\mu}B_s<u-\frac{\log2}{2\mu}\right]d\mu_{T_V}(u).
\end{align*}
Assuming that $t$ is chosen smaller than some constant depending on $\mu$, and that $\mu t$ is not an atom of $\mu_{T_V}$, there exist $c_1,c_2,c_3>0$, only depending on $c,\gamma,\mu $ such that
\begin{align*}
	\int_0^{\mu t} \bbP\left[\inf_{0\leq s\le u/\mu}B_s<u-\frac{\log2}{2\mu}\right]d\mu_{T_V}(u)&\lesssim  \int_0^{\mu t}\exp\Big(-\frac{c_1}{u}\Big)d\mu_{T_V}(u)\\
	& \lesssim \Bigg(\int_0^{\mu t}\exp\Big(-\frac{c_2}{u}-c_2\cdot\Big(\frac{u}{\sigma^2}\Big)^{\frac{\gamma}{\gamma+2}}\Big)du\Bigg)\\
    &\qquad +\exp\Big( -\frac{c_2}{t}-c_2\cdot\Big(\frac{t}{\sigma^2}\Big)^{\frac{\gamma}{\gamma+2}}\Big)\\
	& \lesssim \exp\Big(-c_3\cdot\sigma^{-\frac{\gamma}{\gamma+1}} \Big)+\exp\Big(-c_3 \cdot \Big(\frac{t}{\sigma^2}\Big)^{\frac{\gamma}{\gamma+2}}\Big),
\end{align*}
where the passage from the first to the second line is made by integration by parts and the passage to the last line is simply a weighted arithmetic mean-geometric mean inequality. Combining everything, there exists a constant $c_4>0$ such that for $t$ smaller than some constant,
\begin{equation*}
    \bbP\Big[T>t\Big]\lesssim \exp\Big(-c_4 \cdot \sigma^{-\frac{\gamma}{\gamma+1}}\Big)+\exp\left(-c_4 \left(\frac{t}{\sigma^2}\right)^{\frac{\gamma}{\gamma+2}}\right).
\end{equation*}
Note that this inequality extends to all $t$ because of the first term which is constant in $t$ and dominates as long as $t\gg \sigma^{\gamma/(\gamma+1)}$. This first term encapsulates the fact that $T$ may be very large (or even infinite) because $\langle Z,Z\rangle_t$ does not tend to $\infty$ almost surely, and so this first term is notably a bound for $\bbP[T=\infty]$. Conversely, when $t \ll \sigma^{\frac{\gamma}{\gamma+1}}$ (which is much larger than $\sigma$), it is the second term which dominates.
\end{proof}

\subsection{Large deviation principles for sums of random variables with stretched exponential tails}\label{app:large-deviation-sum}

Finally, in this section, we follow a strategy similar to \cite{aurzada2020large} to prove Lemma \ref{lem:tail_sum_stretched_exp}.

\begin{proof}[Proof of Lemma \ref{lem:tail_sum_stretched_exp}]
First start by separating the probability
\begin{equation*}
    \bbP[X>t]\le \bbP\left[\sup_i (a_iT_i)>t\right]+\bbP\left[X> t,\ \sup_i (a_iT_i)\le T\right]
\end{equation*}
The first term of the RHS of the above equation can be estimated by a simple union bound, using the fact that $t\gtrsim S$:
\begin{align*}
    \bbP\left[\sup_i (a_iT_i)>t\right]&\lesssim \sum_i\exp\left(-c\left(\frac t{a_i}\right)^{\gamma}\right) \lesssim \sum_i\frac{a_i}S\exp\left(-\frac c2\left(\frac t{a_i}\right)^{\gamma}\right) \\
    &\le \exp\left(-\frac c2\left(\frac t{m}\right)^{\gamma}\right).
\end{align*}
We pass to the second term. Here, we introduce a parameter $\lambda$, which we set equal to $\lambda := \frac{ct^{\gamma-1}}{2m^{\gamma}}$. Then we may write
\begin{align*}
    \bbP\left[X> t,\ \sup_i (a_iT_i)\le T\right]&\le e^{-\lambda t}\prod_i\bbE\big[e^{\lambda a_i T_i}\ind(a_i T_i\le t)\big] \\
    &\le \exp\left(-\lambda t + \sum_i\bbE\big[(e^{\lambda a_iT_i}-1)\ind(a_iT_i\le t)\big]\right).
\end{align*}
We now estimate the expectations inside the exponential. Write $\bbP_{T_i}$ for the law of $T_i$, then using integration by parts, one gets (as long as $t/a_i$ is not an atom of $T_i$, while there are only countably many $t$ where this can occur),
\begin{align*}
    \bbE\big[(e^{\lambda a_iT_i}-1)\ind(a_iT_i\le t)\big]&=\lambda a_i\int_0^{t/a_i}e^{\lambda a_i s}\bbP_{T_i}([s,t/a_i])ds\lesssim \lambda a_i\int_0^{t/a_i}e^{\lambda a_i s-cs^{\gamma}}ds.
\end{align*}
For $s\le t/a_i$, one may use that $0<\gamma<1$ to write that
\begin{equation*}
    \lambda a_i s \le \frac{ct^{\gamma-1}}{2m^{\gamma}}a_i\left(\frac t{a_i}\right)^{1-\gamma}s^{\gamma}\le \frac c2s^{\gamma}.
\end{equation*}
Hence one may bound the integral $\int_0^{t/a_i}e^{\lambda a_i s-cs^{\gamma}}ds$ by $\int_0^{\infty}e^{-cs^{\gamma}/2}ds$ which is a finite constant depending only on $c$. From this, one gets that $\bbE[(e^{\lambda a_iT_i}-1)\ind(a_iT_i\le t)]\lesssim \lambda a_i$. Putting everything together and using the value of $\lambda$, we obtain
\begin{equation*}
    \bbP\left[X> t,\ \sup_i (a_iT_i)\le T\right]\le \exp\left(-\frac{c}2 \left(\frac tm\right)^{\gamma}+K\frac{t^{\gamma-1}}{m^{\gamma}}S\right),
\end{equation*}
where $K$ is a constant depending on $c$, which can be made explicit depending on $c$ and the constant of the $\lesssim$ in the previous equation. Now we set $C$ to be $C:=\frac{4K}{c}$ so that if $t\ge CS$, one may write $K\frac{t^{\gamma-1}}{m^{\gamma}}S\le \frac c4\left(\frac tm\right)^{\gamma}$ and therefore for $t\ge CS$, one gets
\begin{equation*}
    \bbP[X>t]\lesssim \exp\left(-\frac c4\left(\frac tm\right)^{\gamma}\right).
\end{equation*}
This concludes the proof.
\end{proof}

\printbibliography

@incollection{pims_lecture_duminil,
  title     = {Lectures on the Ising and Potts models on the hypercubic lattice},
  author    = {Duminil-Copin, Hugo},
  booktitle = {PIMS-CRM Summer School in Probability},
  pages     = {35--161},
  year      = {2017},
  publisher = {Springer}
}

@article{liu2025mixing,
  title={Mixing rate exponent of planar Fortuin-Kasteleyn percolation},
  author={Liu, Haoyu and Wu, Baojun and Zhuang, Zijie},
  journal={arXiv preprint arXiv:2502.09950},
  year={2025}
}

@article{FK_two_arm_exponent,
  title={The two-arm exponent for critical FK-percolation},
  author={Avérous, Emile and Duminil-Copin, Hugo and He, Tiancheng and Manolescu, Ioan},
  journal={To appear},
  year={2025}
}

@inproceedings{FK_scaling_relations,
  title        = {Planar random-cluster model: scaling relations},
  author       = {Duminil-Copin, Hugo and Manolescu, Ioan},
  booktitle    = {Forum of Mathematics, Pi},
  volume       = {10},
  pages        = {e23},
  year         = {2022},
  organization = {Cambridge University Press}
}

@article{BPZ84a,
  title={Infinite conformal symmetry of critical fluctuations in two dimensions},
  author={Belavin, Alexander A and Polyakov, Alexander M and Zamolodchikov, Alexander B},
  journal={Journal of Statistical Physics},
  volume={34},
  number={5},
  pages={763--774},
  year={1984},
  publisher={Springer}
}

@article{fortuin1972random,
  title={On the random-cluster model II. The percolation model},
  author={Fortuin, Cornelis Marius},
  journal={Physica},
  volume={58},
  number={3},
  pages={393--418},
  year={1972},
  publisher={Elsevier}
}

@article{BPZ84b,
  title={Infinite conformal symmetry in two-dimensional quantum field theory},
  author={Belavin, Alexander A and Polyakov, Alexander M and Zamolodchikov, Alexander B},
  journal={Nuclear Physics B},
  volume={241},
  number={2},
  pages={333--380},
  year={1984},
  publisher={Elsevier}
}

@article{Smi07,
  title={Conformal invariance in random cluster models. I. Holmorphic fermions in the Ising model},
  author={Smirnov, Stanislav},
  journal={Annals of mathematics},
  pages={1435--1467},
  year={2010},
  publisher={JSTOR}
}

@article{Smi01,
  title={Critical percolation in the plane: conformal invariance, Cardy's formula, scaling limits},
  author={Smirnov, Stanislav},
  journal={Comptes Rendus de l'Acad{\'e}mie des Sciences-Series I-Mathematics},
  volume={333},
  number={3},
  pages={239--244},
  year={2001},
  publisher={Elsevier}
}

@inproceedings{duminil2021discontinuity,
  title={Discontinuity of the phase transition for the planar random-cluster and Potts models with q> 4},
  author={Duminil-Copin, Hugo and Gagnebin, Maxime and Harel, Matan and Manolescu, Ioan and Tassion, Vincent},
  booktitle={Annales Scientifiques de l'{\'E}cole Normale Sup{\'e}rieure},
  volume={54},
  number={6},
  pages={1363--1413},
  year={2021},
  organization={Soci{\'e}t{\'e} Math{\'e}matique de France}
}

@article{gloria2015quantification,
  title={Quantification of ergodicity in stochastic homogenization: optimal bounds via spectral gap on Glauber dynamics},
  author={Gloria, Antoine and Neukamm, Stefan and Otto, Felix},
  journal={Inventiones mathematicae},
  volume={199},
  number={2},
  pages={455--515},
  year={2015},
  publisher={Springer}
}

@article{egloffe2015random,
  title={Random walk in random environment, corrector equation and homogenized coefficients: from theory to numerics, back and forth},
  author={Egloffe, A-C and Gloria, Antoine and Mourrat, J-C and Nguyen, Thahn Nhan},
  journal={IMA journal of numerical analysis},
  volume={35},
  number={2},
  pages={499--545},
  year={2015},
  publisher={Oxford University Press}
}

@article{fisher1964correlation,
  title={Correlation functions and the critical region of simple fluids},
  author={Fisher, Michael E},
  journal={Journal of Mathematical Physics},
  volume={5},
  number={7},
  pages={944--962},
  year={1964},
  publisher={American Institute of Physics}
}

@article{essam1963pade,
  title={Pad{\'e} approximant studies of the lattice gas and Ising ferromagnet below the critical point},
  author={Essam, John W and Fisher, Michael E},
  journal={The Journal of Chemical Physics},
  volume={38},
  number={4},
  pages={802--812},
  year={1963},
  publisher={American Institute of Physics}
}

@article{WerSmi,
  title={Critical exponents for two-dimensional percolation},
  author={Smirnov, Stanislav and Werner, Wendelin},
  journal={arXiv preprint math/0109120},
  year={2001}
}

@article{camia2006sle,
  title={SLE (6) and CLE (6) from Critical Percolation},
  author={Camia, Federico and Newman, Charles M},
  journal={arXiv preprint math/0611116},
  year={2006}
}

@article{nolin2009asymmetry,
  title={Asymmetry of near-critical percolation interfaces},
  author={Nolin, Pierre and Werner, Wendelin},
  journal={Journal of the American Mathematical Society},
  volume={22},
  number={3},
  pages={797--819},
  year={2009}
}

@article{garban2012noise,
  title={Noise sensitivity of Boolean functions and percolation},
  author={Garban, Christophe and Steif, Jeffrey E},
  year={2012},
  publisher={American Mathematical Society}
}

@incollection{van2020four,
  title={On the four-arm exponent for 2D percolation at criticality},
  author={van den Berg, Jacob and Nolin, Pierre},
  booktitle={In and Out of Equilibrium 3: Celebrating Vladas Sidoravicius},
  pages={125--145},
  year={2020},
  publisher={Springer}
}

@article{duminilmanolesculi,
author = {Hugo Duminil-Copin and Jhih-Huang Li and Ioan Manolescu},
title = {Universality for the random-cluster model on isoradial graphs},
volume = {23},
journal = {Electronic Journal of Probability},
number = {none},
publisher = {Institute of Mathematical Statistics and Bernoulli Society},
pages = {1 -- 70},
year = {2018},
doi = {10.1214/18-EJP223},
URL = {https://doi.org/10.1214/18-EJP223}
}

@article{duminil2021planar,
  title={Planar random-cluster model: fractal properties of the critical phase},
  author={Duminil-Copin, Hugo and Manolescu, Ioan and Tassion, Vincent},
  journal={Probability Theory and Related Fields},
  volume={181},
  number={1},
  pages={401--449},
  year={2021},
  publisher={Springer}
}

@article{duminil2016new,
  title={A new proof of the sharpness of the phase transition for Bernoulli percolation and the Ising model},
  author={Duminil-Copin, Hugo and Tassion, Vincent},
  journal={Communications in Mathematical Physics},
  volume={343},
  number={2},
  pages={725--745},
  year={2016},
  publisher={Springer}
}

@article{duminil2021macroscopic,
  title={Macroscopic loops in the loop O(n) model at Nienhuis' critical point.},
  author={Duminil-Copin, Hugo and Glazman, Alexander and Peled, Ron and Spinka, Yinon},
  journal={Journal of the European Mathematical Society (EMS Publishing)},
  volume={23},
  number={1},
  year={2021}
}

@article{glazman2025delocalisation,
  title={Delocalisation and Continuity in 2D: Loop O(2), Six-Vertex, and Random-Cluster Models},
  author={Glazman, Alexander and Lammers, Piet},
  journal={Communications in Mathematical Physics},
  volume={406},
  number={5},
  pages={1--59},
  year={2025},
  publisher={Springer}
}

@article{duminil2019sharp,
  title={Sharp phase transition for the random-cluster and Potts models via decision trees},
  author={Duminil-Copin, Hugo and Raoufi, Aran and Tassion, Vincent},
  journal={Annals of Mathematics},
  volume={189},
  number={1},
  pages={75--99},
  year={2019},
  publisher={JSTOR}
}

@article{duminil2017continuity,
  title={Continuity of the phase transition for planar random-cluster and Potts models with 1≤ q≤ 4},
  author={Duminil-Copin, Hugo and Sidoravicius, Vladas and Tassion, Vincent},
  journal={Communications in Mathematical Physics},
  volume={349},
  number={1},
  pages={47--107},
  year={2017},
  publisher={Springer}
}

@article{duminil2020rotational,
  title={Rotational invariance in critical planar lattice models},
  author={Duminil-Copin, Hugo and Kozlowski, Karol Kajetan and Krachun, Dmitry and Manolescu, Ioan and Oulamara, Mendes},
  journal={arXiv preprint arXiv:2012.11672},
  year={2020}
}

@article{GFFConv,
  title={GFF Convergence of the 6V model},
  author={Duminil-Copin, Hugo and Kozlowski, Karol Kajetan and Lammers, Piet and Manolescu, Ioan},
  journal={To appear},
  year={2025}
}

@incollection{kesten1982russo,
  title={The Russo-Seymour-Welsh theorem},
  author={Kesten, Harry},
  booktitle={Percolation Theory for Mathematicians},
  pages={126--167},
  year={1982},
  publisher={Springer}
}

@article{kesten1987scaling,
  title={Scaling relations for 2D-percolation},
  author={Kesten, Harry},
  journal={Communications in Mathematical Physics},
  volume={109},
  number={1},
  pages={109--156},
  year={1987},
  publisher={Springer}
}

@article{nolin2008near,
  title={Near-critical percolation in two dimensions},
  author={Nolin, Pierre},
  year={2008}
}

@article{beffara2012self,
  title={The self-dual point of the two-dimensional random-cluster model is critical for q≥ 1},
  author={Beffara, Vincent and Duminil-Copin, Hugo},
  journal={Probability Theory and Related Fields},
  volume={153},
  number={3},
  pages={511--542},
  year={2012},
  publisher={Springer}
}

@article{beffara2012smirnov,
  title={Smirnov's fermionic observable away from criticality},
  author={Beffara, Vincent and Duminil-Copin, Hugo},
  journal={The Annals of Probability},
  pages={2667--2689},
  year={2012},
  publisher={JSTOR}
}

@article{duminil2014near,
  title={The near-critical planar FK-Ising model},
  author={Duminil-Copin, Hugo and Garban, Christophe and Pete, G{\'a}bor},
  journal={Communications in Mathematical Physics},
  volume={326},
  number={1},
  pages={1--35},
  year={2014},
  publisher={Springer}
}

@phdthesis{park2019ising,
  title={Ising Model and Field Theory: Lattice Local Fields and Massive Scaling Limit},
  author={Park, Sung Chul},
  year={2019},
  school={Ecole Polytechnique F{\'e}d{\'e}rale de Lausanne}
}

@article{park2022convergence,
  title={Convergence of fermionic observables in the massive planar FK-Ising model},
  author={Park, SC},
  journal={Communications in mathematical physics},
  volume={396},
  number={3},
  pages={1071--1133},
  year={2022},
  publisher={Springer}
}

@article{park2018massive,
  title={Massive scaling limit of the ising model: Subcritical analysis and isomonodromy},
  author={Park, SC},
  journal={arXiv preprint arXiv:1811.06636},
  year={2018}
}

@article{chelkak2023universality,
  title={Universality of spin correlations in the Ising model on isoradial graphs},
  author={Chelkak, Dmitry and Izyurov, Konstantin and Mahfouf, R{\'e}my},
  journal={The Annals of Probability},
  volume={51},
  number={3},
  pages={840--898},
  year={2023},
  publisher={Institute of mathematical statistics}
}

@phdthesis{wan2022statistical,
  title={Statistical mechanics models via the lens of potential theory},
  author={Wan, Yijun},
  year={2022},
  school={Universit{\'e} Paris-Saclay}
}

@article{chhita2012height,
  title={The height fluctuations of an off-critical dimer model on the square grid},
  author={Chhita, Sunil},
  journal={Journal of statistical physics},
  volume={148},
  number={1},
  pages={67--88},
  year={2012},
  publisher={Springer}
}

@article{rey2024doob,
  title={The Doob transform and the tree behind the forest, with application to near-critical dimers},
  author={Rey, Lucas},
  journal={Probability Theory and Related Fields},
  pages={1--71},
  year={2024},
  publisher={Springer}
}

@article{berestycki2022near,
  title={Near-critical dimers and massive SLE},
  author={Berestycki, Nathana{\"e}l and Haunschmid-Sibitz, Levi},
  journal={arXiv preprint arXiv:2203.15717},
  year={2022}
}

@article{cho1997criticality,
  title={Criticality in the two-dimensional random-bond Ising model},
  author={Cho, Sora and Fisher, Matthew PA},
  journal={Physical Review B},
  volume={55},
  number={2},
  pages={1025},
  year={1997},
  publisher={APS}
}

@article{shalaev1994critical,
  title={Critical behavior of the two-dimensional Ising model with random bonds},
  author={Shalaev, BN},
  journal={Physics reports},
  volume={237},
  number={3},
  pages={129--188},
  year={1994},
  publisher={Elsevier}
}

@article{merz2002two,
  title={Two-dimensional random-bond Ising model, free fermions, and the network model},
  author={Merz, F and Chalker, JT},
  journal={Physical Review B},
  volume={65},
  number={5},
  pages={054425},
  year={2002},
  publisher={APS}
}

@book{armstrong2019quantitative,
  title={Quantitative stochastic homogenization and large-scale regularity},
  author={Armstrong, Scott and Kuusi, Tuomo and Mourrat, Jean-Christophe},
  volume={352},
  year={2019},
  publisher={Springer}
}

@article{aizenman1990rounding,
  title={Rounding effects of quenched randomness on first-order phase transitions},
  author={Aizenman, Michael and Wehr, Jan},
  journal={Communications in mathematical physics},
  volume={130},
  number={3},
  pages={489--528},
  year={1990},
  publisher={Springer}
}

@article{aizenman1989rounding,
  title={Rounding of first-order phase transitions in systems with quenched disorder},
  author={Aizenman, Michael and Wehr, Jan},
  journal={Physical review letters},
  volume={62},
  number={21},
  pages={2503},
  year={1989},
  publisher={APS}
}

@article{Mah25,
  title={The near-critical Ising model via embedding deformation},
  author={Mahfouf, Rémy},
  journal={To appear},
  year={2025}
}

@article{benjamini1999noise,
  title={Noise sensitivity of Boolean functions and applications to percolation},
  author={Benjamini, Itai and Kalai, Gil and Schramm, Oded},
  journal={Publications Math{\'e}matiques de l'Institut des Hautes {\'E}tudes Scientifiques},
  volume={90},
  number={1},
  pages={5--43},
  year={1999},
  publisher={Springer}
}

@article{garban2013pivotal,
  title={Pivotal, cluster, and interface measures for critical planar percolation},
  author={Garban, Christophe and Pete, G{\'a}bor and Schramm, Oded},
  journal={Journal of the American Mathematical Society},
  volume={26},
  number={4},
  pages={939--1024},
  year={2013}
}

@article{tassion2023noise,
  title={Noise sensitivity of percolation via differential inequalities},
  author={Tassion, Vincent and Vanneuville, Hugo},
  journal={Proceedings of the London Mathematical Society},
  volume={126},
  number={4},
  pages={1063--1091},
  year={2023},
  publisher={Wiley Online Library}
}

@book{le2016brownian,
  title={Brownian motion, martingales, and stochastic calculus},
  author={Le Gall, Jean-Fran{\c{c}}ois},
  year={2016},
  publisher={Springer}
}

@misc{skorokhod_vallois,
  title={Le probleme de Skorokhod sur R: une approche avec le temps local, in ‘Seminar on probability, XVII’, Vol. 986 of Lecture Notes in Math},
  author={Vallois, Pierre},
  year={1983},
  publisher={Springer, Berlin}
}

@article{skorokhod_azema-yor_and_vallois,
  title={A unifying class of Skorokhod embeddings: connecting the Az{\'e}ma--Yor and Vallois embeddings},
  author={Cox, Alexander MG and Hobson, DG},
  journal={Bernoulli},
  volume={13},
  pages={114--130},
  year={2007}
}

@article{aurzada2020large,
  title={Large deviations for infinite weighted sums of stretched exponential random variables},
  author={Aurzada, Frank},
  journal={Journal of Mathematical Analysis and Applications},
  volume={485},
  number={2},
  pages={123814},
  year={2020},
  publisher={Elsevier}
}

\end{document}